\documentclass[10pt]{amsart}
\theoremstyle{plain}
\usepackage{amsmath, amsfonts, amsthm, amscd, mathrsfs}
\usepackage{graphics}
\usepackage{enumerate}
\usepackage{color}
\usepackage{epsfig}
\usepackage[all]{xy}
\usepackage{graphicx,amssymb}
\usepackage{tikz}

\graphicspath{ {Figures/} }
\numberwithin{figure}{section}
\usepackage{fullpage}

\newcommand{\Addresses}{{
  \bigskip
  \footnotesize

  James Farre, \textsc{Department of Mathematics, Yale University}\par\nopagebreak
  \textit{E-mail address}: \texttt{james.farre@yale.edu}
  }}

\theoremstyle{plain}
\usepackage{amsmath, amsfonts, amsthm, amscd, mathrsfs}
\usepackage{graphics}
\usepackage{enumerate}
\usepackage{color}
\usepackage{epsfig}
\usepackage[all]{xy}
\usepackage{graphicx,amssymb}

\usepackage{accents}
\newlength{\dhatheight}

\theoremstyle{theorem}

\newtheorem{theorem}{Theorem}[section]
\newtheorem{proposition}[theorem]{Proposition}
\newtheorem{observation}[theorem]{Observation}
\newtheorem{question}[theorem]{Question}
\newtheorem{conjecture}[theorem]{Conjecture}
\newtheorem{definition}[theorem]{Definition}

\newtheorem{fact}[theorem]{Fact}

\newtheorem{corollary}[theorem]{Corollary}

\newtheorem{lemma}[theorem]{Lemma}
\newtheorem{remark}[theorem]{Remark}

\newcommand{\para}[1]{\medskip\noindent\textbf{#1.}}

\newcommand{\G}{\PSL_2 \mathbb{C}}
\newcommand{\g}{\Gamma}

\renewcommand{\to}[1][]{\xrightarrow{\ #1\ }}

\newcommand{\bR}{\mathbb{R}}

\newcommand{\bC}{\mathbb{C}}

\newcommand{\bQ}{\mathbb{Q}}

\newcommand{\bZ}{\mathbb{Z}}
\newcommand{\bN}{\mathbb{N}}
\newcommand{\bH}{\mathbb{H}}
\newcommand{\cC}{\mathcal{C}}
\newcommand{\cE}{\mathcal{E}}

\newcommand{\cG}{\mathcal{G}}
\DeclareMathOperator{\tr}{tr}

\newcommand{\cT}{\mathcal{T}}
\newcommand{\cH}{\mathcal{H}}
\newcommand{\cN}{\mathcal N}
\newcommand{\cEL}{\mathcal{EL}}
\newcommand{\teich}{\mathscr{T}}

\newcommand{\inverse}{^{-1}}
\newcommand{\Hb}[1]{\textnormal{H}_{\textnormal{b}}^{#1}}
\newcommand{\Hbr}[1]{\overline{\textnormal{H}}_{\textnormal{b}}^{#1}}
\newcommand{\Hcb}[1]{\textnormal{H}_{\textnormal{cb}}^{#1}}
\newcommand{\Cb}[1]{\textnormal{C}_{\textnormal{b}}^{#1}}
\newcommand{\Cc}[1]{\textnormal{C}_{#1}}
\newcommand{\Ccc}[1]{\textnormal{C}^{#1}}
\newcommand{\HH}{\textnormal{H}}
\newcommand{\CP}{\mathbb {CP}}

\DeclareMathOperator{\axis}{axis}

\DeclareMathOperator{\Mod}{Mod}

\DeclareMathOperator{\PSL}{PSL}

\DeclareMathOperator{\cen}{center}
\DeclareMathOperator{\interior}{int}
\DeclareMathOperator{\Isom}{Isom}

\DeclareMathOperator{\smear}{smear}
\DeclareMathOperator{\Vol}{Vol}
\DeclareMathOperator{\vol}{vol}
\DeclareMathOperator{\Hom}{Hom}
\DeclareMathOperator{\area}{Area}
\DeclareMathOperator{\im}{im}

\DeclareMathOperator{\inj}{inj}
\DeclareMathOperator{\str}{str}

\DeclareMathOperator{\Out}{Out}

\DeclareMathOperator{\supp}{supp}

\title{Borel and volume classes for dense representations of discrete groups}
\author{James Farre}

\begin{document}

\begin{abstract}
We show that the bounded Borel class of any dense representation $\rho: G\to \PSL_n\bC$ is non-zero in degree three bounded cohomology and has maximal semi-norm, for any discrete group $G$.  When $n=2$, the Borel class is equal to the $3$-dimensional hyperbolic volume class.  Using tools from the theory of Kleinian groups, we show that the volume class of a dense representation $\rho: G\to \PSL_2\bC$ is uniformly separated in semi-norm from any other representation $\rho': G\to \G$ for which there is a subgroup $H\le G$ on which $\rho$ is still dense but $\rho'$ is discrete or indiscrete but stabilizes a point, line, or plane in $\bH^3\cup \partial \bH^3$.   We exhibit a family of dense representations of a non-abelian free group on two letters and a family of discontinuous dense representations of $\PSL_2\bR$, whose volume classes are linearly independent and satisfy some additional properties; the cardinality of these families is that of the continuum.  We explain how the strategy employed may be used to produce non-trivial volume classes in higher dimensions, contingent on the existence of a family of hyperbolic manifolds with certain topological and geometric properties.  
\end{abstract}
\maketitle
\vspace{-1cm}
\section{Introduction}
The bounded cohomology of discrete groups admits an almost entirely algebraic description, but many bounded classes are most naturally understood in the context of the geometry of non-positively curved metric spaces.  
Isometric actions of discrete groups on quasi-trees and Gromov hyperbolic metric spaces are responsible for producing an abundance of non-trivial bounded classes in degree two.  
Conversely, bounded cohomology has been used as a tool to understand the space of isometric actions that a discrete group admits on, say, a non-compact symmetric space $X$. Broadly, we aim to understand to what extent bounded cohomology parameterizes such actions, including those  that are not covering actions or that factor through some other group.  
We narrow our focus and consider the setting that $\Isom^+(X) = \PSL_n\bC$, $n\ge 2$.  
In this paper, we will explain the extent to which bounded cohomology sees the different isometric actions of a discrete group $G$ with a dense orbit in $X$.   
Our investigation relies heavily on the existence of certain \emph{discrete} free subgroups of $\Isom^+(\bH^3) = \PSL_2\bC$.  
Indeed, we will study the geometry of a  complete hyperbolic $3$-manifold homeomorphic to a genus $2$ handlebody in order to show that many non-conjugate dense representations of $G$ yield different volume classes in $\Hb3 (G;\bR)$.

Immersed locally geodesic tetrahedra in a complete hyperbolic $3$-manifold $M$ lift to embedded geodesic tetrahedra in the universal cover $\bH^3$.  
We can measure the (signed) hyperbolic volume of such a tetrahedron, which is bounded above by $v_3=1.01494...$.  This rule defines a bounded cocycle, hence a class in the degree $3$ bounded cohomology of the manifold.  

The bounded cohomology ring is an invariant of the fundamental group $\pi_1(M)$  \cite{Gromov:bdd}, and any (other) action $\rho:\pi_1(M)\to \Isom^+(\bH^3)$  yields a  bounded class $[\rho^*\vol_3]\in \Hb 3(\pi_1(M);\bR)$, which is obtained from a cocycle that measures the volumes of geodesic tetrahedra with vertices in the orbit $\rho(\pi_1(M)).x$, where $x\in \bH^3\cup\partial \bH^3$. 

We say that a representation is \emph{geometrically elementary} if its image stabilizes a line, a totally geodesic plane, or an (ideal) point in $\bH^{3}\cup\partial\bH^{3}$. 

\begin{theorem}[Theorem \ref{thm:v3_norm} and Theorem \ref{thm:main}]\label{thm:intro_main} If $G$ is a discrete group and  $\rho: G \to \G$ is a dense representation, i.e. $\overline {\rho (G)} = \G$, then \[\|[\rho^*\vol_3]\|_\infty =v_3.\]
In particular $[\rho^*\vol_3]\not=0\in \Hb 3(G;\bR)$.  Moreover, if $\rho_0: G\to \G$ is any other representation and there is a subgroup $H\le G$ such that $\overline{\rho(H)} = \G$, but $\rho_0$ is geometrically elementary or discrete restricted to $H$, then \[\|[\rho^*\vol_3]-[\rho_0^*\vol_3]\|_\infty\ge v_3, \] and this bound is sharp.
\end{theorem}
We consider in Section \ref{question}, for a dense representation $\rho: F_2\to \G$, the collection of \emph{$\rho$-dense subgroups} 
\[\mathcal {DS}(\rho) = \{H\le_{f.g.} F_2: \overline {\rho(H)} = \PSL_2\bC\}, \] where $H\le _{f.g.}F_2$ means that $H$ is a finitely generated subgroup of $F_2$.
If $\rho_1,\rho_2\in \Hom(F_2, \G)$ are conjugate, it is clear that $\mathcal{DS}(\rho_1)=\mathcal{DS}(\rho_2)$, while
if $\rho_1$ and $\rho_2$ are dense and $\mathcal{DS}(\rho_1)\not=\mathcal{DS}(\rho_2)$, then Theorem \ref{thm:intro_main} together with Lemma \ref{lem:geom_elem} implies that $\|[\rho_1^*\vol_3]-[\rho_2^*\vol_3]\|_\infty\ge v_3$.  
However, we have not been able to construct examples of non-conjugate dense representations $\rho_1$ and $\rho_2$ such that $\mathcal {DS}(\rho_1) = \mathcal {DS}(\rho_2)$.  
It is possible that $\mathcal {DS}(\rho)$ is a complete invariant of the $\G$ conjugacy class of a dense representation $\rho: F_2 \to \G$; by Theorem \ref{thm:intro_main},  $[\rho^*\vol_3]$ would then be a complete invariant of the $\G$ conjugacy class of a dense $\rho$; see Section \ref{question}.  

Suppose  $M$ is a hyperbolic $3$-manifold of finite volume and $\rho: \pi_1(M)\to \G$; the \emph{volume of $\rho$} is a numerical invariant that can be obtained by pairing the \emph{bounded fundamental class} or \emph{volume class} $[\rho^*\vol_3] \in \Hb 3(\pi_1(M);\bR)$ of the representation with a (relative) fundamental cycle of $M$.  
This numerical invariant has been studied from this perspective in \cite{BBI:mostow}, where Bucher, Burger, and Iozzi show that when the maximal volume for a representation is achieved, the representation must be conjugated by an isometry to the lattice embedding $\pi_1(M)\hookrightarrow \G$.   
We refer to this kind of result informally as a \emph{volume rigidity} result.   Theorem \ref{thm:intro_main} explains that, for a dense representation, while the volume of $\rho$ is non-maximal, the {\it volume class} of that representation has maximal semi-norm.

Dunfield \cite{Dunfield:volume_rigidity} proved a  volume rigidity theorem for closed hyperbolic 3-manifolds following an observation of Goldman \cite{Goldman:maximal} about  Gromov and Thurston's proof of Mostow's famous rigidity theorem.   Francaviglia \cite{Francaviglia:volume_rigidity} and Klaff \cite{Klaff:volume_rigidity} proved  a  volume rigidity theorem for finite volume hyperbolic manifolds, using the notion of \emph{pseudo-developing maps} defined in \cite{Dunfield:volume_rigidity}.   While Dunfield's definition of the volume of a representation via pseudo-developing maps {\it a priori} depends on choices, it turns out to be an invariant of the representation \cite[Theorem 1.1]{Francaviglia:volume_rigidity}.  The interested reader is directed toward \cite{BBI:mostow} for more history of volume rigidity results.

One advantage of using bounded cohomology to define the volume of a representation, is that no choices are made.  We sometimes see that the pullback of a continuous bounded cohomology class {\it itself} can characterize a representation, up to conjugation. For example, \cite[Theorem 3]{BIW:hermitian} states that the pullback of the bounded K\"ahler class is a complete invariant of continuous actions of a broad class of groups on an irreducible Hermitian symmetric space that is not of tube type.  This result fails miserably when the Hermitian symmetric space is of tube type.  For example, $\bH^2$ is such a space, and the bounded K\"ahler class in $\Hcb 2(\PSL_2\bR;\bR)$ is a multiple of the two dimensional hyperbolic volume class (also the Euler class).  For any two discrete and faithful representations $\rho, \rho':\pi_1(S)\to \PSL_2\bR$ of a closed surface group, we have $[\rho^*\vol_2] = \pm[{\rho'}^*\vol_2]$; see, e.g. \cite[Lemme 3.10]{Barge-Ghys:bddsurface}.  The space of $\PSL_2\bR$-conjugacy classes of discrete and faithful representations of a closed hyperbolic surface group is a union of two high dimensional cells that can each be identified with the Teichm\"uller space of that surface. 

We are most interested in hyperbolic manifolds with \textit{infinite volume}, and we work directly with volume classes in bounded cohomology.  For example, let $S$ be a finite type oriented surface with negative Euler characteristic.  We say that two representations $\rho_1, \rho_2: \pi_1(S)\to \G$ are \emph{quasi-isometric} if there is a $(\rho_1, \rho_2)$-equivariant quasi-isometry $\bH^{3}\to \bH^{3}$.  The following quasi-isometric volume rigidity theorem sees a combination of both of the  phenomena in the previous paragraph.  

\begin{theorem}[Theorem \ref{thm:rigidity}]\label{thm:intro_old_work} 
There exists a constant $\epsilon = \epsilon(S)$ such that the following holds.  Suppose that $\rho_0: \pi_1(S)\to \G$ is a discrete and faithful representation without parabolic elements, and that $[\rho_{0}^{*}\vol_{3}]\not=0$.  If $\rho: \pi_1(S) \to \G$ is any other representation without parabolics satisfying \[\|[\rho_0^*\vol_3]- [\rho^*\vol_3]\|_\infty<\epsilon,\]
then $\rho$ is discrete and faithful, and $\rho$ is quasi-isometric to $\rho_0$.  
If $\rho_0$ is totally degenerate, then $\rho_0$ and $\rho$ are conjugate in $\G$.  
\end{theorem}

Theorem \ref{thm:intro_old_work} is proved by combining Theorem \ref{thm:intro_main} with previous work of the author, the classification of finitely generated marked Kleinian groups,  and a theorem of Soma (Theorem \ref{soma's theorem}).  We provide a detailed proof of Theorem \ref{thm:intro_old_work} in Section \ref{sec:qi_classification} after establishing definitions and some background on the geometry and topology of  tame hyperbolic $3$-manifolds of infinite volume. We also elaborate on the role of parabolic cusps and the geometric meaning of the quasi-isometric equivalence relation, therein.  We note that quasi-isometric equivalence of discrete and faithful representations of finitely generated torsion free Kleinian groups is equivalent to the existence of a \emph{volume preserving} bi-Lipschitz homeomorphism of quotient manifolds $\bH^3/\im\rho \to\bH^3/\im \rho_0$ inducing $\rho_0\circ\rho\inverse$ on fundamental groups. 
 
Recently, Bucher, Burger, and Iozzi proved a volume rigidity result for representations of finite volume hyperbolic $3$-manifold groups into $\PSL_n\bC$ with respect to the so-called \emph{Borel invariant} of a representation, defined in \cite{BBI:borel}.  To do so, they computed the continuous bounded cohomology $\Hcb3 (\PSL_n\bC;\bR)$, which is generated by a single class $\beta_n$, called the \emph{bounded Borel class}.  They also computed the semi-norm of the bounded Borel class and how it behaves under various natural inclusions $\PSL_k\bC \hookrightarrow \PSL_n\bC$, $k\le n$ (see Section \ref{borel}).  The bounded Borel class generalizes hyperbolic volume in the sense that $\beta_2 = [\vol_3]$.  
Due to the work of \cite{BBI:borel}, most of the  argument that proves Theorem \ref{thm:intro_main} generalizes to dense representations into $\PSL_n\bC$.
\begin{corollary}[Theorem \ref{Borel:main} and Corollary \ref{cor:Borel_main}]\label{cor:intro_borel}
Let $G$ be a discrete group and $\rho: G \to \PSL_n\bC$ be dense.  Then \[\|\rho^*\beta_n \|_\infty = v_3 \frac{n(n^2-1)}{6}.\]
Suppose that  $\rho_0 : G \to \PSL_2\bC$ is such that there exists a subgroup $H\le G$ such that $\overline{\rho(H)} = \PSL_n\bC$ and $\rho_0$ is  geometrically elementary or  discrete and faithful restricted to $H$.  Then \[\|\rho^*\beta_n - (\iota_k\circ \rho_0)^*\beta_k \|_\infty \ge v_3 \frac{n(n^2-1)}{6},\] for all $ k\ge 2$, where $\iota_k: \G\to \PSL_n\bC$ is the (unique up to conjugation) irreducible representation. 
\end{corollary}
\begin{remark}
Compare the hypotheses of the second statements in Theorem \ref{thm:intro_main} to those in the second statement of Corollary \ref{cor:intro_borel}.  In Corollary \ref{cor:intro_borel} we assume that $\rho_0$ is {\it faithful} in addition to being discrete.  This is because it is easy to construct somewhat explicit dense representations of a free group $F_2$ on two letters to $\G$.  From this, we obtain the additional control needed to remove the extra hypothesis and prove Theorem \ref{thm:intro_main}; see Proposition \ref{prop:dense_not}, Case 2.
\end{remark}

 We will  consider the \emph{reduced bounded cohomology} $\Hbr 3(G;\bR) = \Hb 3(G;\bR)/\mathcal Z$, where $\mathcal Z$ is the subspace of zero-norm bounded cohomology.  The reduced space $\Hbr 3(G;\bR)$ is a Banach space with respect to the quotient norm. 
 Our techniques apply to families of dense representations satisfying some technical condition on subgroups, as in Theorem \ref{thm:intro_main}.
 \begin{theorem}[Theorem \ref{ell 1}]\label{thm:norm}Suppose $\{\rho_i:G\to \G\}_{i = 1}^N$ are dense representations.  If there are subgroups $H_i\le G$ such that $\overline{\rho_i(H_i)} = \PSL_2\bC$   but $\rho_i|_{H_j}$ is geometrically elementary or discrete for $i\not=j$, then for any $a_1, ..., a_N\in \bR$, we have
 \[\left\|\sum_{i = 1}^N a_i[\rho_i^*\vol_3]\right\|_\infty \ge\max\{|a_i|\}\cdot v_3.\]
 Consequently, $\{[\rho_i^*\vol_3]: i = 1,2, ..., N \}\subset \Hbr 3(G;\bR)$ is a linearly independent set.
\end{theorem}

In Section \ref{constructions}, we construct families of representations that satisfy the hypotheses of Theorem \ref{thm:norm}.  More specifically, for a non-abelian free group $F_2$ on two letters and every $\theta\in (0,1)$, we construct a representation $\rho_\theta: F_2 \to \G$; any finite rationally independent set $\Lambda'\subset (0,1)$  yields a family $\{\rho_\theta\}_{\theta\in \Lambda'}$ that satisfies the hypotheses of Theorem \ref{thm:norm}.  In the first line of the following corollary, we invoke the axiom of choice.
\begin{corollary}[Theorem \ref{uncountable family}]\label{dim cor}
Let $\Lambda\subset (0,1)$ be such that $\Lambda\cup\{1\}$ is a basis for $\bR$ as a $\bQ$-vector space, and let $\{\rho_\theta\}_{\theta\in\Lambda}$ be the dense representations constructed in Section \ref{constructions}.
The map \[
\begin{aligned}
\Lambda &\to \Hbr3(F_2)\\
\theta&\mapsto  [\rho_\theta^*\vol_3]
\end{aligned}\]  is injective with discrete image, and $\{[\rho_\theta^*\vol_3]: \theta\in \Lambda\}$ is a linearly independent set.  In particular, $\dim_\bR \langle [\rho_\theta^*\vol_3]: \theta\in \Lambda\rangle_\bR = \#\bR$.
\end{corollary}

As a further curiosity, we show that the spaces $ \Hbr 3(\PSL_2\mathbb \bR ;\bR)$ and $ \Hbr 3(\PSL_2\mathbb \bC ;\bR)$ are quite large when we endow $\PSL_2\bR$ and $\G$ with the {\it discrete} topology.  We construct `wild' field maps $\sigma: \bC \to \bC$ that induce homomorphisms $\rho_\sigma:\PSL_2\bR\to \PSL_2\bC$, which are continuous only if $\PSL_2\bR$ is endowed with the discrete topology.  If $\sigma$ is not the identity or complex conjugation, then $\sigma(\bR)$ is dense in $\bC$, hence $\rho_\sigma(\PSL_2\bR)$ is dense in $\G$.    Indeed, we construct many such $\sigma$ by extending bijections between transcendence bases of $\bC$ over $\bQ$.  
We carefully construct a family of dense representations $\{\rho_t':\PSL_2\bC\to \G\}_{t\in \bR}$ that restrict to dense representations $\{\rho_t: \PSL_2\bR\to \G\}_{t\in \bR}$.  For many free subgroups $F_2\le \PSL_2\bR$,  $\rho_t(F_2)$ is Schottky, while other free subgroups are mapped densely by $\rho_t$.  

\begin{corollary}[Theorem \ref{G dense}]\label{cor:intro_G}
There are dense representations $\{\rho_t': \PSL_2\bC \to \G\}_{t\in \bR}$ that restrict to dense representations $\{ \rho_t: \PSL_2\bR \to \G\}_{t\in \bR}$ such that $\{[\rho_t^*\vol_3]: t\in \bR\}$ is a linearly independent set in $\Hbr 3(\PSL_2\bR;\bR)$ and $\{[ {\rho_t'}^*\vol_3]: t\in \bR\}$ is a linearly independent set in $\Hbr 3(\PSL_2\bC;\bR)$.
\end{corollary}
There seems to be quite a bit of flexibility in our construction of wild field maps $\sigma: \bC\to \bC$; the dimension of the  vector space of bounded $3$-cochains on $\G$ or $\PSL_2\bR$ is $2^{\#\bR}$, which is an upper bound on the real dimension of degree $3$ reduced bounded cohomology.  It certainly seems possible that the dimension of reduced bounded cohomology for these groups is $2^{\#\bR}$; see Question \ref{q:G_dense}.

  The main line of argument used to prove our theorems is as follows: a densely embedded group $G\le \G$ can approximate the geometry of \emph{any} finitely generated Kleinian group, up to a certain scale.  More precisely, given a finitely generated Kleinian group $\g = \langle \gamma_1, ..., \gamma_k\rangle\le \G$, a dense representation $\rho: G\to \PSL_2\bC$, an $\epsilon>0$, and a positive integer $N$, there are $g_1, ..., g_k\in G$ such that all words of length at most $N$ in the $\rho(g_i)$ are in the $\epsilon$-neighborhood of corresponding words in the $\gamma_i$.  Suppose $\bH^3/\g$ has a submanifold $K\subset M$ with large volume and small surface area.   Assume that $K$ is equipped with a straight triangulation by geodesic tetrahedra, and that the triangulation does not have too many triangles on the boundary.   Then we can homotope the triangulation so that there is only one vertex and such that we do not lose too much volume during the homotopy, which is a small miracle of hyperbolic geometry.  The edges of the tetrahedra are now labeled by elements of $\pi_1(M)$, because they are closed, based loops.  This finite triangulation lifts to the universal cover.  The idea is now to use our approximation of words of length at most $N$ in the $\gamma_i$ by words of length at most $N$ in the $\rho(g_i)$ to build a chain on $G$ that has almost the same shape as our lifted chain, via its $\rho$-action.  In this way, we use the geometry of discrete groups to build chains on our abstract group $G$ that have large volume and small boundary area, which is enough to show that $[\rho^*\vol_3]\not=0$; if in addition, the chains on $\g$ are $\epsilon$-\emph{efficient}, we can show that $\|[\rho^*\vol_3]\|\ge \epsilon$.  Adapting a construction of Soma \cite[Lemma 3.2]{Soma:boundedsurfaces}, we are able to construct $(v_3-\epsilon)$-efficient chains, for all $\epsilon>0$; see also Section \ref{sec:efficient_chains} and, in particular, Lemma \ref{manifold chains} where we recreate Soma's construction.  
  
  The preceding paragraph is made precise in Section \ref{approximation section} where we record a key observation of this paper, Proposition \ref{approximation of chains}.  
  Our main line of argument is very generally applicable in the sense that it can be adapted to work in all dimensions, contingent on the existence of a sequence of hyperbolic $n$-manifolds with prescribed topological properties. 
Proposition \ref{prop:higher_dim} formalizes this idea to give one of many sufficient conditions by which one can apply the techniques in this paper in higher dimensions.  First, we need some terminology.  The following definition is made in order to highlight the necessary algebraic and topological ingredients of the proof of Theorem \ref{thm:intro_main}. 

\begin{definition}[Definition \ref{free approximation}] 
Let $\Gamma$ be a discrete group, $\alpha\in \HH_n(\Gamma;\bR)$, and $K>0$.  We say that $\alpha$ is \emph{$K$-freely approximated} if there is an integer $m$, a homomorphism $\varphi: F_m\to \Gamma$, and a chain $Z\in \textnormal C_n(F_m;\bR)$ such that $\varphi_*(Z)\in \alpha$ and $\|\partial Z\|_1\le K$.   
\end{definition}

The conclusion of the following proposition may seem surprising, at first.   
\begin{proposition}[Proposition \ref{prop:higher_dim}]\label{prop:higher_dim_intro}
Suppose $(M_i)$ is a sequence of oriented, closed hyperbolic $n$-manifolds with volume tending to infinity.  Let $[M_i]\in \HH_n(\pi_1(M_i);\bR)$ be the image of the fundamental class of $M_i$ under the natural isomorphism $\HH_n(M_i; \bR) \to \HH_n(\pi_1(M_i);\bR)$.  

If there is a $K$ such that $[M_i]$ is $K$-freely approximated for all $i$, then for any dense representation $\rho: F_2 \to \Isom^+(\bH^n)$,  \[[\rho^*\vol_n]\not=0 \in \Hb n(F_2; \bR).\]
\end{proposition}
\noindent See also Remark \ref{rem:dense_n=4} for a discussion of the (possible) dimension of $\Hb n (F_2;\bR)$, for even $n\ge 4$.
\begin{remark}
For  many sequences $(M_i)$ of closed hyperbolic surfaces and hyperbolic $3$-manifolds with volume tending to infinity, one can produce a bound $K$ depending on that sequence such that $[M_i]$ is $K$-freely approximated.  

However, it is not at all clear if there can be any sequence of hyperbolic $n$-manifolds satisfying the hypotheses of the proposition, for $n\ge 4$.   For example, if it were true that there is some acylindrically hyperbolic group $\Gamma$ and $\Hb n (\Gamma; \bR) = 0$, then $\Hb n(F_2;\bR) = 0$, as well \cite{FPS:quasi-cocycles}.  So if there is such an $n\ge 4$ and $\Gamma$, then no sequence of hyperbolic $n$-manifolds satisfying the topological constraints of the hypotheses of Proposition \ref{prop:higher_dim} can exist. 
\end{remark}

  To obtain the uniform separation of certain volume classes, as in Theorem \ref{thm:intro_main}, we appeal to the classification theory and structure theory of finitely generated Kleinian groups, and make extensive use of the Covering Theorem \ref{thm:covering} and the Tameness Theorem \ref{thm:tameness}, which are now standard tools for studying hyperbolic $3$-manifolds.  The  Ending Lamination Theorem \ref{thm:ELT} also plays an important role and informs our understanding of the geometry of hyperbolic $3$-manifolds of infinite volume, generally.   

The structure of the paper is as follows.  In Section \ref{sec:bdd_cohomology}, we provide some background on bounded cohomology, volume classes, and the bounded Borel class.  In Section \ref{approximation section}, we give a general strategy for approximation and record our key insight Proposition \ref{approximation of chains}. 

 In Section \ref{hyperbolic}, we describe the geometry of ends of hyperbolic $3$-manifolds and give some context for the three major structural results we need in the sequel.  In Section \ref{sec:qi_classification}, we assume Theorem \ref{thm:intro_main} and deduce Theorem \ref{thm:intro_old_work} from previous work of the author, collected in Theorem \ref{old work}.      

In Section \ref{volume classes}, we employ our resources coming from the classification theory of tame hyperbolic $3$-manifolds to prove Theorem \ref{thm:intro_main}. In Section \ref{sec:efficient_chains}, we recall Soma's construction of $(v_3-\epsilon)$-efficient chains, collect some technical facts, and prove the first statement of Theorem \ref{thm:intro_main}.  In Section \ref{sec:separation}, we take a rather technical foray into analyzing infinite index subgroups of discrete and dense representations of a free group $F_2$ on $2$ letters (Lemma \ref{lem:technical} and Proposition \ref{prop:dense_not}).  The rest of Theorem \ref{thm:intro_main} follows quickly, thereafter. 
Then we generalize Theorem \ref{thm:intro_main} to give Theorem \ref{thm:norm}.  

In Section \ref{Borel} we pull back a higher rank  formulation of the volume class, known as the bounded Borel class of a dense representation $\rho : \g \to \PSL_n\bC$; the argument we present for $n=2$ generalizes to the higher rank setting, after a technical detour, thanks to \cite{BBI:borel}, and we record the necessary modifications to the proof of Theorem \ref{thm:intro_main} to obtain Corollary \ref{cor:intro_borel}.  

Section \ref{sec:applications} concerns constructions of certain families of dense representations, `applications' of Theorem \ref{thm:norm} that produce large subspaces of bounded cohomology, and questions.  In Section \ref{constructions}, we construct an uncountable family of non-conjugate representations, every finite subset of which satisfies the hypotheses of Theorem \ref{thm:norm}.  This shows that the dimension of (reduced) bounded cohomology spanned by the volume classes for dense representations of a free group is large.  Later, we use wild field embeddings to produce many linearly independent volume classes in degree three reduced bounded cohomology of $\G$ and $\PSL_2\bR$, endowed with the discrete topology.
We pose some questions that presented themselves to us during this investigation.

Finally, in Section \ref{sec:higher_dim} we discuss the problem of understanding the (non-)triviality of volume classes of dense representations in dimensions $n\ge 4$.  We also show that closed hyperbolic surfaces are $2$-freely approximated; in fact our proof shows that a closed hyperbolic surface of genus at least $2$ can be $\epsilon$-freely approximated for any $\epsilon>0$, using covering space theory; note the similarity of Lemma \ref{lem:surface_approx} with the standard computation of simplicial volume of closed surfaces of negative Euler characteristic.  The main result of Section \ref{sec:higher_dim}, Proposition \ref{prop:higher_dim}, is essentially independent of the more technical work done in Sections \ref{hyperbolic}-\ref{sec:applications}, and so can be read and understood directly after Section \ref{approximation section}.

\section*{Acknowledgments}
The author would like to thank the anonymous referees for thorough readings of this manuscript, insightful mathematical comments, and suggestions that improved the quality and correctness of this article.  We would also like to thank Maria Beatrice Pozzetti for many very enjoyable  conversations, her sustained interest in this work, and for suggesting that we consider the bounded Borel class.  We thank Kenneth Bromberg for his guidance and support and Michael Landry for producing the figure.  The author recognizes the support of the NSF, in particular, grants DMS-1246989 and DMS-1509171.

\section{Continuous bounded cohomology of groups and spaces}\label{sec:bdd_cohomology}
In this section we establish some preliminaries on bounded cohomology, the volume class, and the bounded Borel class.  In Section \ref{approximation section}, we introduce the main novel idea of the paper,  Proposition \ref{approximation of chains}.
 \subsection{Continuous bounded cohomology of groups}\label{bounded groups}
Let $\cG$ be a topological group.  Then $\cG$ acts on the space of continuous functions $\{\cG^k \to \bR\}$ as follows.  If $g\in \cG$ and $f: \cG^k\to \bR$ is continuous, then \[g.f(g_1,...,g_k): = f(g\inverse g_1, ..., g\inverse g_k).\]
We define a cochain complex for $\cG$ by considering the collection of continuous, $\cG$-invariant functions 
\[\textnormal C^n(\cG;\bR) = \lbrace f: \cG^{n+1} \to \bR: g.f = f, ~ \forall g \in \cG\rbrace.\]  The homogeneous co-boundary operator $\delta$ for the trivial $\cG$-action on $\bR$ is, for ${f\in \Ccc n(\cG;\bR)}$, \[\delta f (g_0, ... ,g_{n+1}) = \sum_{i=0}^{n+1} (-1)^i f(g_0, ..., \hat{g}_i,..., g_{n+1}),\] where $\hat g_i$ means to omit $g_i$, as usual.    
The co-boundary operator gives the collection $\Ccc \bullet(\cG;\bR)$ the structure of a cochain complex.  An $n$-cochain $f$ is \emph{bounded} if \[\|f\|_\infty = \sup|f(g_0, ..., g_{n})| < \infty,\] where the supremum is taken over all ($n+1$)-tuples $(g_0, ..., g_{n}) \in \cG^{n+1}$.  The subspace of continuous bounded $n$-cochains is denoted by $\Cb n(\cG;\bR)$.    

The operator  $\delta : \Cb n (\cG;\bR) \to \Cb {n+1}(\cG;\bR)$ is a bounded linear, hence continuous, operator between Banach spaces with operator norm at most $n+2$, so the collection of bounded cochains $\Cb \bullet(\cG;\bR)$ forms a subcomplex of the ordinary cochain complex. The cohomology of $(\Cb \bullet(\cG;\bR),\delta)$ is called the \emph{continuous bounded cohomology} of $\cG$, and we denote it by $\Hcb \bullet (\cG;\bR)$.  The $\infty$-norm $\|\cdot\|_\infty$ descends to a semi-norm on bounded cohomology, so that if $\alpha\in \Hcb \bullet (\cG;\bR)$, \[ \|\alpha\|_\infty = \inf_{a\in \alpha}\|a\|_\infty.\]  A continuous group homomorphism $\varphi: H\to \cG$ induces a map $\varphi^*:\Hcb \bullet (\cG;\bR)\to \Hcb \bullet (H;\bR)$ that is norm non-increasing.  

When $G$ is a discrete group, the continuity assumption on cochains is vacuous, and we write $\Hb \bullet (G;\bR) = \Hcb\bullet (G;\bR)$ to denote its bounded cohomology.  Soma has shown  \cite{soma:nonBanach} that the pseudo-norm is in general not a norm in degree $\ge 3$.  We will consider the quotient $\Hbr 3(G;\bR) =\Hb 3(G;\bR)/\mathcal Z$ where $\mathcal Z\subset \Hb 3(G;\bR)$ is the subspace of zero-semi-norm elements.  Then $\Hbr 3(G;\bR)$ is a Banach space with the quotient norm $\| \cdot \|_\infty$.   

\subsection{Norms on chain complexes}\label{bounded spaces}
Given a connected countable CW-complex $X$, we define a norm on the singular chain complex of $X$ as follows.  Let $\Sigma_n = \lbrace \sigma: \Delta_n\to X\rbrace$ be the collection of singular $n$-simplices.  Write a singular chain $A\in \Cc n\left( X;\bR\right)$ as an $\bR$-linear combination \[A = \sum \alpha_\sigma\sigma,\] where each ${\sigma\in \Sigma_n}$.  The $\ell_1$-norm of $A$ is defined as \[\left\| A \right\|_ 1 = \sum \left| \alpha_\sigma\right|.\]
This  norm promotes the algebraic chain complex $\Cc \bullet (X ;\bR)$ to a chain complex of normed linear spaces; the boundary operator is a bounded linear operator.  Keeping track of this additional structure, we can take the topological dual chain complex
\[\left(\Cc \bullet (X;\bR),\partial, \| \cdot  \|_1\right)^*=\left(\Cb \bullet (X;\bR),\delta, \| \cdot \|_\infty\right).\]
   The $\infty$-norm is naturally dual to the $\ell_1$-norm, so the dual chain complex consists of \emph{bounded} cochains.  
Define the \emph{bounded cohomology} $\Hb \bullet(X ;\bR)$ as the cohomology of this complex.  
Gromov \cite{Gromov:bdd} gave an argument, using the theory of multicomplexes, showing that for reasonable spaces $M$, the homotopy class of a classifying map $ M \to K(\pi_1(M),1)$ induces an isometric isomorphism $\Hb\bullet(\pi_1 (M);\bR)\to\Hb \bullet(M;\bR)$.  Recently, \cite{FM:multicomplexes} provided  a thorough treatment of the theory of multicomplexes and gave self contained account of Gromov's theorem; they also give other applications to bounded cohomology and simplicial volume.  See also \cite{Ivanov:foundations} for an approach using normed homological algebra. 

For a discrete group $G$, we will consider the normed chain complex $(\Cc \bullet (G;\bR), \partial, \|\cdot \|_1)$ that defines the homology of $G$.  The collection of \emph{$n$-co-invariants} $\Cc n(G; \bR)$ of $G$ is the $\bR$-linear span of $\Sigma_n(G)=\{(g_0, ..., g_n): g_i \in G\}/\sim$, where $\sim $ is the equivalence relation generated by $(g_0, ..., g_n) \sim (gg_0, ..., gg_n)$; we denote an equivalence class by $[g_0, ..., g_n]\in \Sigma_n(G)$, and we think of $[g_0, ..., g_n]$ as an $n$-simplex in the universal cover of a $K(G,1)$ for $G$, defined up to covering transformations, thus defining a simplex in the quotient $K(G,1)$.  
A group chain or $n$-co-invariant  $Z\in \Cc n(G;\bR)$ is then a sum \[Z = \sum_{i =1}^k a_i [g_0^i, ..., g_n^i],\]
where $[g_0^i, ..., g_n^i]\not = [g_0^j, ..., g_n^j]$ for $i\not = j$.  The $\ell_1$-norm is defined by $\|Z\|_1 = \sum_{i = 1}^k |a_i|$.  The boundary operator $\partial : \Cc n(G;\bR) \to \Cc {n-1}(G;\bR)$ is the pre-dual of the co-boundary operator $\delta$.  One thinks of $\partial$ as the alternating sum of face maps on $n$-simplices.   If $f\in \Cb n(G;\bR)$ and $Z \in \Cc n(G;\bR)$, then we have a trivial inequality $|f(Z)|\le \|f\|_\infty \|Z\|_1$.  

\subsection{Isometric chain maps}\label{isometric chain maps}
We will be interested in free marked Kleinian groups $\rho: F_d\to \G$, i.e. $F_d$ is a free group of rank $d$ and $\rho$ is a discrete and faithful representation.  Thus $\im\rho = \g$ acts properly discontinuously by orientation preserving isometries on $\bH^3$, and the space $M_\rho = \bH^3/\g$ of orbits is a complete hyperbolic $3$-manifold of infinite volume.  Call the orbit projection $\pi: \bH^3\to M_\rho$.  
There is a natural subspace of the singular chain complex $\Cc \bullet (M_\rho)$ obtained by \emph{straightening}.  We have an $\ell_1$-norm non-increasing chain map  \cite{Thurston:notes}
\[\str : \Cc \bullet (M_\rho;\bR) \to \Cc \bullet (M_\rho;\bR)\] defined by homotoping a singular $n$-simplex  $\sigma: \Delta_n\to M_\rho$,  relative to its vertex set, to the unique locally geodesic hyperbolic tetrahedron $\str \sigma$.  We ignore issues of parameterization, though Thurston provides a natural one in \cite[Chapter 6.1]{Thurston:notes}.  Then $\Cc \bullet^{\str}(M_\rho;\bR)$ denotes the image of $\str$, and
if $\bar x\in M_\rho$, we denote by $\Cc \bullet^{\str}(M_\rho, \{\bar x\};\bR)$ the subcomplex spanned by the straight simplices whose vertices all map to $\bar x$.  We will now construct  a chain map $\str_{\bar{x}}: \Cc \bullet(M_\rho;\bR)\to \Cc \bullet^{\str}(M_\rho, \{\bar x\};\bR)$.

Fix $x\in \pi\inverse(\bar x)$, and let $\mathcal D = \{y\in \bH^3: d(x , y)\le d(\gamma( x) , y) \text{ for all } \gamma \in \g\}$ be the Dirichlet fundamental domain for $\g$ centered at $x$; delete a face of $\mathcal D$ in each face-pair $(F, \gamma F)$ to obtain a connected Borel set of representatives for the action of $\g$ on $\bH^3$, which we still call $\mathcal D$.  Let $\sigma: \Delta_n\to M_\rho$ and choose a lift $\widetilde {\sigma}:\Delta_n \to \bH^3$.  The vertices $v_0, ..., v_n$ of $\widetilde{\sigma}$ are uniquely labeled  by group elements $v_i = \gamma_i y_i$ where $\gamma_i\in \g$,  $y_i \in \mathcal D$.  Define $\str_{\bar x} \sigma = \pi(\sigma_x(\gamma_0, ..., \gamma_n))$, where $\sigma_x(\gamma_0, ..., \gamma_n)$ is the straightening of any simplex whose ordered vertex set is $(\gamma_0x, ..., \gamma_nx)$.  The definition is  independent of the choice of lift, because any other lift of $\sigma$ has vertex set equal to $(\gamma \gamma_0y_0, ..., \gamma\gamma_ny_n)$ for some $\gamma \in \g$.  All maps are chain maps and the operator norm satisfies $\| \str_{\bar x} \|\le 1$.  This is just because some simplices in a chain may collapse or cancel after applying $\str_{\bar x}$.   

 Let $\tau: \Delta_k\to M_{\rho}$ be a straight simplex.  We can apply the prism operator to the straight line homotopy between lifts of $\tau$ and $\str_{\bar x} \tau$ to $\bH^3$.  We obtain a chain  homotopy  $H^\bullet_{\bar x}: \textnormal C^{\str}_\bullet (M_{\rho};\bR) \to  \textnormal C^{\str}_{\bullet+1}(M_{\rho};\bR)$ between $\str_{\bar x}$ and $id$.   That is, \[H_{\bar x}^{\bullet - 1} \partial +\partial H_{\bar x}^\bullet = \str_{\bar x} - id.\]  
 Compare with \cite[Chapter 6.1]{Thurston:notes}.
 The homotopy space $\Delta_k\times I$ is triangulated by the prism operator using $k+1$ simplices of dimension $k+1$, so 
 \begin{equation}\label{eqn:hmtpy_norm}
 \|H^k_{\bar x}\|= k+1.
\end{equation}

If $Z\in \Cc n^{\str}(M_\rho, \{\bar x\};\bR)$, then $Z$ defines a chain in $\Cc n(\g;\bR)$  by linear extension of the rule
\[\pi(\sigma_x(\gamma_0, ..., \gamma_n)) \mapsto [\gamma_0, ... , \gamma_n].\] 
To see that this map is well-defined, observe that $\pi(\sigma_x(\gamma_0, ..., \gamma_n)) = \pi(\sigma_x(\gamma_0', ..., \gamma_n'))$ means that $\sigma_x(\gamma_0, ... , \gamma_n)$ differs from $\sigma_x(\gamma_0',...,\gamma_n')$ by a deck transformation $\gamma\in \g$; hence, $[\gamma_0', ..., \gamma_n'] = [\gamma\gamma_0, ..., \gamma \gamma_n] = [\gamma_0, ..., \gamma_n]$.  We denote this map by $\iota_x: \textnormal C^{\str}_\bullet (M_\rho,\{\bar x\};\bR)\to \Cc \bullet (\g;\bR)$.

One checks easily that $\iota_x$ is an isometric isomorphism of normed chain complexes with their $\ell_1$-norms.  Thus if $Z \in \Cc n^{\str}(M_\rho,\{\bar x\};\bR)$, then we have $\|Z\|_1 = \|\iota_x(Z)\|_1$ and $\|\partial \iota_x (Z)\|_1 = \|\iota_x (\partial Z)\|_1  = \|\partial Z\|_1$.
Conversely, if we have a chain $Z \in \Cc n^{\str} (\g;\bR)$, one sees that $\pi_*(Z.x) = \iota_x\inverse(Z) \in \Cc n^{\str}(M_\rho, \{\bar x\};\bR)$.

\subsection{The volume class}\label{volume class}
Let $x\in \bH^3\cup \partial \bH^3$ and consider the function $\vol^x_3: (\G)^4\to \bR$ which assigns to $(g_0, ..., g_3)$ the signed hyperbolic volume of the convex hull of the points $g_0x, ..., g_3x$.  Any geodesic tetrahedron in $\bH^3$ is contained in an ideal geodesic tetrahedron, and there is an upper bound $v_3$ on the volume that is attained by a \emph{regular} ideal geodesic tetrahedron.  That is, $\displaystyle \| \vol^x_3\|_\infty = v_3$.
One checks using Stokes' Theorem that $\delta\vol^x_3=0$, so that $[\vol^x_3]\in \Hcb 3(\G;\bR)$.  Moreover, for any $x,y\in \bH^3\cup \partial \bH^3$, we have $[\vol^x_3] = [\vol^y_3]$.  This is because the straight line homotopy between geodesic triangles can be triangulated by $3$ (partially ideal) tetrahedra using the prism operator, and so 
\begin{equation}\label{eqn:chain_homotopy}
\vol^x_3 - \vol_3^y = \delta H_{x,y},
\end{equation}
where $H_{x,y} \in \Cb 2(\G;\bR)$ measures the volume of the straight line homotopy between the geodesic triangles $(g_0, g_1, g_2).x$ and $(g_0, g_1, g_2).y$.  In particular, $\|H_{x,y}\|_\infty \le 3v_3$, so that $[\vol^x_3] = [\vol^y_3]$, as claimed; the previous section gives a dual discussion.  If basepoints are not relevant, we write $[\vol_3]$ to denote the class of $[\vol^x_3]$.

The continuous bounded cohomology of $\G$ is generated by $[\vol_3]$, i.e. $\Hcb3(\G;\bR)= \langle [\vol_3]\rangle_\bR$ \cite{Bloch}. In fact, $\|[\vol_3]\|_\infty = v_3$;  see e.g. \cite{BBI:mostow} for a discussion of the hyperbolic volume class in dimensions $n\ge 3$.
If $\g$ is a discrete group and $\rho: \g \to \G$ is a group homomorphism, then $[\rho^*\vol_3]\in\Hb3(\g;\bR)$ is called the \emph{bounded fundamental class of $\rho$} or the \emph{volume class of $\rho$}.  Observe that for any $g\in \PSL_2\bC$, we have an equality at the level of cochains \[ (g\rho g\inverse)^*\vol_3^{gx} = \rho^*\vol_3^{x},\] so that the volume class is an invariant of the $\G$-conjugacy class of $\rho$.  It is also true that $[\rho^*\vol_3]$ is an invariant of the {\it closure} of the action of $\G$ by conjugation, if $\g$ is finitely generated.  Indeed, Thurston \cite[proof of Proposition 1.1]{Thurston:vol} pointed out that for non-conjugate representations $\rho$ and $\rho'$ of a finitely generated, discrete group $\g$ into $\Isom^+(\bH^3)$, $\rho\in \overline{\G.\rho'}$ implies that both $\im\rho$ and $\im \rho'$ are virtually abelian.  By Theorem \ref{soma's theorem} or Lemma \ref{lem:geom_elem} below, the volume classes of such representations are zero.  Hence if $\g$ is finitely generated, we obtain a well-defined, equivariant function on the character variety \[\Hom(\g, \G)/\!\!/\G\to \Hb3(\g;\bR)\] with respect to the natural actions of $\Out(\g)$ on each space.

The following theorem of Soma characterizes when a finitely generated Kleinian representation has non-trivial volume class.  See Section \ref{hyperbolic} for definitions of geometrically finite hyperbolic manifolds and characterizations of ends of geometrically infinite hyperbolic manifolds.  
\begin{theorem}[{\cite[Theorem 1]{Soma:Kleinian}}]\label{soma's theorem}
If $\g\le \PSL_2\bC$ is a finitely generated Kleinian group of infinite co-volume without elliptic elements and  $\rho: \g \to \G$ is any discrete and faithful representation, then the following are equivalent:
\begin{itemize}
\item $[\rho^*\vol_3] = 0\in \Hb 3(\g;\bR)$
\item $\|[\rho^*\vol_3]\|_\infty < v_3$
\item $M_\rho$ is geometrically finite or $\g$ is virtually abelian.
\end{itemize}
\end{theorem}	

\noindent Hence, if $M_\rho$ has a geometrically infinite relative end, then $\|[\rho^*\vol_3]\|_\infty =v_3$.  See  \cite[Lemma 3.2 and Proposition 3.3]{Soma:boundedsurfaces} for the proof of this fact.  We record here an observation.

 \begin{lemma}\label{lem:geom_elem}
 Let $G$ be a discrete group.  If $\rho: G\to \G$ is indiscrete but not dense, then $\rho$ is geometrically elementary and  $[\rho^*\vol_3] = 0\in \Hb3(G; \bR)$.  
 \end{lemma}
 
 \begin{proof}
 Since $\rho$ is indiscrete and not dense, $H = \overline{\rho(G)}\le \G$ is a proper, closed Lie subgroup.  According to \cite[Proposition, p. 246]{Sullivan:stability}, $H$ fixes a point, an ideal point, or stabilizes a geodesic plane, i.e. $H$ is geometrically elementary.  
 
 We claim that $\rho^*\vol^y_3 =0$ for some $y\in \bH^3\cup \partial \bH^3$.  We just need to choose $y$ to be contained in the invariant point or plane so that every tetrahedron has zero volume.  Since $\rho^*\vol^y_3 =0\in[\rho^*\vol_3]$, it follows that $[\rho^*\vol_3]=0$.  
\end{proof}

\subsection{The bounded Borel class}\label{borel}
We will consider a generalization of the hyperbolic volume class $\beta_n \in \Hcb 3(\PSL_n\bC;\bR)$ called the \emph{bounded Borel class}, which coincides with the $3$-dimensional hyperbolic volume class for $n=2$.
Using a stability theorem of Monod \cite{Monod:stability}, Bucher, Buger, and Iozzi \cite[Theorem 2]{BBI:borel} show that $\langle \beta_n\rangle_\bR = \Hcb3(\PSL_n\bC;\bR)$, for all $n\ge3$, and they compute the $\ell_1$-norm of $\beta_n$; see Theorem \ref{Borel class theorem}.

Let $\mathscr F(\bC^n)$ be the space of complete flags of $\bC^n$.  That is, $F \in \mathscr F(\bC^n)$ is a sequence of vector subspaces $\{0\}\le F^1 \le ... \le F^n = \bC^n$ such that $\dim_\bC (F^j) = j$.  We fix $F\in \mathscr F (\bC^n)$ and consider a Borel measurable $\PSL_n\bC$-invariant function 
\begin{align}\label{eqn:almost_definition}
B_n^F: (\PSL_n\bC)^4 &\to\bR \\ \nonumber
(g_0, ..., g_3)& \mapsto B_n(g_0.F, ..., g_3.F) \nonumber
\end{align}
that satisfies the cocycle condition everywhere, but is not everywhere continuous.  Equation (\ref{eqn:borel_def}), below, provides a useful definition of $B_n: \mathscr F(\bC^n)^4\to \bR$ on generic configurations of flags.  The cocycle $B_n$ was defined on generic configurations  of quadruples of flags in \cite[Section 2]{Goncharov:Borel} and extended to non-generic configurations in \cite[Equation (6)]{BBI:borel}.

\para{Continuity} Since $B_n^F$ is not everywhere continuous, $B_n^F\not\in \Cb 3 (\PSL_n\bC ;\bR)$, the continuous cochain group from Section \ref{bounded groups}.  However, in the appropriate cochain complex that computes the continuous bounded cohomology of $\PSL_n\bC$, $B_n^F$ represents $\beta_n$; we prefer to omit the technical details of the construction of the {\it strong resolution of $\bR$ by relatively injective $\PSL_n\bC$-Banach modules} in which $B_n^F$ is a cocycle (see \cite[Sections 3-7]{BBI:borel}).  Instead it is convenient to think of $[B_n^F] \in \Hb 3(\PSL_n\bC;\bR)$, where $\PSL_n\bC$ is given the \emph{discrete} topology.  

More precisely, let $\PSL_n\bC^\delta$ denote $\PSL_n\bC$ with the discrete topology and let $id: \PSL_n\bC^\delta \to \PSL_n\bC$ be the identity.  Then $id$ is continuous, hence induces a map $id^*: \Hcb \bullet (\PSL_n\bC;\bR) \to \Hcb \bullet (\PSL_n\bC^\delta; \bR)$.  We can work with the class 
\[id^*\beta_n = [B_n^F] \in \Hcb 3(\PSL_n\bC^\delta;\bR) = \Hb 3(\PSL_n\bC;\bR).\] 

Our  goal in this section is to describe the main results from \cite{BBI:borel} and to extract a continuity property of the cocycle $B_n^F$ on non-degenerate generic configurations of complete flags.  

\para{The case $n=2$ and the Bloch-Wigner function}
Note that $\mathscr F(\bC^2) = \CP^1 = \partial \bH^3$. The \emph{Bloch--Wigner function} $D: \CP^1\to \bR$ is a variant of the di-logarithm function that computes the volume of the ideal tetrahedron with ordered vertex set $( \infty, 0, 1, z)\in (\CP^1)^4$.  If  $z_0, z_1, z_2 \in \CP^1$ are distinct and $z_3\in \CP^1$ is arbitrary, then there exists a unique $g\in \G$ such that $g.(z_0, ..., z_3) = (\infty, 0, 1, g.z_3)$.  Then $[z_0:...:z_3]:=g.z_3$ is the \emph{cross ratio} of the $4$-tuple, and 
\[\vol(z_0, ..., z_3):=D([z_0:...:z_3])\] is the oriented volume of the ideal geodesic tetrahedron spanned by $(z_0, ..., z_3)$. When at least two of $z_0, ..., z_3$ coincide, $\vol(z_0, ..., z_3) =0$.  

It is well known that $D$ is real analytic on $\bC\setminus \{0, 1, \infty\}$, attains its extreme values at $\zeta_3 = e^{i\pi/3}$ and $\bar \zeta_3$,  and extends continuously to $\CP^1$.  However, $\vol: (\CP^1)^4\to \bR$ is not continuous everywhere.  Indeed, consider the sequence $\underline z_k = (\infty, 0, 2^{-k}, 2^{-k}\zeta_3)$, for $k= 0, 1, ...$; the M\"obius transformation $z\mapsto 2^kz$ takes $\underline z_k$ to $\underline z_0$, hence $\vol(\underline z_k)$ is constant (and positive), but the limiting configuration $(\infty, 0, 0 ,0 )$ has $0$ volume.  For fixed $z\in \CP^1$, it is not difficult to see that the map 
\begin{align*}
B_2^z: (\G)^4&\to \bR\\
(g_0, ... ,g_3)& \mapsto \vol(g_0.z, ..., g_3.z)
\end{align*}
is a $\G$-invariant cocycle.  
The bounded Borel class is just $\beta_2: = [\vol_3^x] \in \Hcb3 (\G;\bR)$, where $x\in \bH^3$ is {\it not} an ideal point, and it is clear that $id^*[\vol_3^x] = [B_2^z] \in \Hb 3(\G;\bR)$ for any $x\in \bH^3\cup \partial \bH^3$; see Section \ref{volume class}.   
In Section \ref{sec:applications}, we show that  $\textnormal H_{\textnormal{b}}^3(\G;\bR)$ is quite large.

\para{Generic configurations of flags} For convenience, we fix a Hermitian inner product $\bC^n\otimes \bC^n \to \bC$, thus identifying a maximal subgroup $K\le \PSL_n\bC$ preserving this inner product.  Then $K$ is a maximal compact subgroup isomorphic to $\textnormal{PSU}(n)$.  
Note that $\PSL_n\bC$ acts transitively on $\mathscr F(\bC^n)$, so that if we choose also an orthonormal ordered basis $(e_1, e_2, ... ,e_n)$, we may identify the stabilizer of the standard flag $\langle e_1\rangle \le \langle e_1, e_2\rangle \le ... \le \langle e_1, ..., e_n\rangle$ with the upper triangular group $P\le \PSL_n\bC$, and $\mathscr F(\bC^n) \cong \PSL_n\bC/ P$. 

Given a flag $F\in \mathscr F(\bC^n)$, using the Gram--Schmidt process and our chosen inner product, we may find an orthonormal basis $(f^1, ..., f^n)$ such that 
\[F^j = \langle f^j\rangle \perp \langle f^1, ..., f^{j-1}\rangle, \]
and each $f^j$ is uniquely determined up to multiplication by a complex number of norm $1$.  Call $(f^1, ..., f^n)$ an \emph{affine representative} of $F$.  

Say  that $(F_0, ..., F_3)\in \mathscr F(\bC^n)^4$ is a \emph{generic} configuration of flags if  whenever $0\le j_0, ..., j_3\le n-1$ satisfy $j_0 + ...+j_3 = k$, then $\dim \langle F_0^{j_0}, ..., F_3^{j_3}\rangle = \min \{n,k\}$; genericity is an open condition among $4$-tuples of flags.   Let 
\[ \mathbb M = \{(j_0, ..., j_3) \in \{0,1, ... ,n-2\}^4 : j_0 + ... + j_3 = n-2 \}, \]
so that if  $(F_0, ..., F_3)$ is a generic configuration of flags and $(j_0, ..., j_3)\in \mathbb M$, then
\begin{equation}\label{eqn:dim_quotient}
\dim_\bC \frac{\langle F_0^{j_0+1}, ..., F_3^{j_3+1}\rangle}{\langle F_0^{j_0}, ..., F_3^{j_3}\rangle}= 2.
\end{equation}
Moreover, $\#\mathbb M = \frac{1}{6}n(n-1)(n+1)$, and if $(j_0, ..., j_3)\in\{0, ..., n-1\}^4\setminus \mathbb M$, then equality (\ref{eqn:dim_quotient}) does not hold  \cite[p. 3147]{BBI:borel}. 

\para{Definition of the cocycle on generic configurations} Let $\mathscr F(\bC^n)^{(4)}\subset \mathscr F(\bC^n)^{4}$ be the subspace of generic configurations of flags.   Let  $\underline F = (F_0, ..., F_3)\in\mathscr F(\bC^n)^{(4)}$ and $\mathbb J= (j_0, ..., j_3)\in \mathbb M$.  Find also an affine representative  $(f_i^1, ..., f_i^n)$ of $F_i$, for $i = 0, ..., 3$.   The functions 
\[\underline F \mapsto V_{\mathbb J} (\underline F) := \langle f_0^1, ..., f_0^{j_0}, ..., f_3^1, ..., f_3^{j_3}\rangle \in \textnormal {Gr}_{n-2}(\bC^n)\] and 
\[\underline F \mapsto V_{\mathbb J} (\underline F)^\perp \in \textnormal {Gr}_{2}(\bC^n)\]
vary continuously in generic configurations $\underline F$.   The orthogonal projection 
\[\bC^n =  V_{\mathbb J} (\underline F) \oplus   V_{\mathbb J} (\underline F)^\perp \to V_{\mathbb J} (\underline F)^\perp.\]
coincides with the quotient
\[\bC^n \to \frac{\langle F_0^{j_0+1}, ..., F_3^{j_3+1}\rangle}{\langle F_0^{j_0}, ..., F_3^{j_3}\rangle}= \frac{\bC^n}{V_{\mathbb J} (\underline F)}\cong V_{\mathbb J}(\underline F)^\perp.\]
Using genericity again, the orthogonal projections of $f_0^{j_0+1} ,..., f_3^{j_3+1}$ to $V_{\mathbb J}(\underline F)^\perp$ are non-zero and  pairwise linearly independent.  Let \[t_{\mathbb{J}}(\underline F)\in \mathbb P (V_{\mathbb J}(\underline F)^\perp)^4 \cong (\CP^1)^4\] be the $4$-tuple consisting of the projectivizations of the orthogonal projections of $f_i^{j_i+1}$ to $V_{\mathbb J}(\underline F)^\perp$.  Observe that $t_{\mathbb J}$ is independent of our choice of affine representatives of $F_0, ..., F_3$, because the projectivization map only depends on the complex line spanned by $f_i^j$. 

Following \cite{Goncharov:Borel} and \cite{BBI:borel}, we define the $\PSL_n\bC$-invariant function 
\begin{align}\label{eqn:borel_def}
B_n: \mathscr F(\bC^n)^{(4)} &\to \bR\\
\underline F &\mapsto  \sum_{\mathbb J \in \mathbb {M}} \vol(t_{\mathbb J}(\underline F)). \nonumber
\end{align}
From the definition, it is clear that \[\sup_{\underline F\in \mathscr F(\bC^n)^{(4)}}|B_n(\underline F)|\le \#\mathbb M\cdot v_3 = \frac{(n^3-n)}{6}\cdot v_3.\]  

By avoiding non-generic configurations of flags, for each $\mathbb J \in \mathbb M$, the function $\underline F \mapsto t_{\mathbb J}(\underline F)$ varies continuously and has image contained in the locus of {\it distinct} $4$-tuples in $(\CP^1)^4$.  Since $\vol$ varies continuously on distinct $4$-tuples, all summands in (\ref{eqn:borel_def}) are continuous, and so $B_n$ is continuous with respect to the subspace topology on $\mathscr F(\bC^n)^{(4)}$.  
We refer the reader to \cite[Section 3]{BBI:borel}, which explains how to extend $B_n$ to a bounded Borel measurable cocycle $\mathscr F(\bC^n)^4\to \bR$.  After defining this extension, for a fixed flag $F\in \mathscr F(\bC^n)$, equation (\ref{eqn:almost_definition}) gives a $\PSL_n\bC$-invariant alternating cocycle $B_n^F$ \cite[Corollary 13]{BBI:borel}.

\para{The Veronese embedding} 
There is a unique irreducible representation $\iota_n: \G\to \PSL_n\bC$, up to conjugation, and it induces an equivariant map  $\hat\iota_n: \partial \bH^3 \to \mathscr F(\bC^n)$ called the \emph{Veronese embedding}.  

Let $x\in \partial \bH^3$; the cocycle $B_n^{\hat \iota_n (x)}$ pulls back to a multiple of the volume cocycle $\vol_3^x$, which allows us to push constructions in $3$-dimensional hyperbolic geometry to the higher rank setting.
\begin{proposition}[{\cite[Proposition 21]{BBI:borel}}]\label{prop:borel_values}
For all $g_0, ..., g_3\in \G$ and $\mathbb J\in \mathbb M$, 
\[ \vol(t_{\mathbb J}(\hat\iota_n(g_0.x), ...,\hat\iota_n(g_3.x))) = \vol(g_0.x, ..., g_3.x),\]
so that
\begin{equation*}
  B_n^{\hat\iota_n(x)}(\iota_n(g_0), ..., \iota_n(g_3)) = \frac{n(n^2-1)}{6}\vol_3^x(g_0, ..., g_3).
  \end{equation*}
  In other words, $\iota_n^*B_n^{\hat\iota_n(x)}=\frac{n(n^2-1)}{6}\vol_3^x$.
\end{proposition}

The function $B_n^F$ is not continuous near non-generic configurations of flags.  However, $B_n^F$ satisfies the following continuity property, which is all we need in the sequel.

\begin{lemma}\label{lem:borel_cont}
Let $x\in \partial \bH^3$ and suppose $g_0, ..., g_3\in \G$ are such that $g_0.x, ..., g_3.x$ are pairwise distinct.  Then there exist open neighborhoods $U_i\subset \PSL_n\bC$ of $\iota_n(g_i)$ such that the map 
\begin{align*}
U_0 \times ... \times U_3 \subset (\PSL_n\bC)^4 & \to \bR\\
(h_0, ..., h_3)&\mapsto  B_n^{\hat\iota_n(x)}(h_0, ..., h_3)
\end{align*}
is continuous.  
\end{lemma}
\begin{proof}
Since $4$-tuples of distinct flags in the image of $\hat\iota_n$ are in generic position, we have  $(\hat\iota_n(g_0.x), ..., \hat\iota_n(g_3.x)) \in \mathscr F(\bC^n)^{(4)}$.  Let $V\subset \mathscr F(\bC^n)^{(4)}$ be an open neighborhood of $(\hat\iota_n(g_0.x), ..., \hat\iota_n(g_3.x))$, so that $V$ is open in $ \mathscr F(\bC^n)^4$, as well.  Consider the stabilizer $P_x$ of $\hat \iota_n(x)$, the quotient projection $\pi: \PSL_n\bC\to  \PSL_n\bC/P_x\cong \mathscr F(\bC^n)$, and the product $\pi^4: (\PSL_n\bC)^4\to \mathscr F(\bC^n)^4$.   Then $(\pi^4)\inverse (V)\subset (\PSL_n\bC)^4$ is an open neighborhood of $(\iota_n(g_0), ..., \iota_n(g_3))$.  Products of open sets form a basis for the topology of the product of spaces, and  $B_n$ is continuous on $\mathscr F(\bC^n)^{(4)}$.  This completes the proof of the lemma.
\end{proof}

\noindent We will need one more important result about the bounded Borel class and its semi-norm.
\begin{theorem}[{\cite[Theorem 2]{BBI:borel}}]\label{Borel class theorem}
For each $n\ge 2$, the Borel class $\beta_n$ generates $\Hcb 3(\PSL_n\bC;\bR)$, and its $\ell_\infty$-norm is \[\|\beta_n\|_\infty = v_3 \frac{n(n^2-1)}{6}.\]
For the irreducible representation $\iota_n: \G\to \PSL_n\bC$, the pullback satisfies \[\iota_n^*\beta_n = \frac{n(n^2-1)}{6}[\vol_3] \in \Hcb 3(\G;\bR).\]
\end{theorem}
For a discrete group $G$ and representation $\rho: G \to \PSL_n\bC$, the \emph{bounded Borel class of $\rho$} or the \emph{Borel class of $\rho$} is $\rho^*\beta_n\in \Hcb 3(G;\bR)$.  Note that by Theorem \ref{Borel class theorem}, if $\rho: F_2 \to \PSL_2\bC$ is discrete, faithful, and geometrically infinite, then $\|(\iota_n\circ\rho)^*\beta_n\|_\infty = \frac{n(n^2-1)}{6} \|[\rho^*\vol_3]\|_\infty  = v_3 \frac{n(n^2-1)}{6}$, where the last equality was by Theorem \ref{soma's theorem}.

\subsection{An approximation scheme}\label{approximation section}

We now give a criterion for the pullback of a continuous bounded class to be non-zero and have positive semi-norm in bounded cohomology.  We claim no originality for the following lemma; it is a distillation and abstraction of a standard argument.  See  \cite[Section 3]{Yoshida:bdd} and \cite[Proposition 3.3]{Soma:boundedsurfaces}. 

\begin{lemma}\label{scheme}
Let $G$ be a discrete group, $\mathcal G$ a group, $\rho: G \to \mathcal G$ a homomorphism, and $[B]\in \Hb n (\mathcal G; \bR)$.  Suppose there exist $\epsilon>0$ and chains $Z_k \in \Cc n(G;\bR)$ for $k = 1, 2, ...$ such that 
\begin{enumerate}[(i)]
\item $\displaystyle \frac{|\rho^*B(Z_k)|}{\|Z_k\|_1}>\epsilon$ for all $k$,
\item $\displaystyle \liminf_{k\to \infty} \frac{\|\partial Z_k\|_1}{\|Z_k\|_1} = 0$. 
\end{enumerate}
Then $[\rho^*B]\not = 0\in \Hb n(G;\bR)$ and $\|[\rho^*B]\|_\infty \ge \epsilon$. 
\end{lemma}

\begin{proof}
Given $b\in \Cb{n-1}(G;\bR)$, we need to show that $\|\rho^*B+\delta b\|_\infty >\epsilon$.   We have the trivial inequality 
\[|(\rho^*B+\delta b)(Z_k)| \le \|\rho^*B+\delta b\|_\infty \|Z_k\|_1.\]
By the triangle inequality, we have 
\[|(\rho^*B+\delta b)(Z_k)| \ge |\rho^*B(Z_k)|-|\delta b(Z_k)|.\]
Another application of the trivial inequality yields
 \[|\delta b(Z_k)| = |b(\partial Z_k)|\le \|b\|_\infty \|\partial Z_k\|_1.\]
 Stringing together the inequalities and dividing through by $\|Z_k\|_1$, we obtain
  \[ \frac{|(\rho^*B+\delta b)(Z_k)|}{\|Z_k\|_1} \ge \frac{|\rho^*B(Z_k)|}{\|Z_k\|_1} - \|b\|_\infty \frac{\|\partial Z_k\|_1}{\|Z_k\|_1}. \]
  
  By passing to a subsequence, we assume that $\displaystyle \lim_{k\to \infty} \frac{\|\partial Z_k\|_1}{\|Z_k\|_1} = 0$, and we see that for any $e>0$, there is a $K$ such that for $k\ge K$, we have $\frac{|(\rho^*B+\delta b)(Z_k)|}{\|Z_k\|_1}>\epsilon-e$; thus $\|\rho^*B+\delta b\|_\infty> \epsilon - e$. Since $b$ and $e$ were arbitrary, this implies that $\|[\rho^*B]\|\ge\epsilon$.  
\end{proof}

The following lemma is an easy consequence of the continuity of multiplication and inversion in a topological group $\mathcal G$.
\begin{lemma}\label{approximation}
Let $\rho_0: F_d \to \mathcal G$ be a homomorphism, $W\subset F_d$ a finite set, and $\{z^1, ..., z^d\}$ a free basis for $F_d$.  Given neighborhoods $V_w$ of $\rho_0(w)$ for $w\in W$, there are neighborhoods $U_{i}$ of $\rho_0(z^i)$ such that if $\rho: F_d\to \mathcal G$ is any homomorphism satisfying $\rho(z^i)\in U_{i}$ for each $i =1, ..., d$, then  $\rho(w)\in V_w$, for all $w\in W$.   
\end{lemma}

Note that in the above lemma, we do not require $\rho$ to be faithful.  For example, given a homomorphism $\rho': G\to \mathcal G$, and a set $\{g_1, ..., g_d\}\subset G$, the rule $z^i\mapsto \rho'(g_i)$ defines a homomorphism  $\rho: F_d \to \mathcal G$.  The following proposition is the key insight in this paper.

\begin{proposition}\label{approximation of chains}
Let $G$ be a discrete group, $\mathcal G$ be a topological group, and $\rho: G\to \mathcal G$ a homomorphism with dense image.  Consider a homomorphism $\rho_0: F_d\to \mathcal G$ and a continuous cocycle $B\in \Cb n(\mathcal G;\bR)$.  For any $Z\in \Cc n (F_d;\bR)$ and for any $\epsilon>0$, there exists $Z(\epsilon) \in \Cc n(G;\bR)$ such that \[|\rho_0^*B(Z)- \rho^*B (Z(\epsilon)) | <\epsilon.\]  
Moreover, $\|Z(\epsilon)\|_1\le \|Z\|_1$ and $\|\partial Z(\epsilon)\|_1 \le \|\partial Z\|_1$.  
\end{proposition}

\begin{proof}

Write $Z = \sum_{j = 1}^M a_j [v_0^j, ..., v_n^j]$, and choose the representative $ (id, w_1^j, ..., w_n^j)\in [v_0^j, ..., v_n^j]$ so that $w_i^j = (v_0^j)\inverse v_i^j$ for each $i = 1, ..., n$.  Take $W = \{w_i^j: i= 1,...,n \text{ and } j = 1, ..., M\}\subset F_d$, and let $\{z^1, ..., z^d\}$ be a free basis for $F_d$. 

The cocycle $B: \mathcal G^{n+1}\to \bR$ is continuous, so for each $i = 1, ..., n$ and $j = 1, ..., M$, there are neighborhoods $V_i^j\subset \mathcal G$ of $\rho_0(w_i^j)$ such that if $\gamma_i^j\in V_i^j$, then 
 \begin{equation}\label{simplex volume}
 |B(id,\rho_0(w_1^j), ..., \rho_0(w_n^j ))- B(id,\gamma_1^j, ..., \gamma_n^j)|<\frac{\epsilon}{\|Z\|_1}, ~\forall j = 1, ... , M.
 \end{equation}

Given the data $\rho_0$, $W\subset F_d$, $\{z^1, ..., z^d\}$, and $V_i^j$, find neighborhoods $U_i\subset \mathcal G$ of $\rho_0(z^i)$ guaranteed to us by Lemma \ref{approximation}.  Now, $\rho$ has dense image, so we can find $z^i_\epsilon\in G$ such that $\rho(z^i_\epsilon) \in U_i$ for each $i = 1, ...,d$.  Define  $\iota_\epsilon : F_{d}\to G$ by $z^i\mapsto z^i_\epsilon$.  Then $\rho(\iota_\epsilon(w_i^j)) \in V_i^j$ for all $w_i^j\in W$.   

We write $Z(\epsilon) = \iota_{\epsilon*}(Z) \in \Cc n(G;\bR)$ for the chain corresponding to $Z$ under this identification. The map $i_{\epsilon*}$ is  an $\ell_1$-norm non-increasing chain map, so $\|Z(\epsilon)\|_1\le \|Z\|_1$ and $\|\partial Z(\epsilon)\|_1 \le \|\partial Z\|_1$. 

 Repeated applications of the triangle inequality and (\ref{simplex volume}) give  
 \[|\rho_0^*B(Z) - \rho^*B(Z(\epsilon)) | < \frac{\epsilon}{\|Z\|_1}\|Z\|_1=\epsilon,\] 
 which is what we wanted to show.
\end{proof}

\begin{remark}\label{rmk:to_infinity}
Note that if $Z_k\in \Cc n (F_d;\bR)$ satisfy $\|Z_k\|_1\to \infty$ and $Z_k(\epsilon)\in \Cc n(G;\bR)$ are obtained as in Proposition \ref{approximation of chains}, then $\|Z_k(\epsilon ) \|_1\to \infty$ as well.
\end{remark}

We will apply Lemma \ref{scheme} and Proposition \ref{approximation of chains} in two settings.  In Section \ref{volume classes}, we set $\mathcal G = \Isom^+(\bH^3) = \G$ and $[B] = [\vol_3^x]\in \Hcb 3 (\Isom^+(\bH^3); \bR)$, where $x\in \bH^3$.  
In Section \ref{Borel}, we set $\mathcal G = \PSL_n\bC$, for $n\ge 2$ and consider the cocycle $ B = B_n^F: (\PSL_n\bC)^4\to \bR$ from Section \ref{borel} representing the  Borel class $\beta_n\in \Hcb 3(\PSL_n\bC;\bR)$.  Later, we consider hyperbolic volume classes in dimensions $4$ and higher and establish a criterion that would guarantee that the volume class of a dense representation does not vanish.  Unfortunately, we do not know if the criterion is ever satisfied.  

We have already encountered a technical issue: the cocycle $B_n^F$ is {\it not} everywhere continuous.  However, Lemma \ref{lem:borel_cont} provides us with {\it enough} continuity.  
It is straightforward to modify the proof of Proposition \ref{approximation of chains}, using Lemma \ref{lem:borel_cont}, to obtain the following corollary.  Recall that the Veronese embedding $\hat \iota_n: \partial \bH^3 \to \mathscr F(\bC^n)$ is a topological embedding that is equivariant with respect to the irreducible representation $\iota_n: \G\to \PSL_n\bC$.   
\begin{corollary}\label{Borel approximation}
Let $\rho_0: F_d\to \PSL_2\bC$ be a homomorphism, $\rho: G\to \PSL_n\bC$ be dense, and  $Z = \sum_{j = 1}^M a_i[w_0^j, ..., w_3^j]\in \Cc 3 (F_d;\bR)$.  

If there is a point $x\in \bH^3$ such that, for each $j= 1, ..., M$, the four points $\rho_0(w_0^j).x, ...., \rho_0(w_3^j).x\subset \partial \bH^3$ are pairwise distinct, then for any $\epsilon>0$, there exists $Z(\epsilon) \in \Cc 3(G;\bR)$ such that \[|(\iota_n\circ \rho_0)^*B_n^{\hat \iota_n (x)}(Z) - \rho^*B_n^{\hat \iota_n (x)} (Z(\epsilon)) | <\epsilon.\]  
Moreover,  $\|Z(\epsilon)\|_1\le \|Z\|_1$ and $\|\partial Z(\epsilon)\|_1 \le \|\partial Z\|_1$.  
\end{corollary}

The reader who is interested only in the question of (non-)vanishing of higher dimensional volume classes can skip directly to Section \ref{sec:higher_dim}, and in particular Proposition \ref{prop:higher_dim}, where the ideas from this section are applied.  

\section{The structure of tame hyperbolic $3$-manifolds}\label{hyperbolic}
In this section, we review the classification theory of finitely generated Kleinian groups.  We use this classification to provide a detailed argument to deduce Theorem \ref{thm:intro_old_work} from Theorem \ref{thm:intro_main}  and previous work of the author, building on work of Soma.  
The presence of parabolic elements in Kleinian  groups significantly complicates the discussion of ends and end invariants.  The only part of the paper in which we need to understand influence of parabolic cusps is in Lemma \ref{lem:technical}, and is essentially independent of the main line of argument.  

We begin by defining some of the basic objects associated to a marked Kleinian group, turn to an observation about groups of isometries of non-positively curved symmetric spaces generated by `small' elements, and discuss the structure of the ends of complete hyperbolic $3$-manifolds with infinite volume and finitely generated fundamental group. We end with the proof of Theorem \ref{thm:intro_old_work}, once we have established these preliminary notions.  

A \emph{Kleinian group} is a discrete subgroup of $\G$.  Let $\Gamma$ be an abstract discrete group, and suppose $\rho: \Gamma \to \G$ is an injective group homomorphism with discrete image, i.e. $\rho$ is discrete and faithful.  We call $\rho$ a \emph{Kleinian representation} or a \emph{marked Kleinian group}. 
We will be most interested in the case that $\Gamma$ is a non-abelian free group or the fundamental group of a closed oriented surface of genus at least two; we will always assume that $\Gamma$  is torsion free, however if $\g$ is finitely generated and admits a Kleinian representation, then $\g$ is virtually torsion free, by Selberg's Lemma. 
 
We spend some time providing statements and context for the structure and classification theorems that we use in the sequel.  Namely, the Tameness Theorem (Theorem \ref{thm:tameness}), the Covering Theorem (Theorem \ref{thm:covering}), and the Ending Lamination Theorem (Theorem \ref{thm:ELT}) for finitely generated marked Kleinian groups.  
In order to state these theorems, we discuss families of hyperbolic surfaces, called \emph{simplicial hyperbolic surfaces}, that exit the geometrically infinite ends of hyperbolic $3$-manifolds. We use properties of simplicial hyperbolic surfaces again in Section \ref{sec:efficient_chains}. 

Given $\rho: \g\to \G$ as above, the space $M_\rho = \bH^3/\im \rho$ of orbits is a complete hyperbolic $3$-manifold. 
Since $\bH^3$ is contractible, $M_\rho$ is a classifying space for $\g$, and $\rho$ induces a homotopy class of maps $K(\Gamma, 1) \to M_\rho$ that labels the homotopy classes of loops in $M_\rho$.  If $\Gamma$ is a closed surface group or a free group, then we can take $K(\Gamma, 1)$ to be a closed surface or a $3$-dimensional handlebody, respectively.  

\subsection{The Margulis Lemma}
Let $\cG$ be a semi-simple Lie group of non-compact type, let $K$ be a maximal compact subgroup, and $X = \cG/K$ the associated Riemannian symmetric space.  In this paper,  we will be most interested in $\cG = \PSL_n\bC$ and $\cG = \Isom^+(\bH^d)$ for $n, d\ge 2$. Then $X = \bH^3$ when $n=2$ or $d = 3$.  Following Thurston \cite[Lemma 5.10.1]{Thurston:notes}, the following is known to hyperbolic geometers as ``The Margulis Lemma," although it is perhaps more accurate to refer to it as the ``Kazhdan--Margulis Theorem" \cite[Theorem 4.53]{Kapovich:book}.

\begin{lemma}[The Margulis Lemma] \label{lem:Margulis}
Given $\mathcal G$ and $X$ as above, there is a number $\mu = \mu(X)>0$ such that the following holds.  If $\Gamma = \langle \gamma_1, ..., \gamma_k\rangle \le \mathcal G$ is a discrete group, and there is a point $x\in X$ such that $d(x, \gamma_i.x)\le \mu$ for each $i = 1, ..., k$, then $\Gamma$ is virtually nilpotent.
\end{lemma}
For the rest of this section, we will only be interested in the case that $G = \G$ and $X = \bH^3$.  In this setting, the largest number $\mu_3 = \mu(\bH^3)$ for which Lemma  \ref{lem:Margulis} holds is called the \emph{$3$-dimensional Margulis constant}.  Any discrete virtually nilpotent subgroup $\Gamma\le \G$ is virtually isomorphic to $1, \bZ$ or $\bZ^2$.

Thurston \cite[Corollary 5.10.2]{Thurston:notes} pointed out that a hyperbolic manifold admits a \emph{thick--thin} decomposition as a consequence of Lemma \ref{lem:Margulis}.  Given a complete hyperbolic $n$-manifold $M$ and $\epsilon>0$, let
\[
M^{\ge \epsilon} = \overline{\{x\in M : \text{the ball of radius $\epsilon$ centered at $x$ is embedded in $M$}\}},
\]
be the \emph{thick part} of $M$ and $M^{<\epsilon} = M\setminus M^{\ge \epsilon}$ the \emph{thin part}.   In general, the \emph{injectivity radius} $\inj_x(M)$ is the supremum of the radii of balls centered at $x$ that embed into $M$. A hyperbolic manifold $M$ is said to have \emph{bounded geometry} if there is a positive lower bound to $\inj_x (M)$, which is independent of $x$.  

\begin{corollary}[Thick--Thin Decomposition]\label{cor:thick--thin}
Every component of $M^{<\mu_3}$ is either 
\begin{itemize}
\item the quotient of a horoball by a group of parabolic isometries that is free abelian of rank $1$ or $2$, or
\item the $r$ neighborhood of a closed geodesic in $M$, where $r\ge 0$.
\end{itemize}
\end{corollary}
\noindent The pre-compact components of $M^{<\mu_3}$ are called \emph{Margulis tubes}.  The horoball quotients are called \emph{rank-$1$} or \emph{rank-$2$  parabolic cusps}, accordingly.

\subsection{Relative ends and tameness}
Fix a Kleinian representation $\rho: \g\to \G$, where $\g$ is finitely generated and torsion free. 
Let $Q_\rho$ denote the union of parabolic cusps in $M_\rho^{<\mu_3}$, and let $M_\rho^\circ = M_\rho\setminus Q_\rho$.  The set $P_\rho = \partial M_\rho^\circ$ is called the \emph{parabolic locus} and consists of a finite number of tori and open annuli corresponding to the frontiers of the rank-$2$ and rank-$1$ parabolic cusps, respectively.  McCullough has shown \cite{McMullough:relative_core} that there is a \emph{relative compact core} $ K_\rho \subset M_\rho^\circ$, a co-dimension $0$ compact submanifold such that the inclusion $K_\rho\hookrightarrow M_\rho$ is a homotopy equivalence, $\partial K_\rho$ contains all toroidal components of $P_\rho$,  and $\partial K_\rho$ meets each annular component in a compact annulus.  We let $P_\rho^\circ = K_\rho \cap P_\rho$ so that $P_\rho^\circ \hookrightarrow \partial M_\rho^\circ$ is a homotopy equivalence.  

Given a relative compact core $K_\rho\subset M_\rho$,  
the ends $\mathcal E_\rho$ of $M_\rho^\circ$ are in one-to-one correspondence with the components of $\partial K_\rho \setminus P_\rho^\circ$, which are oriented surfaces of finite type and negative Euler characteristic.  The elements of $\mathcal E_\rho$ are called \emph{relative ends} of $M_\rho$; for each component $R\subset \partial K_\rho \setminus P_\rho^\circ$, the component $E_R$ of $M_\rho^\circ \setminus K_\rho$ whose closure contains $R$ is a neighborhood of the relative end $[E_R]\in \mathcal E_\rho$.  
 
The following theorem is known as the Tameness Theorem, and has been proved independently by Agol \cite[Section 6 and Theorem 10.2]{Agol:tameness} and Calegari--Gabai \cite[Theorem 7.3]{CG:tameness} building on important partial results of Marden, Thurston, Bonahon, Canary, Souto and many others. 
The statement of the tameness theorem that we include here can be found in \cite[Section 4.2]{Namazi--Souto:density}.   
\begin{theorem}[The Tameness Theorem]\label{thm:tameness}
There is a relative compact core $K_\rho\subset M_\rho^\circ$ such that for each component $R$ of $\partial K_\rho\setminus P_\rho^\circ$, the closure of $E_R$ is homeomorphic to $R\times [0, \infty)$.  Moreover, $M_\rho$ is homeomorphic to $\interior K_\rho$, and   $M_\rho^\circ$ is homeomorphic to the complement in $K_\rho$ of $\partial K_\rho \setminus P_\rho^\circ$.
\end{theorem}
 
 If $K_\rho\subset M_\rho$ is a relative compact core such that the complement of $K_\rho$ in $M_\rho^\circ$ consists of product neighborhoods of ends, as in Theorem \ref{thm:tameness}, say that $K_\rho$ is a \emph{standard} relative compact core.  
 A \emph{compression body} $B$ is an oriented, compact, irreducible $3$-manifold which has a distinguished boundary component $\partial_{ext}B$ whose fundamental group surjects onto $\pi_1(B)$; 
 $\partial_{ext}B$ is called the \emph{exterior boundary} of $B$.  If the exterior boundary of $B$ is incompressible, then $B$ is homeomorphic to the trivial interval bundle over $\partial _{ext} B$. 
 If $B$ is not a trivial interval bundle, the exterior boundary is the unique compressible component of $\partial B$.  The compression body $B$ is a handlebody if and only if $\partial B = \partial _{ext} B$, which is true if and only if $\pi_1(B)$ is a free group; see \cite[Section 2.5]{Souto:tameness_note}.  
  \begin{theorem}[{\cite[Theorem 2.1]{Bonahon:compression}}]\label{thm:compression_body}
 Let $S$ be a component of $\partial K_\rho$.  There is a compression body $B_S\subset K_\rho$ whose exterior boundary is $S$ such that inclusion $B_S\hookrightarrow K_\rho$ induces an injection $\pi_1(B_S)\hookrightarrow \pi_1(K_\rho) = \pi_1(M_\rho)$, and $B_S$ is unique up to isotopy.
 \end{theorem}
The submanifold $B_S$ is called the \emph{characteristic compression body neighborhood} of $S$ for $K_\rho$;  see also \cite[Theorem 1.1.1]{McCullough--Miller}, where characteristic compression body neighborhoods are called \emph{incompressible neighborhoods}.  If a component of $\partial K_\rho$ is incompressible, then the characteristic compression body neighborhood corresponding to that component is just a collar neighborhood.  Furthermore, the covering space corresponding to that surface group is a trivial $\bR$ bundle over that surface by Bonahon's Tameness Theorem \cite[Th\'eor\`eme A]{Bonahon:ends}. 
The following corollary is our main application of the Tameness Theorem \ref{thm:tameness}.  
\begin{corollary}\label{cor:handlebody}
If $\g$ is a free group of rank $k<\infty$, $\rho: \g \to\G$ is discrete and faithful and $K_\rho$ is a standard compact core for $M_\rho$, then there is a homeomorphism $f: \cH_k \to K_\rho$, where $\cH_k$ is a closed handlebody of genus  $k$ and $f_* = \rho$ on fundamental groups. There are homotopically essential and distinct simple closed curves $\nu_1, ..., \nu_m \subset \partial \cH_k$ with disjoint representatives such that $f(\{\nu_1, ... ,\nu_m\})$ are the core curves of the annuli $P_\rho^\circ \subset \partial K_\rho$, and $m\le 3(k-1)$.   Finally, $M_\rho$ is homeomorphic to the interior of $\cH_k$.
\end{corollary}

\begin{proof}
We are guaranteed the existence of a standard compact core $K_\rho$ from Theorem \ref{thm:tameness}.   Then $\pi_1(K_\rho) \cong F_k$, because $K_\rho$ is homotopy equivalent to $M_\rho$, which has fundamental group isomorphic to $\g=F_k$.  So the characteristic compression body neighborhood of $\partial K_\rho$ in $K_\rho$ is a handlebody of genus $k$, which has only one boundary component; hence, $K_\rho$ is a handlebody of genus $k$.  We can thus find $f: \cH_k\to K_\rho$ inducing $\rho$ on fundamental groups.  There are at most $3(k-1)$ homotopically essential, distinct, simple closed curves in a surface  of genus $k$ bounding a closed handlebody of genus $k$.  The identification of $f(\{\nu_1, .., \nu_m\})$ with the core curves of $P_\rho^\circ$ is immediate from Theorem \ref{thm:tameness}, as is the statement that $M_\rho$ is homeomorphic to the interior of $\cH_k$.
\end{proof}
\begin{remark}\label{rmk:not_unique}
The homeomorphism $f$ is not unique, nor are the curves $\nu_1, ..., \nu_m$, even up to isotopy.  However, given any two homeomorphisms $f: \cH_k \to K_\rho$ and $f': \cH_k\to K_\rho$ inducing $\rho$ on fundamental groups, there is a homeomorphism $\phi: \partial \cH_k\to \partial \cH_k$ that extends to a homeomorphism $\Phi: \cH_k\to \cH_k$ such that $f\circ \Phi$ is homotopic to $f'$.  If $\nu_1', ..., \nu_m'$ are the curves corresponding to  $f':\cH_k\to M_\rho$ from Corollary \ref{cor:handlebody}, then $\phi(\{\nu_i\})$ is isotopic in $\partial \cH_k$ to $\{\nu_i'\}$.  See \cite[Section 4.2]{Namazi--Souto:density}.
\end{remark}

Given a component $R$ of $\partial K_\rho\setminus P_\rho^\circ$, let $\Mod_0(R, K_\rho\setminus P_\rho^\circ )\le \Mod(R)$ be the group of homotopy classes of orientation preserving self homeomorphisms of $R$ which extend to homeomorphisms of $K_\rho$ homotopic to the identity on $K_\rho$.  If $R$ is incompressible, then $\Mod_0(R, K_\rho\setminus P_\rho^\circ )$ is trivial.  For a component $S$ of $\partial K_\rho$, define $\Mod_0(S,K_\rho)$ similarly.  

 \subsection{Structure of ends and invariants}
Let $\rho: \g \to \G$ be a discrete and faithful representation of a finitely generated group $\g$ without torsion, as above.  The \emph{limit set} $\Lambda_\rho\subset \partial \bH^3$ is the set of accumulation points of $(\im\rho).x$ for some (any) $x\in\partial \bH^3$.  The complement $\Omega_\rho = \partial \bH^3\setminus \Lambda_\rho$ is called the \emph{domain of discontinuity}, and $\im\rho$ acts properly discontinuously as a group of conformal automorphisms of $\partial \bH^3$ preserving $\Omega_\rho$. Then $\overline M_\rho = \bH^3\cup \Omega_\rho/ \im \rho$ is a $3$-manifold with boundary $\Omega_{\rho}/\im\rho$ that is a disjoint union of Riemann surfaces with finite hyperbolic area. 
The \emph{convex core} $\mathcal {CC}(M_\rho)\subset M_\rho$ is the quotient of the convex hull of $\Lambda_\rho$ by $\im\rho$.    

Let $K_\rho$ be a standard relative compact core for $M_\rho$.  Say that  a relative end $[E_R]\in \cE_\rho$  is {\it geometrically finite} if $E_R$ meets $\cC\cC(M_\rho)$ in a set of finite volume; we also say that the corresponding boundary component $R\subset \partial K_\rho\setminus P_\rho^\circ$ is geometrically finite.    Call $[E_R]$ (or $R$) {\it geometrically infinite} otherwise.  If $R$ is geometrically finite, then the complement $E_R\setminus \cC\cC(M_\rho)$ has \emph{flaring geometry}; the metric on $E_R\setminus \cC\cC(M_\rho)\cong R\times (0, \infty)$ is isotopic to a metric that is bi-Lipschitz equivalent to the metric $\cosh^2 (t) dx^2 +dt^2$, where $dx^2$ is the intrinsic path metric on the component of $\partial \cC\cC(M_\rho)$ corresponding to $R$.

Suppose $R$ is geometrically finite; then the inclusion $R\hookrightarrow E_R$ is isotopic, through level surfaces of the product $E_R$, into $\partial \overline M_\rho$ and defines a point in a quotient of the Teichm\"uller space $\nu(R) \in \teich (R)/\Mod_0(R, K_\rho\setminus P_\rho^\circ )$.  The equivalence class $\nu(R)$ is the \emph{end invariant} of $[E_R]$.  

We would like to give a description of geometrically infinite ends; to motivate Definition \ref{def:simply degenerate}, we start with an example.  Let $S$ be a closed, oriented surface of negative Euler characteristic, and $\varphi: S\to S$ be a pseudo-Anosov homeomorphism.  The equivalence relation generated by $(x, t) \sim(\varphi(x), t+1)$ defines a normal covering projection $\pi: S\times \bR \to N_\varphi$ onto the mapping torus $N_\varphi$ of $\varphi$ with infinite cyclic deck group.  Thurston's Hyperbolization Theorem \cite[Theorem 0.2]{Thurston:vol} states that the mapping torus $N_\varphi$ has a complete hyperbolic metric; thus, so does the Riemannian covering space $\pi: S\times \bR\to N_\varphi$, which gives rise to a discrete and faithful representation $\rho: \pi_1(S) \to \G$.  We assume that the Riemannian covering $\pi: M_\rho\to N_\varphi$ is a local isometry. Let $\gamma_0$ be a simple closed curve in $S\times \{0\}\subset M_\rho$.   Then $\pi(\gamma_0)$ has a unique geodesic representative $\pi(\gamma_0)^*$ in its homotopy class. Each component of the preimage $\pi\inverse(\pi(\gamma_0)^*) = \{\gamma_i^*\}_{i\in \bZ}$ corresponds to a translate of the geodesic representative $\gamma_0^*$ of $\gamma_0$ in $M_\rho$ by an element of the deck group of $\pi$.  Since the action of a generator of the deck group induces $\varphi$ on level surfaces, the curve $\gamma_i^*$ is homotopic to $\gamma_i = \varphi^i(\gamma_0)\subset S\times \{0\}$.  In particular, the two ends of $M_\rho$ in this example are geometrically infinite, because $\cC\cC(M_\rho)$ contains all closed geodesics, and $\{\gamma_i^*\}_{i\in \bZ}\subset \cC\cC(M_\rho)$ is not a compact set, but $\{\gamma_i^*\}$ exits both ends of $M_\rho$.  

Although we will not discuss measured geodesic laminations in detail, we note also that the curves $\{\gamma_i\}_{i>0}$ limit, as projective measured laminations, to the projective class of a measured geodesic lamination that is fixed by $\varphi$, and similarly in the opposite direction. 
\begin{definition}\label{def:simply degenerate}
With notation as above, a relative end $E_R$ of $M_\rho$ is called \emph{simply degenerate} or  \emph{degenerate} if there is a sequence of closed geodesics $\{\gamma_i^*\}_{i \in \bN}\subset E_R$ that exit compact subsets of $E_R$ and which are homotopic  in $E_R$ to {\it simple} curves $\gamma_i\subset R$.  
\end{definition}
By \cite[Theorem 7.2]{CG:tameness}, if $R$ is not geometrically finite, then  $E_R$ is \emph{simply degenerate}.  Given a Riemannian metric $g$ of finite area and pinched negative curvature, a  \emph{geodesic lamination} $\lambda\subset R$ is a closed set foliated by complete $g$-geodesics.   Bonahon \cite{Bonahon:ends} showed that if $\{\gamma_i\}$ and $\{\gamma_i'\}$ are any two sequences of closed geodesics in $M_\rho$ exiting $E_R$, then the \emph{geometric intersection} 
\[i\left(\frac{\gamma_i}{\ell_g(\gamma_i)}, \frac{\gamma_i'}{\ell_g(\gamma_i')}\right)\to 0, \text{ as } i\to \infty.\]
 Here, $\ell_g(\gamma)$ denotes the length on $(R,g)$ of the unique closed geodesic in the homotopy class of $\gamma$.  By Thurston's theory of \emph{measured geodesic laminations} \cite[Chapter 8]{Thurston:notes},  there is a geodesic lamination $\lambda_R$ such that $\gamma_i/\ell_g(\gamma_i)$ converges in measure to a measured geodesic lamination whose support is $\lambda_R$.  It is known (\cite{Thurston:notes, Bonahon:ends, Canary:ends}) that $\lambda_R$ is compactly supported, \emph{minimal}, and  \emph{filling}, i.e. the complement of $\lambda_R$ in $R$ consists of a collection of ideal polygons with geodesic boundary and once punctured polygons, and no leaf of $\lambda_R$ is isolated.  A  compactly supported, minimal and filling geodesic lamination is called an \emph{ending lamination} for $R$, and the set of ending laminations, given a topology by \emph{unmeasuring}, defines a space $\cEL(R)$ of ending laminations supported by $R$.  See \cite{Klarreich:boundary} to see why  the topology on $\cEL(R)$ given by unmeasuring is the natural one, from the perspective of Kleinian groups. See \cite{Gabai:EL} for a precise definition of the topology of $\cEL(R)$ and an investigation of its connectivity properties.  
  For any two negatively curved Riemannian structures on $R$, the spaces of geodesic ending laminations are canonically homeomorphic.  The ending lamination $\lambda_R\in \cEL(R)/\Mod_0(R, K_\rho\setminus P_\rho^\circ )$ is the \emph{end invariant} $\nu(R)$ of the degenerate end $[E_R]$.

The following theorem of Thurston \cite{Thurston:notes} and Canary \cite[Corollary B]{Canary:covering} states that the only way that `new' geometrically infinite relative ends can appear in the total space of a Riemannian covering of complete hyperbolic $3$-manifolds of infinite volume are as finite covers of `old' geometrically infinite relative ends.  All other ends are geometrically finite.
\begin{theorem}[The Covering Theorem]\label{thm:covering}
Let $\rho: \g\to \G$ be a torsion free finitely generated marked Kleinian group such that $M_\rho$ has infinite volume, and let $i: \hat \g \to \g$ be inclusion of a finitely generated subgroup.  Then either
\begin{enumerate}[(a)]
\item $M_{\rho\circ i}$ is geometrically finite, i.e. all relative ends of $M_{\rho\circ i}$ are geometrically finite, or
\item For every simply degenerate end of $M_{\rho\circ i}^\circ$, there is a neighborhood $E_{\hat R}$ of that end and a neighborhood $E_R$ of a simply degenerate end of  $M_\rho^\circ$ such that the covering projection $M_{\rho\circ i}\to M_\rho$ restricts to a finite sheeted covering $ E_{\hat R} \to E_R$. 
\end{enumerate}
\end{theorem}
We may assume that $E_{\hat R}$ and $E_R$ are both products so that the covering $E_{\hat R} \to E_R$ induces a finite sheeted covering $\hat R\to R$ of surfaces.

\subsection{Simplicial Hyperbolic Surfaces}\label{sec:simplicial_hyperbolic}
One way to study the geometry of simply degenerate ends of hyperbolic $3$-manifolds is to `probe' them with negatively curved surfaces.
A \emph {triangulation} $\cT$ of a closed surface $S$, is a 3-tuple $\cT = (V, A, T)$;  $V$ is a finite set of vertices, $A$ is a maximal simple arc system, and $T$ is a union of oriented $2$-simplices with embedded interiors in $S$ that are compatible with $V$ and $A$; the sum of simplices in $T$ is required to represent the fundamental class of $S$.  With this definition, a triangulation $\cT$ is not necessarily associated to the geometric realization of a simplicial complex, e.g. $\cT$ may only have the structure of a $\Delta$-complex.

The following definitions are essentially due to Bonahon \cite{Bonahon:ends}; we make some modifications to allow for the possibility that a relative compact core for $M_\rho$ has compressible boundary components.  Namely, we will be interested in the case that $\rho: F_k\to \G$ is discrete and faithful and without parabolics, so that $M_\rho$ is the interior of a genus $k$ handlebody, by Corollary \ref{cor:handlebody}.  There is only one component $S$ of the boundary of a standard relative compact core  for $M_\rho$.  In what follows, we encourage the reader to keep in mind this example with  $E=E_S$, a product neighborhood of the end of $M_\rho$.   See also \cite[Section 4]{Canary:covering}.

Let $E\subset M_\rho$ be a codimension $0$ submanifold homeomorphic to a product $S\times \bR$ where $S$ is a closed oriented surface of genus at least $2$.  A \emph {simplicial pre-hyperbolic surface} is a pair $(f,\cT)$ where $f :S\to E$ is a $\pi_1$-injective continuous map; $\cT$ is a triangulation of $S$ so that for each $\alpha\in A$, ${\im f\circ \alpha \subset E}$ is an $M_\rho$-local geodesic segment, and for each $2$-simplex $\tau$ of $\cT$, $\im  f\circ \tau\subset E$ is $M_\rho$-geodesically immersed.  
A simplicial pre-hyperbolic surface is a \emph{simplicial hyperbolic surface} if the cone angle about each vertex in the intrinsic metric $g_f$ on $S$ induced by $f$ is at least $2\pi$.  By a lemma of Ahlfors \cite{Ahlfors:lemma}, there is a unique hyperbolic metric $g$ on $S$ in the conformal class of $g_f$ such that the identity $(S,g)\to (S, g_f)$ is $1$-Lipschitz.  Thus, geometric bounds on hyperbolic surfaces translate to geometric bounds for simplicial hyperbolic surfaces, further giving us geometric estimates on $E$, since the composition mapping $S\to E\to M_\rho$ is $1$-Lipschitz. 

From a pre-simplicial hyperbolic surface $(f: S\to E, \cT)$ with one vertex, we can often construct a simplicial hyperbolic surface $(g: S\to E, \cT)$ homotopic within $E$ to $f$ as follows. Since $\cT$ has only one vertex $v$, every arc $\alpha\in A$ maps to a locally geodesic loop based at $f(v)$; the image is not necessarily smooth at $f(v)$. Suppose there is an arc $\alpha \in A$ such that the geodesic representative $f(\alpha)^*$ of $f(\alpha)$ in $M_\rho$ is contained in $E$, as is (the projection of) the straight line homopty between (appropriate lifts of) $f(\alpha)$ and $f(\alpha)^*$.  There is a new map $g: S\to E$ homotopic to  $f$ obtained by `dragging' all arcs along the image of $v$ under the straight line homotopy between $f(\alpha)$ and $f(\alpha)^*$ and re-straightening all $1$- and $2$-simplices in the image relative to $g(v)$.  As long as the straight line homotopy between $f$ and $g$ is contained in $E$, then $(g,\cT)$ is a simplicial hyperbolic surface.  Indeed, we now just need to check that the cone angle about $v$ is at least $2\pi$.  But this follows from the fact that $g(v)$ lies on a smooth geodesic subsegment of $g(\alpha)=f(\alpha)^*$; the intersection of $g(S)$ with a small sphere about $g(v)$ is a piecewise spherical geodesic loop passing through antipodal points, hence the cone angle in the metric induced by $g$ about $v$ is at least $2\pi$ \cite[Lemma 4.2]{Canary:covering}.

We will make use of the following result for simplicial hyperbolic surfaces, due to Bonahon \cite{Bonahon:ends}.
\begin{lemma}[Bounded Diameter Lemma]\label{lem:bounded_diameter}
For any compact set $K\subset M_\rho$, there is a compact set $K'\subset M_\rho$ such that if $f: S\to E\subset M_\rho$ is a simplicial hyperbolic surface and $\im f \cap K\not=\emptyset$, then $\im f\subset K'$.  The diameter of $f(S)$ is bounded above by a constant that depends only on the topology of $S$ and the injectivity radius of $M_\rho$ in $K'$.
\end{lemma}

Suppose $\rho: F_k\to \G$ is discrete and faithful without parabolics and $M_\rho$ has a geometrically infinite end $E_S$, where $S$ is closed. Then using also our characterization of geometrically infinite ends of a hyperbolic $3$-manifold as those that have a sequence $\{\gamma_i^*\}$ of  closed geodesics exiting the end $E_S$ of $M_\rho$ homotopic to simple curves in $S\times \{0\}$, we can construct an infinite sequence of `well spaced' simplicial hyperbolic surfaces exiting $E_S$.  
Choose points $v_i\in \gamma_i^*$, and let $A_i$ be a maximal system of based loops containing $\gamma_i\subset S\times \{0\}$, cutting $S$ into triangles. For large enough $i$, the straight line homotopy between $A_i$ and the corresponding locally geodesic loops based at $v_i$ is contained in an open neighborhood of $E_S$ \cite{Canary:ends}, thus we can `hang' simplicial hyperbolic surfaces $f_i : S\to E_S$ from the geodesics $\gamma_i^*$, as above.  Since  $\{\gamma_i^*\}$ exit all compact subsets of $E_S$ to $[E_S]$, by Lemma \ref{lem:bounded_diameter}, we can pass to a subsequence so that $\im f_i$ has empty intersection with $\im f_j$, if $i\not= j$.  

\subsection{The Ending Lamination Theorem}
We continue with our notation from before; $\rho: \g \to \G$ is a torsion free finitely generated marked Kleinian group, and $K_\rho$ is a standard relative compact core for $M_\rho$.  We summarize how to collect the end invariants $\nu$.  Let $S_1, ..., S_k$ be the components of $\partial K_\rho$ with negative Euler characteristic; they are closed surfaces that inherit the boundary orientation from $K_\rho$.  The annular components of $P_\rho^\circ$ have core curves that are identified with homotopically distinct, essential simple closed curves $\nu(P_\rho^\circ)\subset \sqcup S_i$.  Then for each component $R\subset \partial K_\rho\setminus P_\rho^\circ$ we record a piece of data; if $R$ is geometrically finite, then $\nu(R)\in \teich (R)/\Mod_0(R, K_\rho\setminus P_\rho^\circ)$ is an equivalence class of conformal structure at infinity.  If $R$ is geometrically infinite, then $\nu(R)\in \cEL(R)/\Mod_0(R,K_\rho\setminus P_\rho^\circ)$ is the corresponding equivalence class of ending lamination. To $\rho$, we associate all of these data $\nu(\rho)$.  

 The motto of the Ending Lamination Theorem is, ``the topology and geometry at infinity determine the metric.''  The Ending Lamination Theorem is a classification theorem for finitely generated Kleinian groups; it helps answer many of the questions about the behavior of hyperbolic $3$-manifolds, their deformations spaces, and  limiting behavior.
\begin{theorem}[The Ending Lamination Theorem]\label{thm:ELT}
Let $\g$ be a finitely generated non-abelian group without torsion, let  $\rho, \rho': \g \to \G$ be marked Kleinian groups, and let $K_{\rho}\subset M_{\rho}^\circ$ and $K_{\rho'}\subset M_{\rho'}^\circ$ be standard relative compact cores.  

If there is a homeomorphism $\phi: K_\rho\setminus P_\rho^\circ\to K_{\rho'}\setminus P_{\rho'}^\circ$ such that $\phi_* = \rho'\circ {\rho}\inverse$ and $\phi(\nu(\rho)) = \nu(\rho')$, then $\phi$ extends to a homeomorphism $\Phi: M_\rho\to M_{\rho'}$ that is isotopic to an isometry inducing ${\rho'}\circ \rho\inverse$ on fundamental groups.  Equivalently, there is a $g\in \G$ such that $g\rho g\inverse = \rho'$.

In other words, if the topological type of $M_\rho$ and $M_{\rho'}$ agree and so do the relative end invariants, then $M_\rho$ is isometric to $M_{\rho'}$ in the homotopy class of the classifying map determined by ${\rho'}\circ \rho\inverse$.  
\end{theorem}
We owe the reader attributions and references.   The final ingredients to prove Theorem \ref{thm:ELT} were given by Brock--Canary--Minsky \cite[Ending Lamination Theorem for Incompressible Ends]{BCM:ELTII}, and in a follow up paper \cite{BCM:ELTIII} that carries out the details needed to  modify the proof when $M_\rho$ has compressible ends, using Canary's Branched Cover Trick \cite{Canary:ends}, as outlined in \cite[Section 1.2]{BCM:ELTII}.   For Kleinian surface groups, Minsky built a \emph{model manifold} $M_\nu$ from a list of end invariants $\nu$ out of \emph{blocks} and \emph{tubes}.  He also constructed a Lipschitz map (with Lipschitz constant depending only on the topology of the surface)  $M_\nu\to M$, where $M$ is a hyperbolic manifold with end invariants $\nu$ \cite[Extended Model Theorem]{Minsky:ELTI}.  The geometry and arrangement of the tubes in the model manifold $M_\nu$ is extracted from a \emph{hierarchy of tight geodesics} and careful analysis of the geometry of Harvey's \emph{complex of curves} carried out by Masur--Minsky in \cite{MM:complexI, MM:complexII}.  The Lipschitz model map was then promoted to a bi-Lipschitz model map that extends to a conformal map at infinity \cite[Bi-Lipschitz Model Theorem]{BCM:ELTII}.  So, if $M_\rho$ and $M_{\rho'}$ have the homotopy type of a finite type surface (with a parabolicity condition on the boundary curves of the surface) and they have the same end invariants, as seen from some reference surface, then there is a bi-Lipschitz homeomorphism between them, compatible with markings, that extends to a conformal map at infinity; by Sullivan's Rigidity Theorem \cite{Sullivan:rigidity}, the bi-Lipschitz mapping is homotopic to an isometry.  

If all of the relative ends of $M_\rho$ are simply degenerate, we say that $M_\rho$ is \emph{totally degenerate}.  The Ending Lamination Theorem \ref{thm:ELT} implies that the marked isometry type of a totally degenerate manifold $M_\rho$ is completely determined by the topology of a standard relative compact core and list of ending laminations.

The general statement for the Ending Lamination Theorem (including the case that there are compressible relative ends) is made possible by the Tameness Theorem \ref{thm:tameness}.  One studies the covering spaces associated to the relative ends of a given  manifold $M$, builds models of the ends from the end invariants, and assembles the pieces to obtain a bi-Lipschitz model for $M$.  All of this builds on the important contributions of Thurston, Ahlfors, Bers, Marden, Maskit, Sullivan, Bonahon, Otal, O'shika and many others.  
Soma \cite{Soma:ELT} has recently given an alternate strategy using methods from bounded cohomology and volume rigidity.

\subsection{Quasi-isometric classification of marked Kleinian surfaces and free groups}\label{sec:qi_classification}
For constants $M\ge1$ and $A\ge 0$, an $(M,A)$-quasi-isometry $f: X\to Y$ between metric spaces is a map that satisfies
\[\frac1M d_X(x,x')-A\le d_Y(f(x), f(x'))\le M d_X(x, x')+A\] 
for all $x, x'\in X$.  
Thurston defined the \emph{quasi-isometric topology} on the space of hyperbolic manifolds with a given homotopy type in \cite[Section 1]{Thurston:vol}; say that two discrete and faithful representations $\rho_1, \rho_2: \g \to \PSL_2\bC$ of a finitely generated non-elementary group $\g$ are \emph{quasi-isometric}, and write $\rho_1\sim_{q.i}\rho_2$, if there are $M$ and $A$ as above and a  $(\rho_1, \rho_2)$-equivariant $(M,A)$-quasi-isometry $\bH^3\to \bH^3$.  An equivariant quasi-isometry of representations extends to an equivariant quasi-conformal map at infinity, where the quasi-isometry constants give control on the quasi-conformal constant of the map at infinity.

A weaker version of the implication `(1) $\Rightarrow$ (3)' in the following theorem was given by Soma \cite[Theorem A]{Soma:boundedsurfaces}; in that paper, all manifolds are homotopy equivalent to a closed surface, have bounded geometry, and two degenerate ends.  The number $\epsilon$ in Soma's \cite[Theorem A]{Soma:boundedsurfaces} depends on the lower bound $\epsilon'$ for the injectivity radius of the manifolds involved and $\epsilon$ goes to zero with $\epsilon'$.  In particular, the techniques that prove \cite[Theorem A]{Soma:boundedsurfaces} do not apply to manifolds with unbounded geometry.  Part of the novelty of Theorem \ref{old work} is that our techniques are equally amenable to manifolds with unbounded geometry, a topologically generic set at the boundary of the deformation space of hyperbolic metrics, and our constant $\epsilon$ is uniform over all metrics; it only depends on the homotopy type of the manifolds $M_{\rho_i}$ but not their geometry.  

\begin{theorem}[\cite{Farre:bdd, Farre:relations}]\label{old work}
Let $S$ be an orientable hyperbolic surface of finite type.  There is a constant $\epsilon = \epsilon(S)>0$ such that if $\rho_1, \rho_2: \pi_1(S)\to \G$ are discrete and faithful with no parabolic elements, then the following are equivalent.
\begin{enumerate}[$(1)$]
\item $\|[\rho_1^*\vol_3] - [\rho_2^*\vol_3]\|_\infty <\epsilon$
\item $[\rho_1^*\vol_3] = [\rho_2^*\vol_3]$
\item $\rho_1\sim_{q.i.}\rho_2$.
\end{enumerate}
Moreover, if $M_{\rho_1}$ has a geometrically infinite end and $\rho_3: \pi_1(S)\to \G$ is an arbitrary representation satisfying $\|[\rho_1^*\vol_3] - [\rho_3^*\vol_3]\|_\infty <\epsilon$, then $\rho_3$ is faithful.
\end{theorem}
\begin{proof}[Proof of Theorem \ref{old work}]
Let $\epsilon< \min\{\epsilon'/2, v_3\}$, where $\epsilon'$ is as in \cite[Theorem 1.2]{Farre:bdd} and only depends on the topology of $S$.  The last statement is \cite[Theorem 1.3]{Farre:bdd}, which states that $\rho_3$ is faithful. 

Condition (2) implies condition (1), trivially.   
Now we show that (3) implies (2).  We need the classical fact that an equivariant quasi-isometry extends to an equivariant quasi-conformal homeomorphism at infinity (see, e.g. \cite[Corollary 5.9.6]{Thurston:notes}).  This quasi-conformal homeomorphism extends to a {\it volume preserving} bi-Lipschitz diffeomorphism $M_{\rho_1}\to M_{\rho_2}$ inducing $\rho_2\circ \rho_1\inverse$ on fundamental groups; see \cite[Appendix B]{McMullen:book},  \cite[Theorem 5.6]{BB:inflexibility}, or \cite[Theorem 3.1]{Farre:relations}, where those results are summarized.  We can apply \cite[Corollary 3.6]{Farre:relations} or the proof of \cite[Theorem 3.2]{Farre:relations} to see that  $[\rho_1^*\vol_3] = [\rho_2^*\vol_3]$.  

Finally, we show that (1) implies (3).  With $\epsilon$ as in the first paragraph, \cite[Theorem 1.2]{Farre:bdd} states that the geometrically infinite end invariants of $M_{\rho_1}$ must be the same as the geometrically infinite end invariants of $M_{\rho_2}$.   If $M_{\rho_1}$ is totally degenerate, then $\nu({\rho_1})=\nu(\rho_2)$.  By the Ending Lamination Theorem \ref{thm:ELT}, $\rho_1$ is conjugate, hence quasi-isometric, to $\rho_2$.  

So, we now assume that at least one end of $M_{\rho_1}$ is geometrically finite.  If $\pi_1(S)$ is free, then by Corollary \ref{cor:handlebody}, $M_{\rho_1}$ and $M_{\rho_2}$ are  handlebodies of the same genus, and since $\im \rho_i$ has no parabolic elements, there is only one relative end of $M_{\rho_i}$, for $i = 1,2$.  For geometrically finite hyperbolic manifolds with no parabolic cusps, $\cC\cC(M_{\rho_i})$ is a standard compact core for $M_{\rho_i}$;  we may find a bi-Lipschitz homeomorphism $\phi: \interior \cC\cC(M_{\rho_1})\to \interior \cC\cC(M_{\rho_2})$ inducing $\rho_2\circ \rho_1\inverse$ on fundamental groups.  
The neighborhood  $E_i = M_{\rho_i}\setminus \cC\cC(M_{\rho_i})$ of the end of $M_{\rho_i}$ is a product $S'\times (0, \infty)$, where $S'$ is a closed surface of genus equal to the rank of $\pi_1(S)$, and its metric is bi-Lipschitz equivalent to $\cosh^2 (t) dx_i^2 +dt^2$, where $dx_i$ is the induced path metric on $\partial \cC\cC(M_{\rho_i})$.  Thus $\phi$ extends to a homeomorphism mapping level surfaces of $E_1$ to level surfaces of $E_2$ with respect to the aforementioned product structure on $E_i$.   Since the path metrics on $\partial \cC\cC(M_{\rho_i})$ are bi-Lipschitz equivalent to each other and $\cC\cC(M_{\rho_1})$ and $\cC\cC(M_{\rho_2})$ are compact, the map $\phi:M_{\rho_1}\to M_{\rho_2}$ is a bi-Lipschitz homeomorphism that lifts to an equivariant bi-Lipschitz homeomorphism $\bH^3\to \bH^3$.  Thus $\rho_1\sim_{q.i.}\rho_2$.  

If $S$ is closed then $M_{\rho_i}\cong S\times \bR$ by Bonahon's Tameness Theorem \cite[Th\'eor\`em A]{Bonahon:ends} and since $\im\rho_i$ contains no parabolics, standard compact cores for $M_{\rho_i}$ are of the form $S\times [0,1]$, for each $i = 1,2$. There are two possibilities for the geometry of the ends of $M_{\rho_i}$.  If $M_{\rho_1}$ has no geometrically infinite ends, then the same is true for $M_{\rho_2}$.  A similar argument to the previous paragraph produces an equivariant bi-Lipschitz mapping $\bH^3\to \bH^3$, and so $\rho_1\sim_{q.i}\rho_2$.  The last possibility is that $M_{\rho_i}$ each has one degenerate end and one geometrically finite end.  In this case, the convex cores of $M_{\rho_i}$ are neighborhoods of their degenerate ends, and \cite[Bi-Lipschitz Model Theorem]{BCM:ELTII} supplies us with a bi-Lipschitz homeomorphism $\phi: \cC\cC(M_{\rho_1})\to  \cC\cC(M_{\rho_2})$ inducing $\rho_2\circ \rho_1\inverse$ on fundamental groups.  Again, we extend $\phi$ to a bi-Lipschitz homeomorphism $M_{\rho_1} \to M_{\rho_2}$ mapping flaring level surfaces to flaring level surfaces in the complement of the convex core.  The conclusion of the theorem follows, as in the previous cases.
\end{proof}

There are geometrically finite representations of a closed surface $S$ of genus $g\ge 2$ that are not quasi-isometric to each other.  For example, take pants decompositions $\alpha$ and $\beta$ of a closed surface $S$ with no common curves.  There is a unique conjugacy class $\rho: \pi_1(S)\to \G$ such that the quotient manifold $M_\rho$ has $6g-6$ rank-$1$ cusps.  $K_\rho$ is a trivial interval bundle over $S$, with two boundary components $S^+$ and $S^-$.  The components of $P_\rho^\circ$ are annular subsurfaces of $S^+$ and $S^-$ with core curves $\alpha\subset S^+$ and $\beta\subset S^-$; each complementary component of $\partial K_\rho\setminus P_\rho^\circ$ is  homeomorphic to a $3$-times punctured sphere.  The only simple closed curves on three times punctured spheres are parallel to the punctures and do not have geodesic representatives in a finite area metric of non-positive curvature, so there are no ending laminations, being limits of geodesic simple closed curves,  on a three times punctured sphere.  Every relative end of this example is geometrically finite.  In fact, since the Teichm\"uller space of a three times punctured sphere is a point,  every relative end is isometric to every other relative end in this example.  The volume class vanishes by Theorem \ref{soma's theorem}, but $\rho$ is not quasi-isometric to any Fuchsian representation.  Thus, the assumption that representations have no parabolic elements cannot be dropped in  Theorem \ref{old work}.

  Now we prove Theorem \ref{thm:intro_old_work}, assuming Theorem \ref{thm:intro_main}, which is the last ingredient for our quasi-isometric volume rigidity result.
\begin{theorem}\label{thm:rigidity} 
There exists a constant $\epsilon = \epsilon(S)$ such that the following holds.  Suppose that $\rho_0: \pi_1(S)\to \G$ is a discrete and faithful representation without parabolic elements, and that $[\rho_{0}^{*}\vol_{3}]\not=0$.  If $\rho: \pi_1(S) \to \G$ is any other representation without parabolics satisfying \[\|[\rho_0^*\vol_3]- [\rho^*\vol_3]\|_\infty<\epsilon,\]
then $\rho$ is discrete and faithful, and $\rho$ is quasi-isometric to $\rho_0$.  
If $\rho_0$ is totally degenerate, then $\rho_0$ and $\rho$ are conjugate in $\G$.  
\end{theorem}
\begin{proof}
Choose $\epsilon$ as in Theorem \ref{old work}. 
Then $\rho_0$ is discrete, faithful, and $[\rho_0^*\vol_3]\not= 0$, so we can apply Theorem \ref{soma's theorem} to see that $\|[\rho_0^*\vol_3]\|_\infty=v_3$, and $M_{\rho_0}$ has a geometrically infinite end.
We can apply the last statement of Theorem \ref{old work} to deduce that $\rho$ is faithful.

 By Theorem \ref{thm:intro_main}, since $\rho_0$ is discrete and $\|[\rho_0^*\vol_3]- [\rho^*\vol_3]\|_\infty<\epsilon<v_3$, $\rho$ cannot have dense image.   
The triangle inequality gives $\|[\rho^*\vol_3]\|> v_3-\epsilon> 0$.  
Thus $[\rho^*\vol_3]\not=0$ and so $\rho$ has discrete image by Lemma \ref{lem:geom_elem}.
Now that we know that $\rho$ is discrete and faithful,  we can apply the main body of Theorem \ref{old work} to see that $\rho\sim_{q.i.}\rho_0$.  This concludes the proof of the first statement of the theorem.   

If $M_{\rho_0}$ is totally degenerate, then all of the end invariants of $M_\rho$ and $M_{\rho_0}$ coincide.  By the Ending Lamination Theorem \ref{thm:ELT}, $\rho_0$ is conjugate to $\rho$ in $\G$, as in the proof of Theorem \ref{old work}.  
\end{proof}

\section{Volume classes of dense representations}\label{volume classes}

We will prove our main theorems for the $3$-dimensional volume class in this section by building efficient chains from a geometrically infinite Kleinian free group $\g = \langle a, b\rangle\le \G$ that contains no parabolics.  Then we will approximate the shape of those chains using a dense representation.  Call a $n$-chain $Z$ \emph{$\epsilon$-efficient} or just \emph{efficient} if $\frac{|\vol_n(Z)|}{\|Z\|_1}>\epsilon$.  We recall Soma's construction of efficient chains in a hyperbolic manifold $\bH^3/\Gamma$, where $\Gamma$ is any finitely generated, torsion free, infinite co-volume, geometrically infinite Kleinian group $\Gamma\le \G$.  Then we use tools from Section \ref{isometric chain maps} to turn these efficient chains into efficient $3$-chains with only one vertex, which then define chains on $\g$.  If we have a dense representation $\rho: G \to \G$, we can approximate $a$ and $b$ by sequences $a_n = \rho(x_n)$ and $b_n = \rho(y_n)$ for suitably chosen $x_n, y_n \in G$.  The shape of the chain on $\g$ with large volume and small boundary area can be approximated by chains on the groups $\langle a_n ,b_n\rangle$.  The chains in the approximates have essentially the same volume as the chains which came from the manifold $\bH^3/\g$, as in Section \ref{approximation section}.  

We have a natural way to collapse a straightened singular $3$-chain in $\bH^3/\Gamma$ with many vertices onto a chain with only $1$ vertex.  We lose some volume during this collapsing process, but the loss of volume is controlled in terms of the $\ell_1$-norm of the boundary of the original chain, which is uniformly bounded.  We need to prove a number of technical facts about these one-vertex chains that rely on geometrical facts concerning hyperbolic $3$-manifolds.  
This allows us to show that the volume class of any dense representation has maximal semi-norm.  

Then we prove some auxiliary results about certain subgroups of Kleinian groups to show that certain volume classes are separated in semi-norm, as announced in Theorem \ref{thm:intro_main} and Theorem \ref{thm:norm}.

\subsection{Constructing efficient chains}\label{sec:efficient_chains}

 There are several chain complexes in which a $3$-chain $Z$ could live, e.g. in a group $\g$ or a quotient manifold $\bH^3/\g$, and $\vol_3$ should be interpreted in whichever context it makes sense.  In what follows, we show that for chains $Z$ whose boundary has small $\ell_1$-norm, it is not so important where we compute the volume; this is most of the content of Proposition \ref{group chains}.

In \cite{Farre:bdd}, we constructed sequences of $3$-chains on geometrically infinite genus $g$ handlebodies with large volume and small boundary area.  These chains were $\epsilon_g$-efficient, where $\epsilon_g>0$ depends only on the topology of the handlebody (but not on its geometry), and for which the boundary surfaces of the chains were well controlled.  In fact, they were simplicial hyperbolic surfaces; see Section \ref{sec:simplicial_hyperbolic}.  Soma \cite{Soma:Kleinian,Soma:boundedsurfaces} constructed $(v_3-\epsilon)$-efficient chains, for any $\epsilon>0$, but the boundaries of his chains are not well controlled and grow wilder as the efficiency constant gets closer to $v_3$.  

We will revisit Soma's construction of efficient chains for a geometrically infinite manifold with free fundamental group to extract a technical feature that we require  in the proof of Lemma \ref{Borel group chains}.  Soma's chains are constructed via `smearing,' as in Thurston's construction of efficient (measure) cycles representing the fundamental class of a closed hyperbolic manifold \cite[Chapter 6]{Thurston:notes}.

We will follow \cite[Section 3]{Soma:boundedsurfaces}; note  that although Soma requires that $\Gamma$ be isomorphic to a closed surface group, \cite[Lemma 3.2]{Soma:boundedsurfaces} only depends on the topological and geometrical structure of a geometrically infinite end of a topologically tame hyperbolic $3$-manifold.  See the proof of \cite[Theorem 1]{Soma:Kleinian} for further comments.  

 For the rest of this section, fix a discrete, faithful, and geometrically infinite representation $\rho_0: F_d\to \G$ with no parabolic elements, and let $\g_0$ denote $\im \rho_0$.  Let $\pi: \bH^3 \to \bH^3/\g_0 = M_{\rho_0}$ be the covering projection. It will be convenient to assume that $M_{\rho_0} = \bH^3/\g_0$ has bounded geometry\footnote{Generically, a simply degenerate end of a complete hyperbolic $3$-manifold homeomorphic to a handlebody without parabolics has a sequence of closed geodesics with length tending to $0$.  However, bounded geometry geometrically infinite handlebodies arise, for example, as geometric limits of closed hyperbolic $3$-manifolds with bounded Heegaard genus and gluing data represented by powers of a pseudo-Anosov homeomorphism going to infinity; see \cite{NSouto}.}.  By Corollary \ref{cor:handlebody}, $M_{\rho_0}$ is homeomorphic to a handlebody $\cH_d$ of genus $d$.  There is a standard compact core $K_{\rho_0}$ that is also a handlebody, and $E = M_{\rho_0}\setminus K_{\rho_0}$ is a neighborhood of the end of $M_{\rho_0}$ that is simply degenerate; its closure is homeomorphic to $S\times [0, \infty)$, where  $S$ is an oriented surface of genus $d$.

 Find a sequence $\{(f_n:S \to E, \mathcal T_n)\}$ of  simplicial hyperbolic surfaces, homotopic to the inclusion $S\hookrightarrow S\times \{0\}$, with $X_n=\im(f_n) $ exiting $E$ toward $[E]$.  The maps $f_n$ are not embeddings, but $E\setminus X_n$ consists of some compact components and a non-compact component $E_n\in [E]$. Pass to a subsequence such that $\{X_n\}$ are pairwise disjoint (see the discussion following the Bounded Diameter Lemma \ref{lem:bounded_diameter}) and $E_m\supset E_n$ if $m<n$.   Also let $L_{m,n}$ denote the closure of $E_m\setminus E_n$, which is a compact set.  
 
Let $\sigma: \Delta_3\to \bH^3$ be a non-degenerate straight simplex; abusing notation we ignore the parameterization of the map and identify $\sigma$ with its image in $\bH^3$ or even its ordered vertex set.  
In each straight simplex, there exists a unique inscribed ball, meeting each face in a point.
Let $\cen(\sigma)\in \bH^3$ denote the center of this inscribed ball.   
We describe a Borel measure $\smear\sigma$ on the space of locally straight $3$-simplices 
\[\mathscr S(M_{\rho_0}) = (\bH^3)^4/\g_0\] 
in $M_{\rho_0}$.   If $\sigma $ is a straight simplex in $\bH^3$, then $\pi(\sigma)$ denotes the corresponding element of $\mathscr S(M_{\rho_0})$.   

Consider the Haar measure $\mu$ on $\G$ normalized so that, for any $x\in \bH^3$ and any Borel measurable set $K\subset \bH^3$, 
\[\mu(\{g\in \G: g.x \in K \subset \bH^3 \})=\Vol(K),\] 
the hyperbolic volume of $K$.  Then $\mu$ descends to a measure (with the same name) on $\g_0\backslash\G$. 
Finally, we define $\smear \sigma$ as follows: 
for a Borel measurable set $K\subset \mathscr S(M_{\rho_0})$, we define
\[\smear\sigma(K) = \mu (\{\bar g \in \g_0\backslash\G: \pi(g.\sigma) \in K\}).\]
 
Also to $K\subset M_{\rho_0}$ we associate the set of simplices
\[\mathscr S_\sigma(K) = \{\pi(g.\sigma) \in \mathscr S(M_{\rho_0}): \pi(\cen(g.\sigma)) \in K\},\]
 so that \[\smear\sigma (\mathscr S_\sigma(K))= \Vol(K).\]  

Fix a geodesic plane $\bH^2 \subset \bH^3$ and let $r: \bH^3\to \bH^3$ be reflection in $\bH^2$, so that the simplex $r\circ \sigma$ is isometric to $\sigma$ with the opposite orientation.  Define \[z(\sigma) = \frac12(\smear \sigma -\smear r\circ \sigma). \]
By restricting the support of  $z(\sigma)$ to $\mathscr S_\sigma(L_{0,k})$, we obtain a family of Borel measures $z_k(\sigma)$ with total variation $\|z_k(\sigma)\| = \Vol(L_{0,k})$ and 
\begin{equation}\label{eqn:volume_computation}
\int_{\mathscr S(M_{\rho_0})}\vol_3(\sigma') ~dz_k(\sigma)(\sigma') = \vol_3(\sigma)\Vol(L_{0,k}).
\end{equation}
  
  In Thurston's smooth measure homology theory, each $z_k(\sigma)$ defines a \emph{measure chain}, i.e. a signed measure on the space of smooth singular simplices satisfying a local finiteness condition.  A straight singular chain $Z = \sum a_i\sigma_i\in \textnormal C_k^{\str}(M_{\rho_0})$ defines a smooth measure chain $\sum a_i\delta_{\sigma_i}$, where $\delta_{\tau}$ is the Dirac measure supported on $\tau$.  This correspondence induces a continuous chain map with respect to the topology induced by the $\ell_1$- and the total variation norms, respectively.  In fact, singular homology and smooth measure homology are isometrically isomorphic with respect to the $\ell_1$- and total variation semi-norms \cite[Theorem 1.2]{Loh:measure_homology}. 

\begin{lemma}[{\cite[Lemma 3.2]{Soma:boundedsurfaces}}]\label{manifold chains}
Given $d$, there exists a constant $K_d>0$ depending only the topology of the surface $S_d \cong \partial K_{\rho_0}$, such that for every $\epsilon>0$, there is a sequence $V_k\in \textnormal{C}_3^{\str} (M_{\rho_0};\bR)$ such that the following properties hold:
\begin{enumerate}[(a)]
\item $\displaystyle \frac{|\vol_3(V_k)|}{\|V_k\|_1}>v_3-\epsilon$ for all $k$, \label{a}
\item $\|\partial V_k\|_1 \le K_d$, for all $k$, and \label{b}
\item $\displaystyle |\vol_3(V_k)|\to \infty$. \label{c}
\end{enumerate}
\end{lemma}
\begin{proof}[Sketch of proof]  We consider the regular simplex $\sigma_t$ with edge lengths $t>\!\!>0$ and the Borel measures $\{z_k(\sigma_t)\}_{k=1}^\infty$ on $\mathscr S (M_{\rho_0})$.  We know that $|\vol_3(\sigma_t) - v_3|$ decreases exponentially in $t$ \cite[Theorem 6.4.1]{Thurston:notes}, but we only need the fact that for $t$ large enough, $\vol_3(\sigma_t) > v_3 - \epsilon/2$.   

For the appropriate definition of the boundary of a measure chain, we have $\|\partial z_k(\sigma_t)\| \le K_d$, where $K_d$ only depends on the topology of the surface $S_d\cong \partial K_{\rho_0}$.  
Indeed, consider a simplex $\pi(g.\sigma_t)$ with $\cen (\pi(g.\sigma_t)) \in L_{0,k}$ but far from $\partial L_{0, k}$.  
Let $\tau$ be a face of $\sigma_t$ and let $r_\tau$ be reflection through the plane containing $\tau$; $r_\tau$ is conjugate to $r$ in $\G$.
A computation shows that 
\begin{equation}\label{center close}
d_{\bH^3}(\cen(g.r_\tau\circ \sigma_t),\cen(g.\sigma_t))<2.
\end{equation}
 Thus $\cen (\pi(g.r_\tau\circ \sigma_t))$ is far from $\partial L_{0,k}$.   The faces corresponding to $\tau$ coming from $\pi(g.\sigma_t)$ and the reflected copy match and  cancel after applying $\partial$ to $z_k(\sigma_t)$.  If $\cen (\pi (g.\sigma_t))$ is close to $\partial L_{0,k}$,  then $\pi(g.r_\tau\circ \sigma_t)$ may not be in the support of $z_k(\sigma_t)$.  However, $\pi(g.r_\tau\circ \sigma_t)\in \mathscr S_{r\circ \sigma_t}(\cN_2(\partial L_{0,k}))$.  Thus $\supp(\partial z_k(\sigma_t))$ is contained in the set of locally straight triangles which are faces of tetrahedra with centers in $\cN_2(\partial L_{0,k})$; we conclude that $\|\partial z_k(\sigma_t)\|\le \Vol(\cN_2(\partial L_{0,k}))$ \cite[Proof of Lemma 3.2]{Soma:boundedsurfaces}. Using the fact that the induced metric on $X_k$ has curvature everywhere at most $-1$, we can find  a universal constant $V>0$ such that $\Vol(\cN_2(\partial L_{0,k}))\le V\cdot |\chi(S_d)| = K_d$ (see, for example,  \cite[Proof of Theorem 1]{Soma:Kleinian}).

To each measure chain $z_k(\sigma_t)$ we will associate a straight  $3$-chain $V^t_{k}\in \textnormal C_3^{\str}(M_{\rho_0};\bR)$.  For $t$ large enough, every $1/t$-ball in $M_{\rho_0}$ is embedded, because we have assumed that $M_{\rho_0}$ has bounded geometry.  Find a maximal $(1/t)$-separated collection of points $\{p_i^t\}_{i=1}^\infty\subset M_{\rho_0}$, and let $\{U_i^t\}$ be the Voronoi cells generated by $\{p_i^t\}$.  The boundary of the closure of each cell has zero $3$-dimensional Lebesgue measure, and each cell is connected,  simply connected, precompact, and has a distinguished point $p_i^t\in U_i^t$.  For each $k$, there is a finite subset of  $\{U_i^t\} $ that meet any simplex in the  support of $z_k(\sigma_t)$, because $L_{0,k}$ is compact, so that for each $\pi(g.\sigma_t)$ with center in $L_{0,k}$, we have $\pi(g.\sigma_t)\subset \cN_{t}(L_{0,k})$.   So, only finitely many terms in the sum \begin{equation}\label{eqn:correspondence}
V_k^t := \sum z_{k}(\sigma_t)(\{\pi(g.\sigma_t): g.\sigma_t^{(0)} \in  \tilde U_{i_0}^t\times ...\times \tilde U_{i_3}^t\} ) \cdot \pi(\sigma(\tilde p_{i_0}^t, ..., \tilde p_{i_3}^t)) \in  \textnormal{C}_3^{\str} (M_{\rho_0};\bR),
\end{equation}
are non-zero.  The sum (\ref{eqn:correspondence}) ranges over all sets of the form $\tilde U_{i_0}^t\times ...\times \tilde U_{i_3}^t$ where $\tilde U_{i_j}^t$ is a lift of $U_{i_j}^t$, and $\sigma(x,y,z,w)$ denotes the straight simplex in $\bH^3$ with ordered vertex set $(x,y,z,w)$.  
The correspondence given by (\ref{eqn:correspondence}) actually defines a norm non-increasing chain map between smooth measure chains and straight chains.  In particular, $\|\partial V_k^t\|_1\le \|\partial z_k(\sigma_t)\| \le K_d$.   

For large $t$, the cells $U_i^t$ have very small diameter, so we can ensure that
\begin{equation}\label{eqn:vol_estimate}
|\vol_3(\sigma(\tilde p_{i_0}^t, ..., \tilde p_{i_3}^t)) - \vol_3(\sigma_t))| <\epsilon/2 ~\text{ and }~  \vol_3(\sigma_t)\ge v_3-\epsilon/2.
\end{equation}

For $\mu$-almost every $g\in \G$, the vertices of $\pi(g.\sigma_t)\in \supp z_k(\sigma_t)$ all lie in the interior of cells $U_{i_0}^t, ..., U_{i_3}^t$.  The Voronoi cells $\{U_i^t\}$ form a measurable partition of $M_{\rho_0}$, from which it follows that
\begin{equation}\label{eqn:chain_size}
 \|V_k^t\|_1 = \|z_k(\sigma_t)\|.
 \end{equation}
 By construction,  $\|z_k(\sigma_t)\| = \Vol(L_{0,k})$, and $\Vol(L_{0,k})\to \infty$ as $k\to \infty$, 
 because $\cup_k L_{0, k} = E_0 \in [E]$.  
 
Finally, from (\ref{eqn:vol_estimate}), (\ref{eqn:chain_size}),  and (\ref{eqn:volume_computation}), we obtain
\[ \vol_3(V_k^t) \ge \|V_k^t\|_1(v_3-\epsilon),\] 
if $t$ is large enough; set $V_k = V_k^t$.  
\end{proof}

\begin{proposition}\label{group chains}
For every positive integer $d\ge 2$, there is a constant $K_d>0$ depending only on the topology of $S_d \cong \partial K_{\rho_0}$ with the following properties.   For every $\epsilon>0$ there exists a sequence of chains $Z_k\in \Cc3(F_d;\bR)$ such that for any $x\in \bH^3$:
\begin{enumerate}[(i)]
\item $\displaystyle \frac{|\rho_0^*\vol_3^x(Z_k)|}{\|Z_k\|_1}>v_3 - \epsilon$, for all $k$, and \label{(i)}
\item $\|\partial Z_k\|_1\le K_d$, for all $k$, and \label{(ii)}
\item $\|Z_k\|_1\to \infty$ and $\displaystyle \lim_{k\to \infty} \frac{\|\partial Z_k\|_1}{\|Z_k\|_1}=0$.\label{(iii)}
\end{enumerate}
\end{proposition}
\begin{proof}
 
 Choose a point $\bar x \in M_{\rho_0}$, and construct the chain map $\str_{\bar x}: \textnormal C_\bullet(M_{\rho_0};\bR) \to \textnormal C^{\str}_\bullet (M_{\rho_0},\{\bar x\};\bR)$ from  Section \ref{isometric chain maps}; recall that the operator norm satisfies $\|\str_{\bar x}\|\le 1$ and we have a chain homotopy $H^\bullet_{\bar x}$ between $\str_{\bar x}$ and $id$, such that $\|H^k_{\bar x}\| = k+1$; see (\ref{eqn:hmtpy_norm}).
Since $\vol_3$ is a bounded cocycle on the straight chains $V_k\in \textnormal C^{\str}_3(M_{\rho_0};\bR)$ provided by Lemma \ref{manifold chains}, we have
 \[|\vol_3(\str_{\bar x} V_k - V_k )| = |\vol_3(H^2_{\bar x}\partial V_k)|\le \|\vol_3\|_\infty\|H^2_{\bar x}\partial V_k\|_1.\]
 Using property (\ref{b}) and the fact that $\|H^2_{\bar x}\|=3$, we get
 \[|\vol_3(\str_{\bar x} V_k - V_k )|\le 3v_3\cdot K_d.\]

Apply the map $(\rho_{0*})\inverse\circ\iota_x: \textnormal{C}^{\str}_\bullet (M_{\rho_0}, \{\bar x\};\bR)\to \textnormal{C}_\bullet(\g_0 ;\bR)\to \textnormal{C}_\bullet(F_d ;\bR)$ to obtain a sequence 
\[Z_k = (\rho_{0*})\inverse\circ \iota_x (\str_{\bar x} V_k)\]
 so that $\rho_0^*\vol^x_3(Z_k) = \vol_3^x\iota_x(\str_{\bar x}(V_k))$. 
Since $(\rho_{0*})\inverse\circ\iota_x$ is an isometric chain map and $\|\str_{\bar x}\| \le 1$, we have $\|Z_k\|_1\le \|V_k\|_1$ and $\|\partial Z_k\|_1\le \|\partial V_k\|_1 \le K_d$, establishing property (\ref{(ii)}).

 Since $|\vol_3(V_k)|\to \infty$ as $ k \to \infty$, it follows that $|\rho_0^*\vol^x_3(Z_k)|\to \infty$, as well.  From this we see that $\|V_k\|_1$ and $\|Z_k\|_1$ tend to $\infty$, from which we obtain property (\ref{(iii)}).  Using the triangle inequality, for $k$ large enough, we have 
 \[v_3-\epsilon<  \frac{|\rho_0^*\vol^x_3(Z_k)| }{\|Z_k\|_1},\]
which establishes property (\ref{(i)}) after reindexing.  
\end{proof}

Qualitatively, $\sigma_t$ has extremely long and thin spikes; the spikes of $\pi(\sigma_t)$ wander circuitously around $M_{\rho_0}$ and are generically recurrent to any compact set in the limit as $t\to \infty$.    
\begin{lemma}\label{to infinity}
In the statement of Proposition \ref{group chains}, we may take $x\in \partial \bH^3$ such that 
\begin{itemize}
\item (\ref{(i)}) holds for all $k$;
\item if for each $k$,  we write
\[Z_k =  \sum_{j = 1}^{M_k}a^{j,k} [w_0^{j,k}, ..., w_3^{j,k}] \in \textnormal C_3(F_d; \bR),\]
then, for each $j=1, ..., M_k$, the points $\rho_0(w_0^{j,k}).x, ..., \rho_0(w_3^{j,k}).x\in \partial \bH^3$ are pairwise distinct.
\end{itemize}
\end{lemma}

\begin{proof}
From Section \ref{volume class},  for any point $y\in\bH^3 \cup\partial \bH^3$ and for any $Z\in \Cc 3(F_d;\bR)$, 
\[|\rho_0^*\vol_3^x(Z) - \rho_0^*\vol_3^y(Z)|\le \|\partial Z\|_13v_3. \]
Using the triangle inequality and the fact that $\partial Z_k$ is bounded for all $k$, we can choose $x\in \partial \bH^3$ so that (\ref{(i)}) holds, after reindexing. 

The group $F_d$ is countable, so there are countably many attracting and repelling fixed points of non-trivial elements of $\im \rho_0$.  Choose $x\in\partial\bH^3$ that is not one of these fixed points.  Then for each $k \ge 1$ and $j = 1, ... , M_k$, if $w_0^{j,k}, ..., w_3^{j,k}$ are pairwise distinct, then  $\rho_0(w_0^{j,k}).x, ..., \rho_0(w_3^{j,k}).x\in \partial \bH^3$ are pairwise distinct. Thus, we just need to show no simplex in $Z_k$ is degenerate.  For this, we revisit the construction of Soma's chains from Lemma \ref{manifold chains}.  Fix $x\in \bH^3$ and  let $\mathcal D\subset \bH^3$ be the Dirichlet fundamental domain for $\g_0$ centered at $x$.  By definition of $\iota_x: C^{\str}_\bullet (M_{\rho_0},\{\bar x\};\bR)\to \Cc \bullet (\g_0;\bR)$ from Section \ref{isometric chain maps}, it is enough to show, given $k$,  that if $t$ is large enough, then  every pair of vertices of any simplex $\pi(\sigma(\tilde p_{i_0}, ..., \tilde p_{i_3}))$ appearing in equation (\ref{eqn:vol_estimate}) defining $V_k^{t}$ lie in distinct translates of $\mathcal D$.  By invariance and without loss of generality, it is enough to show that if $t$ is large enough, then no pair of vertices are both in $\mathcal D$.  

Suppose not.  Recall that $\mathcal D$ is convex being the intersection of countably many half-spaces, so that if two vertices of a straight simplex $ \sigma$ lie in $\mathcal D$, then an edge of $\sigma$ lies in $\mathcal D$.  Thus, for a fixed positive integer $k$, and any $n$, there is a $t_n\ge n$ and simplex $\sigma_n$ appearing in the sum $V_k^{t_n}$ with a lift $\tilde{\sigma_n}$ whose geodesic edge $\gamma_n= \tilde \sigma^n([0,1])$ is contained in $\mathcal D$.  By construction, the center of $\sigma_n$ is contained in $L_{0,k}\subset E$, which is a compact set.  Pass to a subsequence so that $\cen(\sigma_n)\to y\in M_{\rho_0}$.  Then the lift $\tilde y$ of $y$ in $\mathcal D$ is distance at most $2$ from $\gamma_n$, since the midpoint of a geodesic edge of a straight regular simplex passes close to its center (see Equation \eqref{center close}).  Passing to a further subsequence,  the Arzel\`a--Ascoli Theorem guarantees that the geodesic maps $\gamma_n$ converge to a bi-infinite geodesic $ \gamma$.  Since $\gamma_n$ is contained in $\mathcal D$, $ \gamma\subset \overline{\mathcal D}$.

  The locally geodesic projection $\bar \gamma$  of $ \gamma$ to $M_{\rho_0}$ cannot return to any compact set in $M_{\rho_0}$ in forward or reverse time, by definition of $\mathcal D$.  Thus each end of $\bar\gamma$ must exit the only end $[E]$ of $M_{\rho_0}$.  In fact, the two ends of $\gamma$ are necessarily asymptotic in $\bH^3$.  Indeed, by Lemma \ref{lem:bounded_diameter} every surface in the  sequence $\{X_n\}$ of simplicial hyperbolic surfaces exiting $E$ has uniformly bounded diameter in $M_{\rho_0}$, because we assumed that $M_{\rho_0}$ has bounded geometry.  This means that $\pi\inverse(X_n) \cap \mathcal D$ has uniformly bounded diameter.  Since each $X_n$ is separating in $M_{\rho_0}$, $\pi\inverse (X_n) \cap \mathcal D$ separates $\mathcal D$.  From this we conclude, that there is an $N$ such that for all $n\ge N$, each end of $\gamma$ must meet $\pi\inverse (X_n)\cap \mathcal D$.  Thus each end of $\gamma$ meets the same collection of uniformly bounded diameter subsets of $\bH^3$ as they tend to infinity.  But bi-infinite geodesics have distinct endpoints at infinity.   We have reached a contradiction, and we conclude that for each $k$, there is a $t_k$ large enough such that no tetrahedron in the $Z_k$ that we construct from $V_k^{t_k}$ via Proposition \ref{group chains} is degenerate.  This completes the proof of the lemma.
  \end{proof}

\begin{remark}One of the anonymous referees suggested that we could avoid the proof of Lemma \ref{to infinity} by instead appealing to \cite[Lemma 2.5]{Kuessner}.
\end{remark}

We would like to use our approximation scheme from Section \ref{approximation section} to transfer this information to our dense representation.

\begin{proposition}\label{dense group chains}
Let $G$ be a discrete group and fix $x\in \bH^3 \cup\partial \bH^3$.  If $\rho: G\to \G$ is dense, then for every $\epsilon>0$, there is a sequence of chains $D_k\in \Cc3(G; \bR)$ satisfying
\begin{enumerate}[(I)]
\item $\displaystyle \frac{|\rho^*\vol_3^x(D_k)|}{\|D_k\|_1}>v_3-\epsilon$, for all $k$, and \label{(I)}
\item $\displaystyle \lim_{k\to\infty} \frac{\|\partial D_k\|_1}{\| D_k\|_1} = 0$. \label{(II)}
\end{enumerate}
\end{proposition}

\begin{proof}
Apply Proposition \ref{group chains} to obtain $Z_k\in \Cc3(F_d; \bR)$ that satisfy the conclusions (\ref{(i)}), (\ref{(ii)}), and (\ref{(iii)}).    For each $k$, we can now apply Proposition \ref{approximation of chains} to obtain $Z_k(1)\in \Cc3(G; \bR)$ such that 
 \[|\rho^*\vol_3^x (Z_k(1)) - \rho_0^*\vol_3^x(Z_k) | <1,\]
  $\|Z_k(1)\|_1 \le \|Z_k\|$, and $\|\partial Z_k(1)\|_1 \le \|\partial Z_k\|_1\le K_d$.  Note that $\|Z_k(1)\|$ tends to $\infty$, because $\|Z_k\|$ does; see Remark \ref{rmk:to_infinity}.  By property (\ref{(i)}) and the above approximation, we have  
 \[\displaystyle \frac{|\rho^*\vol_3^x(Z_k(1))|}{\|Z_k(1)\|_1}\ge \frac{|\rho_0^*\vol_3^x(Z_k) |}{\|Z_k\|} - \frac{1}{\|Z_k(1)\|} >v_3-\epsilon,\] for large enough $k$, because
 $\|Z_k(1)\|_1$ tends to $\infty$.  Since $\|\partial Z_k(1)\|_1$ stays bounded,  $\displaystyle \lim_{k\to\infty} \frac{\|\partial Z_k(1)\|_1}{\| Z_k(1)\|_1} = 0$.  Take $D_k = Z_k(1)$. 
\end{proof}

\noindent We can now prove the first part of Theorem \ref{thm:intro_main}.  
 
 \begin{theorem}\label{thm:v3_norm}If $G$ is a discrete group and $\rho: G\to \G$ is dense, then $[\rho^*\vol_3]\not= 0 \in \Hb3(G;\bR)$ and $\|[\rho^*\vol_3]\|_\infty = v_3$. 
 \end{theorem}
 
 \begin{proof} 
  The chains $D_k$ from Proposition \ref{dense group chains} satisfy the hypotheses of Lemma \ref{scheme}, so that $\|[\rho^*\vol_3]\|_\infty \ge v_3-\epsilon$.  But $\epsilon>0$ was arbitrary, so $\|[\rho^*\vol_3]\|_\infty \ge v_3$.  On the other hand, $\|[\rho^*\vol_3]\|_\infty \le \|[\vol_3]\|_\infty= v_3$.
 \end{proof}
 
\subsection{Separation of volume classes in semi-norm}\label{sec:separation}
 For a finitely generated Kleinian group, the Covering Theorem \ref{thm:covering} suggests that `{most}' infinite index subgroups are geometrically finite.  We know that the volume classes for geometrically finite classes are trivial by Soma's Theorem \ref{soma's theorem}.  The following technical lemma makes repeated use of the Tameness Theorem \ref{thm:tameness} and the Covering Theorem \ref{thm:covering}. 
 
 Let $X$ be a space and $p: \hat X \to X$ be a covering space.  A map $\hat f: \hat Y\to \hat X$ is an \emph{elevation} of a map $f: Y\to X$ if $\hat Y$ is connected and  $q: \hat Y \to Y$ is a minimal covering such that $p\circ \hat f = f\circ q$.  We sometimes identify an elevation with its image in $\hat X$.  
 
\begin{lemma}\label{lem:technical}
Suppose $\rho: F_2 \to \G$ is discrete and faithful.  There is a finite index subgroup $H\le F_2$ such that for any $\hat H\le H$, isomorphic to a free group of rank $2$, $\rho\circ i$ is geometrically finite with infinite co-volume, where $i: \hat H \to F_2$ denotes inclusion.  
\end{lemma}

\begin{proof}
By Corollary \ref{cor:handlebody}, there is a standard compact core $K_\rho$ of  $M_\rho$ and a homeomorphism $f: \cH_2\to K_\rho$ inducing $\rho$ on fundamental groups.  The collection of core curves $\nu=\{\nu_1, ..., \nu_m\} \subset \partial \cH_2$ of $f\inverse (P_\rho^\circ )$ consists of at most $3$ homotopically essential distinct disjoint simple closed curves.   The inclusion $\partial\cH_2\to \cH_2$ induces a surjection $\iota: \pi_1(\partial \cH_2) \to \pi_1(\cH_2) = F_2$.  
We abuse notation and write $\langle \nu \rangle\le \pi_1(\partial \cH_2)$ to denote the conjugacy classes of the cyclic subgroups corresponding to the components of $\nu$.  The  $F_2$-conjugacy classes of $\iota(\langle \nu\rangle )$ are non-trivial and pairwise distinct, as they correspond to the distinct parabolic cusps of $M_\rho$.

Note that if $P_\rho^\circ$ has $3$ components, then $\nu$ is a pants decomposition of $\partial\cH_2$, hence $\rho$ is maximally cusped hence every relative end is geometrically finite.  Therefore, since $M_\rho$ has no geometrically infinite relative ends,  the Covering Theorem \ref{thm:covering} guarantees that $\rho$ restricts to a geometrically finite representation on {\it any} subgroup $\hat H\le F_2$.  

Now, we assume that $M_\rho$ does have at least one geometrically infinite relative end.  If $P_\rho^\circ$ is empty, then $M_\rho$ has exactly one geometrically infinite end $E = M_\rho\setminus K_\rho$.
We claim that for any proper subgroup $\hat H \le F_2$ of rank $2$, $\rho\circ i:\hat H \to \G$ is geometrically finite.  As above, $K_{\rho\circ i}$ is homeomorphic to a closed handlebody $\hat \cH_2$ of genus $2$ with boundary $\partial \hat \cH_2$.  If  $M_{\rho\circ i}$ does have a geometrically infinite relative end $[\hat E]$,  the Covering Theorem \ref{thm:covering} supplies us with a finite sheeted cover $\hat E\to E$ that defines a finite sheeted cover $\hat Y \to \partial \cH_2$, where $\hat Y\subset \partial \hat \cH_2$ is a homotopically essential subsurface. The only possibility is that $\hat Y = \partial \hat \cH_2$, and that $\hat Y \to \partial \cH_2$ is degree one, hence $\hat H$ maps onto $F_2$.  This is a contradiction to the assumption that $\hat H$ is a proper subgroup of $F_2$,  and so $\rho\circ i$ is geometrically finite.    
Thus we may take $H$ to be  any proper, finite index subgroup of $F_2$, so that any rank $2$ free subgroup $\hat H \le H$ has infinite index in $H$, by Euler characteristic considerations.   In this case, $\rho\circ i$ is geometrically finite, as desired.

The remaining case to consider is that $P_\rho^\circ$ has either $1$ or $2$ components.  For a based loop $\alpha\in \pi_1(\partial \cH_2)$ representing a component of $\nu$, write $\iota (\alpha) = w_\alpha\in F_2$.  Since free groups are residually finite \cite{Stallings:graphs}, there is a finite index normal subgroup $H\triangleleft F_2$ such that $\{w_\alpha : \alpha \subset \nu\} \subset F_2\setminus H$.  Since $H$ is normal,  no conjugate of $w_\alpha$ lies in $H$.  In terms of our notation from before, $\iota(\langle \nu\rangle )\subset F_2\setminus H$.  The covering $M_{\rho|_H}\to M_\rho$ has finite degree $[F_2: H]<\infty$ and $M_{\rho|_H}$ is homeomorphic to the interior of a closed handlebody $\cH_k$ of genus $k= [F_2:H]+1$.  Moreover, $M_{\rho|_H}\to M_\rho$ extends to a covering $\cH_k\to \cH_2$ of closed handlebodies restricting to a finite cover $\partial \cH_k \to \partial \cH_2$; the geometrically infinite relative ends of $M_{\rho|_{H}}$ are elevations of the geometrically infinite relative ends of $M_\rho$.  

By construction of $H$, no conjugate of $\alpha$ lifts to $\partial \cH_k$, so any elevation $\tilde \nu \subset p\inverse (\nu)$ covers a component of $\nu$ with degree at least $2$.  Thus if $Y$ is any component of $\partial \cH_2 \setminus \nu$, and $\tilde Y\subset p\inverse(Y)$ is an elevation, then $p$ restricts to a cover $\tilde Y \to Y$ and further restricts to a cover $\partial \tilde Y \to \partial Y$ of degree at least $2$ on every component of $\partial \tilde Y$.  
If $Y$ is a one holed torus or three holed sphere, and $\tilde Y \to Y$ is a degree $2$ covering, then $\tilde Y$ can only be a $4$-holed sphere or a torus with $2$ holes.  In either case, there is a boundary component of $\tilde Y$ which maps homeomorphically onto a boundary component of $Y$.  Thus our construction requires that the degree of $\tilde Y \to Y$ must be at least $3$; equivalently, $|\chi(\tilde Y)|\ge 3$.
If $Y$ is a two holed torus or four holed sphere, then $|\chi (\tilde Y)| \ge  2|\chi(Y)| \ge 4$.  In particular, $|\chi (\tilde Y)| \ge 3$, for any component $\tilde Y \subset \partial \cH_k \setminus p\inverse(\nu)$.  

Finally, if $\hat H\le H$ is free of rank $2$, then $\hat H$ is a proper subgroup of $H$.  We apply Corollary \ref{cor:handlebody} once more to see that  $M_{\rho\circ i}$ is the interior of a closed genus $2$ handlebody $\hat \cH_2$.  If $M_{\rho\circ i}$ has a geometrically infinite relative end $\hat E$, then $\hat E \cong \hat Y \times [0,\infty)$, where $\hat Y$ is an essential subsurface of a genus $2$ surface.  In particular, either $\hat Y$ is closed or $|\chi(\hat Y)|\le 2$.  By Euler characteristic considerations, $\hat Y$ cannot cover any component $\tilde Y\subset \partial \cH_k\setminus p\inverse (\nu)$.  Therefore, by the Covering Theorem \ref{thm:covering}, $M_{\rho\circ i}$ has no geometrically infinite ends.  This completes the proof of the lemma.
\end{proof}

We would like to  showcase the utility of Lemma \ref{lem:technical} by finding rank $2$ free subgroups of $F_2$ on which one representation is dense and another is geometrically finite or geometrically elementary.  We are thankful to one of the anonymous referees for pointing out a useful observation. 

\begin{observation}\label{obs:finite_index_dense}
Let $\mathcal G$ be a connected topological group, $G\le \mathcal G$ be a dense subgroup, and suppose $H\le G$ has finite index.  Then $H$ is dense in $\mathcal G$. 
\end{observation}
\begin{proof}
Since $\overline H$ has finite index in $\overline G = \mathcal G$, the left cosets $\{g\overline H: g\in \mathcal G\}$ form a finite partition of $\mathcal G$, so that $\overline H$ is open, being the complement of a finite union of closed sets.  Then $\overline H$ is an open, closed, and non-empty subset of a connected space $\mathcal G$; thus $\overline H = \mathcal G$.
\end{proof}   
\begin{proposition}\label{prop:dense_not}
Let $\rho: F_2\to \G$ be dense and faithful and suppose $\rho_0: F_2\to \G$ is discrete.  Then there is a free rank $2$ subgroup $i:\hat H \to F_2$ such that $\rho\circ i$ is a dense representation and  $\rho_0\circ i$ is either geometrically elementary or discrete, faithful, and geometrically finite.  In particular, 
\begin{equation*}\label{eqn:dense_zero}
\|[(\rho\circ i)^*\vol_3]\|_\infty = v_3 \text{ and } [(\rho_0\circ i)^*\vol_3]= 0 \in \Hb 3(\hat H;\bR).
\end{equation*}
\end{proposition}

\begin{proof} We break the proof into two cases:

\para{Case 1}
Assume that $\rho_0$ is faithful.  By Lemma \ref{lem:technical}, there is a finite index subgroup $H< F_2$ such that for any subgroup $\hat H\le H$, which is free of rank $2$, $\rho_0\circ i $ is geometrically finite, where  $i: \hat H\to H\to F_2$ is inclusion.  By Observation \ref{obs:finite_index_dense}, $\overline{\rho(H)} = \G$.  According to \cite[Theorem 1.1]{BG:dense}, there is a free subgroup $\hat H\le H$ of rank $2$ such that $\overline{\rho(\hat H)} =\G$.  Applying Lemma \ref{lem:technical}, $\rho_0\circ i$ is geometrically finite.  Thus $[(\rho_0\circ i)^*\vol_3]=0$, by Soma's Theorem \ref{soma's theorem}, while  Theorem \ref{thm:v3_norm} gives $\|[(\rho\circ i)^*\vol_3]\|_\infty = v_3$. 

\para{Case 2} 
Assume that $\rho_0$ is not faithful, and let $K = \ker \rho_0$.  Then $K\triangleleft F_2$ is a free group of rank at least $2$ (perhaps infinite rank).  Since $\rho$ is faithful, $\rho(K)$ is not virtually abelian. It follows, e.g. from \cite[Lemma 2.3]{MatTan}, that  either there is an $a\in K$ such that $\rho(a)$ is loxodromic or $\rho(K)$ fixes a point in $\bH^3$.    

For sake of contradiction, suppose $\rho(K)$ fixes a point in $\bH^3$.  Since $K$ is normal in $F_2$, the set of points fixed by $\rho(K)$ is a non-empty $\rho(F_2)$-invariant subset of $\bH^3$.  But $\rho$ is dense, so the closure is all of $\bH^3$.  This  can only happen if $\rho(K) = \{1\}$.  However, $\rho$ is faithful and $K$ is non-trivial, which is a contradiction.  Thus $\rho(K)$ does not have a global fixed point.

A loxodromic element $\gamma\in \G$ stabilizes a geodesic  $\axis(\gamma)\subset\bH^3$, and acts as a ``screw motion'' of complex length $\tau(\gamma) = \ell + i \theta$ along $\axis(\gamma)$.  That is, $\ell\in\bR$ is the translation length of $\gamma$, which is realized along $\axis(\gamma)$, and $\theta\in\bR/2\pi\bZ$ represents rotation around $\axis (\gamma)$. 

We have an $a\in K$ such that $\rho(a)$ is loxodromic.  Let $y\in \bH^3$ be a point on $\axis (\rho(a))\subset \bH^3$, and let $\epsilon>0$ be much smaller than the $3$-dimensional  Margulis constant $\mu_3$. 
 Using density of $\rho$, we can find $b\in F_2$ such that
\begin{itemize}
\item   $\rho(b)$ is loxodromic,
\item  $|\tau(\rho(a)) - \tau(\rho(b))|<\epsilon$,  
\item  $\tau(\rho(b))$ is not real,
\item the attracting and repelling fixed points of $\rho(a)$ and $\rho(b)$ at infinity are pairwise distinct, and 
\item  $\axis(\rho(b))$ is within distance $\epsilon/10$ of $\axis (\rho(a))$ in a  $2|\tau(\rho(a))|$-neighborhood of $y$.
\end{itemize}
Let $\hat H = \langle a, b\rangle$; our goal is to show that $\rho(\hat H)$ is dense in $\G$.  Taking $z = \rho(a).y$, the above conditions imply that $\rho(ab\inverse).z$ and $\rho(ab^{-2}a).z$ are both closer than $\mu_3$ to $z$.  Let $H' =\langle ab\inverse, ab^{-2} a\rangle$; by the Margulis Lemma \ref{lem:Margulis}, $\rho(H')\le \G$ is either indiscrete or virtually abelian.  Since $\rho$ is faithful and $H'$ is free, $\rho(H')$ is free, hence not virtually abelian.   So $\overline{\rho(H')}$ is a closed Lie subgroup of $\G$  with positive dimension that is not virtually abelian, and since $\hat H \ge H'$,  $\overline{\rho(\hat H)}$ has the same property.    Moreover, the endpoints of the axes of $\rho(a)$ and $\rho(b)$ at infinity are pairwise distinct, so $\overline {\rho(\hat H)}$ does not stabilize an ideal point, and $\overline{\rho(\hat H )}$ is not conjugate into $\PSL_2\bR$, since the complex translation length of $\rho(b)$ is not real.  By Lemma \ref{lem:geom_elem}, $\rho|_{\hat H}$ is dense.

By construction, $\rho_0(\hat H)$ is cyclic, hence geometrically elementary.  We conclude as in Case 1,  supplementing Soma's Theorem \ref{soma's theorem} with the proof of Lemma \ref{lem:geom_elem} in case $\rho_0(\hat H)$ has torsion, say.  

\end{proof}
\begin{remark}
The argument given in Case 2 also proves the proposition in Case 1; we will use the argument from Case 1 in the next section.
Proposition \ref{prop:dense_not} is an improvement of a result in a previous draft of this manuscript, and strengthens our main theorems from that version.
\end{remark}

Our first main theorem now follows quite easily.  We only need to observe that restrictions to subgroups induce semi-norm non-increasing maps in bounded cohomology.  The first step in the proof of the following is a reduction; we use \cite[Theorem 1.1]{BG:dense} to find a free subgroup of rank $2$ of $H$, densely embedding into $\G$ via $\rho$. 

 \begin{theorem}\label{thm:main}
 Suppose $\rho: G \to \G$ is a dense representation of a discrete group $G$.  If $\rho_0: G\to \G$ is any other representation and there is a subgroup $H\le G$ such that $\overline{\rho(H)} = \G$, but $\rho_0$ is geometrically elementary or discrete restricted to $H$, then \[\|[\rho^*\vol_3]-[\rho_0^*\vol_3]\|_\infty\ge v_3. \] 
 \end{theorem}

\begin{proof}
By \cite[Theorem 1.1]{BG:dense}, there is a free subgroup $F_2\le H$ such that $\rho$ is dense and faithful on $F_2$.  Moreover, $\rho_0$ is still geometrically elementary or discrete on $F_2\le H$, because these properties pass to subgroups, as we have seen. If $\rho_0$ is geometrically elementary, the volume class of $\rho_0$ vanishes when restricted to this $F_2$, by the proof of Lemma \ref{lem:geom_elem}, so we set $\hat H = F_2$ and let $i: \hat H \to F_2\to G$ denote inclusion.  Otherwise, we apply Proposition \ref{prop:dense_not} to obtain a free rank $2$ subgroup $i: \hat H\to  F_2\to G$ where $\|[(\rho\circ i)^*\vol_3]\|_\infty = v_3$ and $[(\rho_0\circ i)^*\vol_3]= 0$.  

Since $i^*: \Hb\bullet(G;\bR)\to \Hb\bullet (\hat H;\bR)$ is norm non-increasing, we have
\begin{align*}
v_3&= \|[(\rho\circ i)^*\vol_3]\|_\infty \\ &= \|[(\rho\circ i)^*\vol_3] - [(\rho_0\circ i)^*\vol_3]\|_\infty \\ &   =\|[i^*([\rho^*\vol_3] - [\rho_0^*\vol_3])\|_\infty \\  &\le \|[\rho^*\vol_3] - [\rho_0^*\vol_3]\|_\infty.  \label{norm nonincreasing 2}
\end{align*}
This is what we wanted to show.
\end{proof}

If we are more careful, we can obtain the following generalization without too much extra work.
\begin{theorem}\label{ell 1}
Suppose $\{\rho_i: G\to \G\}_{i = 1}^N$ is a collection of dense representations of a discrete group $G$ such that for every $i = 1,2, ... , N$ there is a subgroup $\iota_i:H_i\hookrightarrow G$ such that $\overline{\rho_i(H_i)} = \G$, and $\rho_i\circ \iota_j: H_j \to \G$ is geometrically elementary or discrete for $i\not=j$.  Then for any $a_1, ..., a_N\in \bR$, we have 
\[\left\| \sum_{i = 1}^N a_i[\rho_i^*\vol_3] \right \|_\infty \ge \max\{|a_i|\} \cdot v_3. \]
Consequently, $\{[\rho_i^*\vol_3]\}\subset \Hbr3(G;\bR)$ is a linearly independent discrete set.
\end{theorem}
\begin{proof}
For convenience, we assume that $|a_1| = \max\{|a_i|\}$.
First of all, by \cite{BG:dense} we may assume that $H_1\cong F_2$ and that $\rho_1|_{H_1}$ is faithful.  We inductively define a nested family of rank $2$ subgroups 
\[H_1^{(N)}\le ... \le H_1^{(2)}\le H_1^{(1)} = H_1\]
 by successively applying Proposition \ref{prop:dense_not}.  The result is that $\rho_1|_{H_1^{(i)}}$ is dense and faithful for all $i$.  Additionally, for each $j\ge 2$, denote inclusion by $\iota_1^{(j)}: H_1^{(j)}\to G$ so that for $i\in\{2, ..., N\}$, 
  \[[(\rho_i\circ \iota_1^{(j)})^*\vol_3] = 0 \in \Hb 3(H_1^{(j)};\bR).\]

As in Theorem \ref{thm:main}, the operator $\iota_1^{(N)*}: \Hb 3(H_1^{(N)};\bR)\to \Hb 3(G; \bR)$ is norm non-increasing.  Applying Theorem \ref{thm:v3_norm}, we see

\begin{align*}
|a_1|\cdot v_3&= \left\|a_1[(\rho_1\circ \iota_1^{(N)} )^*\vol_3]\right\|_\infty \\ 
&= \left\|[a_1(\rho_1\circ \iota_1^{(N)})^*\vol_3] + \sum_{i = 2}^N a_i[(\rho_i\circ \iota_1^{(N)})^*\vol_3]\right\|_\infty \\ 
&   =\left\|\iota_1^{(N)*}\left(\sum_{i = 1}^N a_i[\rho_i^*\vol_3] \right)\right\|_\infty \\ 
& \le \left\| \sum_{i = 1}^N a_i[\rho_i^*\vol_3] \right \|_\infty,
\end{align*}
as promised.  
\end{proof}

\section{Borel classes of dense representations}\label{Borel}
In this section, we show that pullbacks of the Borel class under dense representations have maximal semi-norm.  
Since the structure of discrete subgroups $\g\le \PSL_n\bC$ is not well understood, we cannot give a simple criterion for the differences of pullbacks of Borel classes to be separated in semi-norm for arbitrary representations.  Recall that there is a unique conjugacy class of irreducible representations $\iota_n: \G\to \PSL_n\bC$.  We will  work with a dense representation $\rho: G\to \PSL_n\bC$ and another representation $\rho_0': G\to \PSL_n\bC$ that factors through $\G$ via $\iota_n$.    We can then use tools developed in previous sections and $3$-dimensional hyperbolic geometry to give some criteria for which pullbacks to $G$ of the bounded Borel class are separated in semi-norm.  

It would be interesting to investigate whether maximality of the semi-norm of the pullback of the Borel class under discrete and faithful representations $\rho: \pi_1(S) \to \PSL_n\bC$ can serve as a profitable definition for `geometrical infiniteness', where $S$ is an orientable surface of finite type.  In fact, this paper grew out of an attempt to initiate such an investigation.  

\begin{theorem}\label{Borel:main}
Let $G$ be a discrete group and $\rho: G \to \PSL_n\bC$ be dense.  Then \[\|\rho^*\beta_n \|_\infty = v_3 \frac{n(n^2-1)}{6}.\]
\end{theorem}
  The following will be immediate from the proof of Theorem \ref{Borel:main} and Theorem \ref{thm:main}.  
\begin{corollary}\label{cor:Borel_main}
 Let $G$ be a discrete group and  $\rho: G \to \PSL_n\bC$ be dense.  Suppose also that  $\rho_0 : G \to \PSL_2\bC$ is such that there exists a subgroup $H\le G$ such that $\overline{\rho(H)} = \PSL_n\bC$ and $\rho_0$ is geometrically elementary or discrete and faithful restricted to $H$.  Then \[\|\rho^*\beta_n - (\iota_k\circ \rho_0)^*\beta_k \|_\infty \ge v_3 \frac{n(n^2-1)}{6},\] for all $ k\ge 2$.
\end{corollary}
\begin{proof}
We apply \cite[Theorem 1.1]{BG:dense} to obtain $ F_2 \le   H$ such that $\rho|_{F_2}$ is faithful and dense. The argument in Proposition \ref{prop:dense_not} (Case $1$) provides us with a rank $2$ subgroup $i: \hat H\to F_2$ such that $\rho\circ i$ is dense and faithful, while $[(\rho_0\circ i)^*\vol_3] = 0$.    Then $(\iota_k\circ\rho_0\circ i )^*\beta_k =0$ for all $k\ge2$, by Theorem \ref{Borel class theorem}, but $\|(\rho\circ i)^*\beta_n\|_\infty= v_3\frac{n(n^2-1)}6$ by Theorem \ref{Borel:main}.  Since $i^*$ is semi-norm non-increasing on bounded cohomology, the corollary follows (as in the proof of Theorem \ref{thm:main}). 
\end{proof}
Note that in the previous section, we obtain stronger results; in Corollary \ref{cor:Borel_main} we have made the additional assumption that $\rho_0$ is faithful, in addition to being discrete on $H$.  This is because the structure of positive dimensional Lie subgroups of $\G$ is particularly simple (see the proof of Case 2, Proposition \ref{prop:dense_not}).  We believe that with more work, it should be possible to upgrade the results in this section.
\begin{remark} Another easy consequence of Theorem \ref{Borel:main} and the triangle inequality is that if $\rho_1: G \to \PSL_n\bC$ and $\rho_2: G \to \PSL_k\bC$ are both dense, then \[\|\rho_1^*\beta_n - \rho_2^*\beta_k \|_\infty\ge v_3\left| \frac{n(n^2-1)}{6} - \frac{k(k^2-1)}{6}\right|. \] 
\end{remark}

We begin proving Theorem \ref{Borel:main} by finding a sequence of efficient chains on which the Borel class evaluates to a large number with controlled boundary.  These chains come from hyperbolic geometry, together with the explicit description of the cocycle $B_n^F$ given by \cite{BBI:borel} (see Proposition \ref{prop:borel_values}) and the explicit formula for the semi-norm of the pullback under the irreducible representation $\iota_n: \G\to \PSL_n\bC$; see Theorem \ref{Borel class theorem}.
\begin{lemma}\label{Borel group chains}
If $G$ is a discrete group and $\rho: G\to \PSL_n\bC$ is dense, then for every $\epsilon>0$, there is a $y\in \mathscr F(\bC^n)$ and a sequence of chains $D_k\in \Cc3(G; \bR)$ such that 
\begin{enumerate}[(A)]
\item $\displaystyle \frac{|\rho^*B_n^y(D_k)|}{\|D_k\|_1}>(v_3-\epsilon)\frac{n(n^2-1)}{6}$, for all $k$, and
\item $\displaystyle \lim_{k\to\infty} \frac{\|\partial D_k\|_1}{\| D_k\|_1} = 0$.
\end{enumerate}
\end{lemma}

\begin{proof}
As in the proof of Proposition \ref{dense group chains}, fix a geometrically infinite discrete and faithful representation $\rho_0: F_2\to \PSL_2\bC$ with no parabolic elements, take $\epsilon>0$ and apply Proposition \ref{group chains}.  We now have $K_2>0$ and $Z_k\in \Cc 3(F_2;\bR)$ that satisfy conditions (\ref{(i)}), (\ref{(ii)}), and (\ref{(iii)}) of the conclusion of Proposition \ref{group chains}.  By Lemma \ref{to infinity} we may assume that $x\in \partial \bH^3$ for the conclusion (\ref{(i)}), and no ideal simplex in $({\rho_0}_*Z_k).x$ is degenerate.  Recall that the irreducible representation $\iota_n$ induces an equivariant continuous map of boundaries $\hat\iota_n: \partial \bH^3 \to \mathscr{F}(\bC^n)$.  Take $y =\hat\iota_n(x)$, so that by equivariance of $\hat\iota_n$, we have $\hat\iota_n( \rho_{0*}(Z_k).x) = (\iota_n\circ\rho_0)_*(Z_k).y$.  By Proposition \ref{prop:borel_values}, $\iota_n^*B_n^y = \frac{n(n^2-1)}6 \vol_3^x$.  Thus,
\[(\iota_n\circ\rho_0)^*B_n^y(Z_k) = \frac{n(n^2-1)}6 \rho_0^*\vol_3^x(Z_k).\] 
Since no simplex is degenerate, thanks to Lemma \ref{to infinity}, we can apply Corollary \ref{Borel approximation} for each $k$ to obtain $Z_k(1)\in \Cc 3(G;\bR)$ such that 
 \[|\rho^*B_n^y (Z_k(1)) - (\iota_n\circ\rho_0)^*B_n^y(Z_k) | <1.\]
As in the proof of Proposition \ref{dense group chains}, the two conclusions now follow with $D_k = Z_k(1)$.
\end{proof}

\begin{proof}[Proof of Theorem \ref{Borel:main}]
The chains from Lemma \ref{Borel group chains} satisfy the hypotheses of Lemma \ref{scheme}.  Thus $\|\rho^*\beta_n \|_\infty \ge (v_3-\epsilon) \frac{n(n^2-1)}{6}$, for all $\epsilon>0$. By Theorem \ref{Borel class theorem}, $\|\rho^*\beta_n \|_\infty\le v_3 \frac{n(n^2-1)}{6}$, which yields the desired equality.
\end{proof}

\section{Constructions, examples, and questions}\label{sec:applications}

We now give some applications of the work that we have done to show that subspaces of bounded cohomology spanned by the pullback of volume classes can be quite large.  We also pose some questions.
\subsection{Constructing incompatible representations}\label{constructions}
We will construct a family of representations $\{\rho_\theta: F_2\to \G: \theta\in \Lambda\}$ such that, given any finite subset $\{\rho_{\theta_1}, ..., \rho_{\theta_N}\}$, there are subgroups $H_i\le F_2$ such that $\overline{\rho_{\theta_i}(H_i)}= \G$, but $\rho_{\theta_i}|_{H_j}$ is discrete, faithful, and convex co-compact, i.e. $\rho_{\theta_i}|_{H_j}$ is a \emph{marked Schottky group} group for $i\not = j$.   Marked Schottky groups are, in particular, geometrically finite, so we can apply Theorem \ref{ell 1} to show that the volume classes of these representations are linearly independent in reduced bounded cohomology.  Furthermore, the set $\Lambda$ has cardinality that of the continuum.  The construction of dense representations in this section shares some features with the construction found in \cite[Appendix A]{GGKW}.

We start with  a loxodromic element $a \in  \PSL_2\bC$ such that the complex translation length $\tau(a)$ of $a$ has translational part $r$ and non-zero rotational part strictly between $0$ and $\pi/8$.  In particular, $a$ is not conjugate into $\PSL_2\bR$. 
Find an elliptic element $b(\theta)$  with rotation angle $2\pi\theta$, where $\theta\in \bR/\bZ$ is irrational and with fixed line meeting $\axis( a)$ orthogonally in a point $x$.  For infinitely many values $n$, $\axis (b(\theta)^nab(\theta)^{-n})$  makes a very small angle with $\axis (a)$ at $x$. 
\begin{lemma}\label{threshold} Let $r>0$ be given.  There is a threshold $\tau_0\in (0,1/2)$ such that if $n \theta \in (-\tau_0, \tau_0)\mod 1$, then $\langle a, b(\theta)^nab(\theta)^{-n}\rangle$ is dense in $\G$.
\end{lemma}
\begin{proof}
For any $\tau \in (0,1)$, $\axis (b(\tau)ab(\tau)\inverse)= b(\tau)\axis (a)$, and  $\axis(a) \cap b(\tau)\axis (a) = \{x\}$, since the fixed line of $b(\tau)$ meets $\axis(a)$ at $x$.    Using hyperbolic trigonometry, there is a $\tau_0 = \tau_0(r)$ such that if $-\tau_0<\tau<\tau_0\mod 1$, then in an $(r+1)$-neighborhood of $x$, the Hausdorff distance between $\axis(a)$ and $b(\tau)\axis (a)$ is at most $\mu_3$.   Specifically, take $\tau_0\in (0,1)$ such that $2\sinh\inverse (\sin(2\pi\tau_0)\sinh(r+1))<\mu_3$.  Suppose $-\tau_0<\tau<\tau_0 \mod 1$, and for convenience, write $b$ for $b(\tau)$.  

Then $d(ax, bab\inverse x)<\mu_3$, because $d(x,ax)= d(x,bab\inverse x)=r<r+1$.  This means that $d(ba\inverse b\inverse a x, x)<\mu_3$.  Similarly, $d(bab\inverse a\inverse x, x)<\mu_3$.    Set $c = bab\inverse$; by the Margulis Lemma \ref{lem:Margulis}, the group $\langle c\inverse a, ca\inverse \rangle$ is indiscrete if it is not virtually abelian. We will show that $\langle c\inverse a, ca\inverse\rangle$ contains a free subgroup, so it is not virtually abelian.  
 
 An element $\gamma\in \G$ is loxodromic if and only if the trace of a lift of $\gamma$ to $\textnormal {SL}_2\bC$ is not in the interval $[-2,2]$.  Trace identities give 
\[\tr(c\inverse a) = \tr(ac\inverse) = \tr(ca\inverse).\]
By construction, $a$ and $c$ can be represented by matrices
\begin{equation*}
a = \begin{pmatrix}
e^{\tau (a)/2} & 0 \\
0 & e^{-\tau (a)/2}
\end{pmatrix} \text{ and } 
c = 
\begin{pmatrix}
\cos(2\pi\tau) & -\sin(2\pi \tau) \\
\sin(2\pi\tau) & \cos(2\pi\tau)
\end{pmatrix}
\begin{pmatrix}
e^{\tau (a)/2} & 0 \\
0 & e^{-\tau (a)/2}
\end{pmatrix} 
 \begin{pmatrix}
\cos(2\pi\tau) & \sin(2\pi \tau) \\
-\sin(2\pi\tau) & \cos(2\pi\tau)
\end{pmatrix}.
\end{equation*}
Explicit computation with matrices shows that as long as the rotational part of $\tau(a)$ is not an integer multiple of $\pi$ and $b$ is not rotation by  $0$ or $\pi$ (i.e. $\tau \not=0,1/2$), then the imaginary part of $\tr(c\inverse a ) = \tr(c a\inverse)$ is different from $0$.  Thus our assumptions guarantee  that both $c\inverse a$ and $ca\inverse$ are loxodromic.  
The traces are equal, so the complex translation lengths are equal.  This means that if $\axis(c\inverse a)= \axis( ca\inverse)$, then $c\inverse a$ and $ca\inverse$ are either equal or inverse to one another.  Inspection shows that $c\inverse a.x \not = ca\inverse.x$, if the rotational part of $\tau(a)$ is small enough (less than $\pi/8$, for example).  If $c\inverse a= (ca\inverse)\inverse$, then $c$ commutes with $a$, which only happens if $b$ is rotation by $0$ or $\pi$.  

So, the axes of $c\inverse a$ and $ca\inverse$ are different.  If the axes are asymptotic, then again since $\tau(c\inverse a ) = \tau(ca\inverse)$, the product $c\inverse a (ca\inverse)\inverse=c\inverse a^2 c\inverse$ stabilizes a horosphere, hence has trace equal to $\pm 2$.  Another explicit computation shows that the imaginary part of the $\tr(c\inverse a^2c\inverse)$  vanishes only when $b$ is rotation by $0$ or $\pi$ or the rotational part of $\tau(a)$ is an integer multiple of $\pi/8$.  We have assumed that the rotational part of $\tau(a)$ is non-zero and less than $\pi/8$, so the set of fixed points of $c\inverse a$ and $ca\inverse$ at infinity are distinct.  
By   the Ping-Pong Lemma, for large enough $k$, $\langle (c\inverse a)^k  , (ca\inverse)^k \rangle $ is a Schottky group.  In particular, $\langle c\inverse a, ca\inverse \rangle$ contains a free subgroup, and so it is indiscrete.  Thus $\langle a, bab\inverse\rangle$ is indiscrete, because $\langle c\inverse a, ca\inverse \rangle\le \langle a, bab\inverse\rangle$.

Finally,  $\langle a, bab\inverse\rangle$ does not fix an ideal point, nor does it stabilize a plane, since the rotational part of $\tau(a)$ is non-trivial.   By the proof of Lemma \ref{lem:geom_elem}, $\langle a, bab\inverse\rangle$ is dense.  
\end{proof}

There are also infinitely many values of $n$ such that $\axis(a)$ is nearly orthogonal to  $b(\theta)^n\axis(a)$.  The following is immediate from the Ping-Pong Lemma and another direct computation in $\bH^2\subset \bH^3$.

\begin{lemma}\label{long enough} For $r> \frac{1+\cos(\pi/8)}{1-\cos(\pi/8)}$, if the translational part of $\tau(a)$ is at least $r$ and 
$n\theta \in ( 1/8, 3/8)\mod 1$, then $\langle a, b(\theta)^nab(\theta)^{-n}\rangle $ is a Schottky group of rank $2$.  
\end{lemma}
\noindent Now fix $a\in \PSL_2\bC$ with translation length $r + it$, with $r> \frac{1+\cos(\pi/8)}{1-\cos(\pi/8)}$ and $t\in (0,\pi/8)$.
\begin{lemma}\label{N reps}
Let $\{ \theta_1, ..., \theta_N\}\subset (0,1)$ be a rationally independent set of irrational numbers, $F_2 = \langle z_1, z_2\rangle$, and let $\rho_{\theta_i} : F_2\to \G$ be defined by $\rho_{\theta_i}(z_1) = a$ and $\rho_{\theta_i}(z_2) = b(\theta_i)$.  There are integers $n_1, ..., n_N$ such that $H_i : = \langle z_1, z_2^{n_i} z_1 z_2^{-n_i}\rangle$  satisfies $\rho_{\theta_i}(H_i)\le \G$ is dense but $\rho_{\theta_i}|_{H_j}$ is a marked Schottky group. 
\end{lemma}
\begin{proof}
Since $\{\theta_i\}$ is a rationally  independent set of irrational numbers, the self homeomorphism of the $N$-torus $\Theta : (\bR/\bZ)^N\to (\bR/\bZ)^N$ defined by  $(x_1, ..., x_N)\mapsto (x_1+\theta_1, ..., x_N + \theta_N)$ is topologically minimal, i.e. every orbit is dense.   Let $\tau_0$ be the threshold from Lemma \ref{threshold}, and let $D = (-\tau_0, \tau_0) \subset \bR/\bZ$ and $F = (1/8, 3/8)\subset \bR/\bZ$.  For each $i$, let $p_i: (\bR/\bZ)^N\to \bR/\bZ$ be projection onto the $i$th factor and take  $U_i \subset (\bR/\bZ)^N$ to be the product of $D$'s and $F$'s where $p_i(U_i) = D$ and $p_j(U_i) = F$, if $i\not = j$.  By minimality of $\Theta$, there is an $n_i$ such that $\Theta^{n_i}(0)\in U_i$. By Lemma \ref{threshold}, $\overline{\rho_{\theta_i} (H_i)}= \G$, and $\rho_{\theta_i}|_{H_j} : H_j\to \G$ is discrete,  faithful, and convex co-compact for $i\not=j$, by Lemma \ref{long enough}. 
\end{proof}

Using the axiom of choice, we can find a basis $\Lambda\sqcup \{1\}$ for $\bR$ as a $\bQ$-vector space.  We may assume that $\Lambda\subset (0,1)$.  For each $\theta\in \Lambda$ we have the representation $\rho_\theta: F_2 \to \G$ defined as in Lemma \ref{N reps}.
\begin{theorem}\label{uncountable family}
The map \[
\begin{aligned}
\Lambda &\to \Hbr3(F_2)\\
\theta&\mapsto  [\rho_\theta^*\vol_3]
\end{aligned}\]  is injective with discrete image.  Moreover, $\{[\rho_\theta^*\vol_3]: \theta\in \Lambda\}$ is a linearly independent set, and $\#\Lambda = \# \bR$. 
\end{theorem}

\begin{proof}
Any finite subset $\{\theta_1, ..., \theta_N\}\subset \Lambda$ is rationally independent.  By Lemma \ref{N reps}, the collection $\{\rho_{\theta_i}: i = 1, ..., N\}$ satisfies the hypotheses of Theorem \ref{ell 1}.  
This shows that $\{ [\rho_\theta^*\vol_3]: \theta\in \Lambda\} \subset \Hbr3(F_2)$ is linearly independent and discrete.  Injectivity follows from linear independence.  
\end{proof}

\subsection{On the spaces $\Hbr 3(\PSL_2\mathbb \bR;\bR)$ and $\Hbr 3(\PSL_2\mathbb \bC;\bR)$}
As a further application of Theorem \ref{ell 1}, we will show that both $\Hbr 3 (\PSL_2\bR; \bR)$ and $\Hbr 3 (\G;\bR)$ have dimension at least $\#\bR$.  We remind the reader that we are computing bounded cohomology of discrete groups, i.e.  $\PSL_2\bR$ and $\G$ are endowed with the discrete topology.  The key is to use the axiom of choice to find `wild' field automorphisms of $\bC$ that induce embeddings $\PSL_2\bR\hookrightarrow \G$ and $\G\hookrightarrow \G$;  these embeddings would be highly discontinuous when these groups are considered with their standard smooth topologies.  

\begin{theorem}\label{G dense} The real dimension of the degree three reduced bounded cohomology of $\PSL_2\mathbb \bR$ and $\G$ is at least $\#\bR$.  More specifically, there are dense representations $\{\rho_t: \PSL_2\bR \to \G\}_{t\in \bR}$ such that $\{[\rho_t^*\vol_3]: t\in \bR\}$ is a linearly independent set in $\Hbr 3(\PSL_2\bR;\bR)$.  The family $\{\rho_t\}$ are the restrictions of representations $\{\rho_t': \G\to \G\}$, thus $\{[{\rho_t'}^*\vol_3]: t\in \bR\}$ is a linearly independent set in $\Hbr 3(\PSL_2\bC;\bR)$.
\end{theorem}
\begin{proof} 
Let $\Hom(\bR, \bC) $ denote the set of injective homomorphisms of fields $\sigma: \bR \to \bC$.  Every field mapping takes polynomial identities with rational coefficients to polynomial identities with rational coefficients, hence induces a group homomorphism $\rho_\sigma : \PSL_2\bR \to \G$.  It is not hard to see that if $\sigma (\bR) \subset \bR$, then $\sigma$ is order preserving, hence $\sigma$ restricts to the identity mapping on $\bR$, i.e. $\sigma$ is trivial \cite[Theorem 3]{Yale:field_maps}.  In fact, if $\sigma\in \Hom(\bR, \bC)$ is not trivial, then $\sigma(\bR\setminus \bQ)$ is dense in $\bC$ \cite[Theorem 4]{Yale:field_maps}.  Hence $\rho_{\sigma}(\PSL_2\bR)$ is dense in $\G$ for non-trivial $\sigma$.  By Theorem \ref{thm:v3_norm}, $\|[\rho_{\sigma}^*\vol_3]\|_\infty=v_3$ for all non-trivial $\sigma\in \Hom(\bR, \bC)$.  

Next we will produce a family of non-trivial field maps $\sigma(t)\in \Hom(\bR, \bC)$ indexed by a set with cardinality that of the continuum, which we assume is $\bR$ for simplicity.  The field maps $\sigma(t)$  induce dense representations $\rho_t = \rho_{\sigma(t)}: \PSL_2\bR\to \G$.  Finite subsets will satisfy the hypotheses of Theorem \ref{ell 1}, but in fact any subset will satisfy the hypotheses of Theorem \ref{ell 1}.  

Let $\alpha, \beta\in \bC$, and consider
\[x(\alpha) = 
\begin{pmatrix}
 \alpha  & 0 \\
 0 & \alpha \inverse
\end{pmatrix},
\text{ and } 
y (\beta) = \begin{pmatrix}
 (\beta + \beta\inverse)/2 & (\beta - \beta\inverse)/2 \\
 (\beta - \beta\inverse)/2 &  (\beta + \beta\inverse)/2
\end{pmatrix}
\in \G.\]
We consider the group
\[ H_{\alpha, \beta} = \left\langle x(\alpha), y(\beta)\right\rangle \le \G.\]
Geometrically, as long as $|\alpha|, |\beta|\not=1$, then the two generators are loxodromic with axes that meet orthogonally in a point, hence are contained in a hyperbolic plane.  If $|\alpha|$ and $|\beta|$ are large enough, by playing Ping-Pong, we see that  $H_{\alpha, \beta}$ is Schottky, hence geometrically finite.  If $|\alpha|, |\beta|<\mu_3$, then $H_{\alpha, \beta}$ is dense in $\G$ as long as at least one of $\alpha$ or $\beta$ has argument not a multiple of $\pi/2$, essentially by an easier argument than found in the proof of Lemma \ref{threshold}. 

Since $\bC$ is algebraic over $\bQ(\bR) = \bR$, there is a transcendence basis $T\subset \bR$ for $\bC$ over $\bQ$  \cite[Theorem 19.14]{Morandi:fields}.  Let $a, b \in \bC\setminus \bR\cup i\bR$ be $\bQ$-linearly independent, algebraic numbers with magnitude  smaller than $\mu_3$.  By the Lindemann-Weierstrass Theorem \cite[Theorem 1.4]{Baker:transcendental}, $\{e^a, e^b\}$ is an algebraically independent set.  Since $\bQ(\{z\in \bC: |z|>100\}) = \bC$ is algebraic over $\bC$, we may find a transcendence basis $T'$ for $\bC$ over $\bQ$ such that $\{e^a, e^b\}\subset T'\subset \{|z|>100\}\cup \{e^a, e^b\}$   \cite[Theorem 19.14]{Morandi:fields}.  Two transcendence bases for $\bC$ over $\bQ$ have the same cardinality, that of the continuum.  The algebraic closures of $\bQ(T)$ and $\bQ(T')$ are both $\bC$, so if $\pi: T\to T'$ is a bijection, then there is a field isomorphism $\sigma_\pi : \bC \to \bC$ extending $\pi$, by the Isomorphism Extension Theorem.

Find a partition $T=A\sqcup B$, with $\#A = \#B$ and bijections $\bR \to A$ and $\bR\to B$ denoted $t\mapsto \alpha_t$ and $t\mapsto \beta_t$, respectively.  Now, for each real number $t$, we extend the assignment $\alpha_t \mapsto e^a$ and $\beta_t \mapsto e^b$ to a bijection $\pi(t): T\to T'$.  We thus induce a field isomorphism $\sigma(t)= \sigma_{\pi(t)}: \bC\to \bC$ that restricts to a field map $\bR\to\bC$ with the same name.   We obtain, for each $t\in \bR$, a homomorphism $\rho_t = \rho_{\sigma(t)}: \PSL_2\bR \to \G$ and also $\rho_t': \G\to \G$.  

By construction, $\rho_t(x(\alpha_t)) = x(e^a)$ and $\rho_t(y(\beta_t)) = y(e^b)$, so that $\rho_t(H_{\alpha_t, \beta_t})= H_{e^a,e^b}$ is dense in $\G$, by our choice of $a$ and $b$.  For $s\not= t$, we see that the group $\rho_t(\langle x(\alpha_s), y(\beta_s)\rangle)$ is Schottky, because $\sigma_t(\alpha_s), \sigma_t(\beta_s) \in \{|z|>100\}$.  Thus {\it every} subset of representations of $\{\rho_t: \PSL_2\bR\to \G : t\in \bR\}$ satisfies the hypotheses of Theorem \ref{ell 1}.  As in the proof of Theorem \ref{uncountable family}, $\{[\rho_t^*\vol_3] : t\in \bR\}$ is a linearly independent subset of $\Hbr 3 (\PSL_2\bR;\bR)$.  
\end{proof}

The cardinality of the set of functions from a continuum to itself is $2^{\#\bR}$, as is the set of bounded functions.  Thus $\dim_{\mathbb R} \Hb 3(\PSL_2 \mathbb K;\bR) \le 2^{\#\bR}$, where $\mathbb K = \bC$ or $\bR$.     The set of field mappings $\bR\to \bC$ or $\bC\to \bC$ has cardinality $2^{\#\bR}$.  
\begin{question}\label{q:G_dense}
Let $\mathbb K\in \{\bR, \bC\}$.  Is  $\dim_\bR \Hbr 3 (\PSL_2\mathbb K;\bR)=2^{\#\bR}$?  Assuming the axiom of choice, does the set of field embeddings $\bC\to \bC$ define an injective map with discrete and linearly independent image to $\Hb 3 (\PSL_2\mathbb K;\bR)$ by pulling back the standard volume class?  Without the axiom of choice, what is the dimension of $\dim_\bR \Hb 3 (\PSL_2\mathbb K;\bR)$?  Do pullbacks of volume classes span all of $\Hb 3(\PSL_2\mathbb K;\bR)$?
\end{question}

\subsection{On the collection of $\rho$-dense subgroups}\label{question}
Let $G$ be a discrete group and suppose $\rho : G\to \PSL_2\bR$ is dense.  Define the set $\mathcal {DS}(\rho)$ of \emph{$\rho$-dense subgroups} by \[\mathcal {DS}(\rho) = \{H\le_{f.g.} G: \overline {\rho(H)} = \PSL_2\bR\}, \] where $H\le _{f.g.}G$ means that $H$ is a finitely generated subgroup of $G$.   The following fact may seem surprising, at first.  An outline of the proof  was communicated to the author by Yair Minsky.  
\begin{fact} \label{density spectrum R} Let $\rho_1, \rho_2: F_2 \to \PSL_2\bR$ be dense.  If $\mathcal{DS}(\rho_1) = \mathcal{DS}(\rho_2)$, then $\rho_1$ is conjugate to $\rho_2$ in $\PSL_2\bR$.  
\end{fact}

The proof of Fact \ref{density spectrum R} uses the fact that dense representations into $\PSL_2\bR$ contain elliptic elements, since $(0, 2)\subset \bR$ is  open in the image of the absolute value of the trace function.  Generically, dense representations $F_2 \to \G$ do not contain any elliptics.  Due to the fractal nature of the boundary of Schottky space, one might expect an answer to the following to be more involved.
\begin{question}\label{density question}Let $\rho_1, \rho_2: F_2 \to \PSL_2\bC$ be dense.  If $\mathcal{DS}(\rho_1) = \mathcal{DS}(\rho_2)$, then is $\rho_1$ necessarily conjugate to $\rho_2$ in $\PSL_2\bC$?
\end{question}
If Question \ref{density question} has an affirmative answer, then the volume class of (the conjugacy class of) a dense representation is distinguished and separated in semi-norm from every other such class, by Theorem \ref{thm:intro_main} and Lemma \ref{lem:geom_elem}.  We note however, that Question \ref{density question} may have a negative answer, and all conjugacy classes of dense representations can still be separated in semi-norm, because there was a tremendous amount of freedom in our choice of discrete and faithful representation used to define chains on $F_2$ in Lemma \ref{manifold chains} and Proposition \ref{dense group chains}.  Otherwise, given $\rho_0: F_2\to \PSL_2\bC$, it would be interesting to understand the set of representations that are not conjugate to $\rho_0$ but that have the same volume class in bounded cohomology.  We expect that nothing interesting happens, however.
\begin{conjecture}
For a dense representation $\rho:F_2\to \G$, if $\rho': F_2\to \G$ is any other representation such that \[ \| [\rho^*\vol] - [{\rho'}^*\vol]\|<v_3,\]
then $\rho'$ is conjugate to $\rho$.  
\end{conjecture}

\section{Higher and lower dimensional volume classes}\label{sec:higher_dim}
For even $n\ge 4$, it is known that there is an $\epsilon_n>0$ such that the Cheeger constant of $\bH^n/\rho(H)$ is at least $\epsilon_n$ when $H$ is a free group of finite rank, a closed surface group of genus at least $2$, or a finite volume hyperbolic $3$-manifold group and $\rho$ is discrete and faithful \cite{Bowen:cheeger}.  Vanishing of the Cheeger constant is equivalent to the non-vanishing of the $n$-dimensional volume class of a discrete and faithful representation $\rho: H \to \Isom^+(\bH^n)$ \cite{KK:cheeger}, so we cannot hope to prove that a dense representation $\rho: F_2\to \Isom^+(\bH^n)$ has non-zero volume class by approximating chains built from a discrete and faithful representation $\rho_0: F_2\to \Isom^+(\bH^n)$.

In this section, we give a criterion to ensure that a dense representation $\rho: F_2\to \Isom^+(\bH^n)$ has non-vanishing $n$-dimensional volume class.  We stress, however, that we do not know if our criterion is ever satisfied for $n\ge 4$.   We recall a definition  from the introduction.

\begin{definition}\label{free approximation} Let $\Gamma$ be a discrete group, $\alpha\in \HH_n(\Gamma;\bR)$, and $K>0$.  We say that $\alpha$ is \emph{$K$-freely approximated} if there is an integer $m\ge 2$, a homomorphism $\varphi: F_m\to \Gamma$, and a chain $Z\in \textnormal C_n(F_m;\bR)$ such that $\varphi_*(Z)\in \alpha$ and $\|\partial Z\|_1\le K$.  
\end{definition}

The conclusions of the following proposition should now feel somewhat believable.  The proof more or less follows directly from the definitions and an application of the approximation scheme introduced in Section \ref{approximation section}.
\begin{proposition}\label{prop:higher_dim}
Let $n\ge 2$ and suppose $(M_i)$ is a sequence of closed and oriented hyperbolic $n$-manifolds with volume tending to infinity.  Let $[M_i]\in \HH_n(\pi_1(M_i);\bR)$ be the image of the fundamental class of $M_i$ under the natural isomorphism $\HH_n(M_i; \bR) \to \HH_n(\pi_1(M_i);\bR)$.  

If there is a $K>0$ such that $[M_i]$ is $K$-freely approximated for all $i$, then for any dense representation $\rho: F_2 \to \Isom^+(\bH^n)$,  \[[\rho^*\vol_n]\not=0 \in \Hb n(F_2; \bR).\]

\end{proposition}
\begin{proof}
By definition of $K$-free approximation, there are maps $\varphi_i : F_{m_i}\to \pi_1(M_i)$ and  chains $Z_i \in \textnormal C_n(F_{m_i};\bR)$ such that $\varphi_{i*}(Z_i)\in [M_i]$ and $\|\partial Z_i\|_1<K$.  Let $\rho_i: \pi_1(M_i) \to \Isom^+(\bH^n)$ be a hyperbolization of $M_i$;  consider $\rho_i\circ\varphi_i: F_{m_i}\to \Isom^+(\bH^n)$.  Since $\varphi_{i*}(Z_i)\in [M_i]$, for any $x\in \bH^n$, we have 
\[\langle [\varphi_{i*}(Z_i)], [\rho_i^* \vol_n^x]\rangle = \int _{M_i}d\vol =  \vol(M_i)\to \infty.\]
Indeed, 
\[\vol(M_i) = \langle \varphi_{i*}(Z_i), \rho_i^* \vol_n^x\rangle= \langle Z_i, (\rho_i\circ \varphi_i)^* \vol_n^x\rangle. \]

By Proposition \ref{approximation of chains}, there is a chain $Z_i(1)\in \textnormal C_n(F_2;\bR)$ such that 
\[|\rho^*\vol_n^x(Z_i(1))|>\vol(M_i) -1,\]
 $\|Z_i(1)\|_1\le \|Z_i\|_1$, and $\|\partial Z_i(1)\|_1\le \|\partial Z_i\|_1<K$.  If  $b \in \textnormal C^{n-1}(F_2;\bR)$ is such that $\delta b = \rho^*\vol_n^x$, then 
\[|\rho^*\vol^x_n(Z_i(1))|  = |b(\partial Z_i(1))|.\]
But then 
\[ \|b\|_\infty >\frac{\vol(M_i) - 1}{K}\] 
for every $i$. The right hand side goes to infinity with $i$, so $b$ is not bounded.  Thus $[\rho^*\vol^x_n] \not  = 0 \in \Hb n (F_2;\bR)$.   
\end{proof}

\begin{remark}\label{rem:dense_n=4}
Suppose that one can find a sequence of $n$-manifolds $(M_i)$ satisfying the hypotheses of Proposition \ref{prop:higher_dim}.   If one is able to construct dense representations $\{\rho_\theta: F_2 \to \Isom^+(\bH^n)\}_{\theta\in \Lambda}$ by analogy with those defined in Section \ref{constructions}, then for even integers $n\ge 4$, $\dim_\bR \Hb n (F_2;\bR) = \#\bR$.  Indeed, assume we have a finite collection $\{\rho_{\theta_1}, ..., \rho_{\theta_m}\}$ and subgroups $\iota_j: H_{\theta_j}\to  F_2$ such that $\rho_{\theta_k}(H_{\theta_j})$ is dense if $k=j$, and a Schottky group otherwise.  Then $[(\rho_{\theta_k}\circ \iota_j)^*\vol_n] = 0 \in \Hb n (H_{\theta_j};\bR)$ by the combination of \cite{KK:cheeger} and \cite{Bowen:cheeger}, as long as $k\not= j$.  We can then apply a slight modification of the proof of Theorem \ref{uncountable family}; we would only be able to prove linear independence of volume classes of these dense representations in $\Hb n(F_2;\bR)$ but not necessarily in the reduced space.  It seems that such  representations $\{\rho_\theta: F_2 \to \Isom^+(\bH^n)\}_{\theta\in \Lambda}$ should not be too difficult to construct by hand, as we have done for $n=3$.
\end{remark}

The above criterion is just an example of how to apply Proposition \ref{approximation of chains} to show that some bounded classes may be non-trivial, with the necessary auxiliary information.  There are notions of straight chains with bounded volume in complex hyperbolic space; the statement of the proposition can be modified appropriately.  Also, the manifolds $(M_i)$ may be chosen to be non-compact with finite volume; the relevant feature is that the fundamental group of a cross section of a cusp is amenable.  Many additional modifications can be made, and the reader is encouraged to make them.

We would like to convince the reader that free approximation of fundamental classes of manifolds is not an entirely contrived concept; although we admit that it may be a low dimensional phenomenon.  

\begin{lemma}\label{lem:surface_approx}	Let $X$ be a closed and oriented hyperbolic surface of genus $g\ge 2$.  Then $[X]$ is $2$-freely approximated.
\end{lemma}
\begin{proof}[Sketch of proof]
There is a nice $1$-vertex triangulation of a hyperbolic surface; see Fig. 1. 
We can find a base point $\bar x$ in $X$ and a standard generating set $\{a_1, b_1, ..., a_{g}, b_{g}\}$ for $\pi_1(X, \bar x)\le \PSL_2\bR$ such that the union of the geodesic representatives of the $a_i$'s and $b_i$'s based at $\bar x$ is embedded.  Cutting $X$ open along these arcs produces a convex identification $4g$-gon with geodesic sides that can be embedded in the hyperbolic plane with vertex set contained in the preimage of $\bar x$ under the covering projection.  By choosing a vertex $x\in \bH^2$ of this $4g$-gon, we can join every vertex not adjacent to $x$ with a geodesic segment.  This process triangulates the identification polygon and defines a chain $V\in \textnormal C_2(\pi_1(X);\bR)$ that represents $[X]$; we have $\|V\|_1 = 4g -2$.  Let $x_1,y_1, ..., x_{g}, y_{g}$ be a free basis for $F_{2g}$.  We define a homomorphsim $\varphi: F_{2g}\to \pi_1(X)$ given by $\varphi(x_j) = a_j$ and $\varphi(y_j ) = b_j$.  

\begin{figure}[ht]
    \centering
    \includegraphics{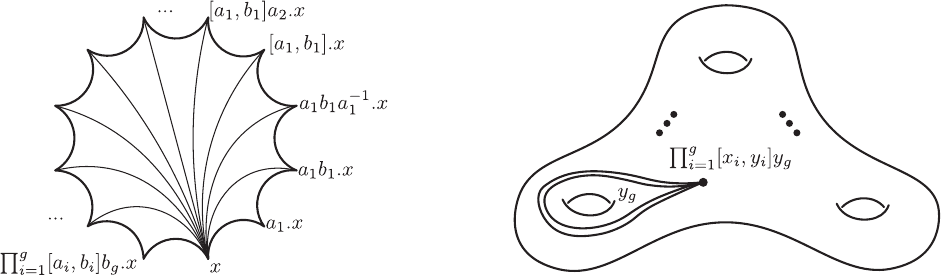}
    \label{fig:surface}
    \caption{Left: $X$ cut along a standard generating set and a nice $1$-vertex locally geodesic triangulation of $X$.  Right: Using $F_{2g}$, the polygon on the left does not quite close up to form a closed surface, because $y_g\not= \prod_{i =1}^g[x_i, y_i]y_g\in F_{2g}$.}
\end{figure}

There is a chain $Z\in \textnormal C_2(F_{2g};\bR)$ such that $\varphi_*(Z) = V$ and such that $\|Z\|_1 = \|V\|_1$ and $\|\partial Z\|_1=2$;  $Z$ is obtained by replacing words in $\{a_i, b_i\}$ with the corresponding words spelled with $\{x_i, y_i\}$.  Topologically, the quotient of $Z$ by the identifications induced by $F_{2g}$ is a $2$-dimensional CW-complex that is obtained by taking the metric completion of $X$ after removing the interior of the based geodesic $b_g$.  The boundary of this complex consists of two $1$-cells.  
\end{proof}
Take any sequence $(X_g)$ of closed hyperbolic surfaces with genus $g$ tending to $\infty$.  By Lemma \ref{lem:surface_approx}, $[X_g]$ is $2$-freely approximated for all $g$.  In the sketch of the proof of Lemma \ref{lem:surface_approx}, the chains $V_g$ are $\frac{\area (X_g)}{\|V_g\|_1}= \frac{4\pi(g-1)}{4g-2}$-efficient.    Combined with Lemma \ref{scheme}, we obtain the following.
\begin{corollary}\label{surface corollary} 
If $\rho:F_2\to \PSL_2\bR$ is a dense representation, then $\|[\rho^*\vol_2]\|_\infty = \pi$.
\end{corollary}
Corollary \ref{surface corollary} is obtained easily from the sophisticated theory of computing bounded cohomology via boundary maps; see, e.g. \cite[Section 4.3]{BIW:hermitian}.  However, we emphasize that our hands on approach leads to a completely elementary and geometric proof.

\begin{remark} For any closed hyperbolic $3$-manifold $M$, there is a sequence of covers $M_d\to M$ such that $[M_d]$ is $K$-freely approximated for all $d$; the constant $K$ is a function of the genus of a certain immersed $\pi_1$-injective surface in $M$, that of a virtual fiber.  The proof, which we omit, relies on the positive resolution of the Virtual Fibering Conjecture.
\end{remark}

\providecommand{\bysame}{\leavevmode\hbox to3em{\hrulefill}\thinspace}
\providecommand{\href}[2]{#2}

\Addresses

\end{document}